\documentclass[11pt]{amsart}
\usepackage{graphicx}
\usepackage[mathcal]{euscript}
\usepackage{float}
\usepackage{cite}
\usepackage{verbatim}
\usepackage[normalem]{ulem}
\usepackage{colortbl}
\usepackage{mathtools}
\allowdisplaybreaks

\restylefloat{figure}

\newtheorem{theorem}{Theorem}[section]
\newtheorem{lemma}[theorem]{Lemma}
\newtheorem{corollary}[theorem]{Corollary}
\newtheorem{proposition}[theorem]{Proposition}

\newtheorem{conjecture}[theorem]{Conjecture}
\newtheorem{question}[theorem]{Question}

\theoremstyle{definition}

\newtheorem{example}[theorem]{Example}

\theoremstyle{remark}

\numberwithin{equation}{section}

\def\Span{\operatorname{span}}
\def\alt{\operatorname{alt}} 
\def\dalt{\operatorname{dalt}}

\def\rank{\operatorname{rank}}

\def\width{\operatorname{width}}

\usepackage{tikz}

\usetikzlibrary{arrows,shapes,decorations.pathreplacing,backgrounds,patterns}

\pgfdeclarelayer{background}
\pgfdeclarelayer{background2}
\pgfdeclarelayer{background2a}
\pgfdeclarelayer{background2b}
\pgfdeclarelayer{background3}
\pgfdeclarelayer{background4}
\pgfdeclarelayer{background5}
\pgfdeclarelayer{background6}
\pgfdeclarelayer{background7}

\pgfsetlayers{background7,background6,background5,background4,background3,background2b,background2a,background2,background,main}

\definecolor{lsupurple}{RGB}{70,29,124}
\definecolor{lsugold}{RGB}{253,208, 35}
\definecolor{Red}{RGB}{204,37,41}
\definecolor{Blue}{RGB}{57,106,177}
\definecolor{Orange}{RGB}{218,124,48}
\definecolor{Green}{RGB}{62,150,81}
\definecolor{Purple}{RGB}{107,76,154}
\definecolor{Maroon}{RGB}{146,36,40}
\definecolor{Brown}{RGB}{139,69,19}

\begin{document}

\title{On the Turaev genus of torus knots}

\author{Kaitian Jin}
\address{Department of Mathematics and Statistics\\
Vassar College\\
Poughkeepsie, NY} 
\email{kajin@vassar.edu}

\author{Adam M. Lowrance}
\address{Department of Mathematics and Statistics\\
Vassar College\\
Poughkeepsie, NY} 
\email{adlowrance@vassar.edu}

\author{Eli Polston}
\address{Department of Mathematics and Statistics\\
Vassar College\\
Poughkeepsie, NY} 
\email{elpolston@vassar.edu}

\author{Yanjie Zheng}
\address{Department of Mathematics and Statistics\\
Vassar College\\
Poughkeepsie, NY} 
\email{yazheng@vassar.edu}

\thanks{This paper is the result of a summer research project in Vassar College's Undergraduate Research Science Institute (URSI). The second author is supported by Simons Collaboration Grant for Mathematicians no. 355087.}

\subjclass{}
\date{}

\begin{abstract}
The Turaev genus and dealternating number of a link are two invariants that measure how far away a link is from alternating. We determine the Turaev genus of a torus knot with five or fewer strands either exactly or up to an error of at most one. We also determine the dealternating number of a torus knot with five or fewer strand up to an error of at most two. Additional bounds are given on the Turaev genus and dealternating number of torus links with five or fewer strands and on some infinite families of torus links on six strands.  
\end{abstract}

\maketitle

\section{Introduction}
\label{section:Intro}

The Turaev surface of a link diagram was first constructed by Turaev \cite{Turaev:Jones} to give an alternate method of proving Kauffman \cite{Kauffman:StateModels}, Murasugi \cite{Murasugi:Jones}, and Thistlethwaite's \cite{Thistlethwaite:Jones} theorem that the Jones polynomial gives a lower bound on the crossing number of a link. Specifically, if $L$ is a link with diagram $D$, crossing number $c(L)$, Turaev surface of genus $g_T(D)$, and Jones polynomial $V_L(t)$, then Turaev proved that
$$\Span V_L(t) + g_T(D) \leq c(L).$$

Dasbach, Futer, Kalfagianni, Lin, and Stoltzfus \cite{DFKLS:Jones} defined the Turaev genus $g_T(L)$ of the link $L$ to be the minimum genus of the Turaev surface of any link diagram of $D$. Turaev showed that the genus $g_T(D)$ of the Turaev surface is zero if and only if $D$ is a connected sum of alternating link diagrams, and consequently the Turaev genus of a link is zero if and only if the link is alternating. One can view Turaev genus as a filtration on links where links with large Turaev genus can be interpreted as being far away from alternating.

In this article, we compute the Turaev genus of several infinite families of torus knots on six or fewer strands. For positive, coprime integers $p$ and $q$, let $T_{p,q}$ be the $(p,q)$-torus knot. Since the Turaev genus of a link and its mirror are equal, it is enough to only consider positive torus knots. Because $T_{2,q}$ is alternating for all $q$, it follows that $g_T(T_{2,q})=0$. Abe and Kishimoto \cite{AbeKishimoto:3-braid} and Lowrance \cite{Lowrance:Twisted} proved that the Turaev genus of torus knots on three strands is given by
$$g_T(T_{3,3n+i})=n$$
for each $n\geq0$ and for $i=0, 1,$ or $2$. The following theorem gives the Turaev genus of all torus knots on four or five strands up to an error of at most one and gives the Turaev genus of an infinite family of six stranded torus knots.
\begin{theorem}
\label{theorem:TuraevExact}
For each non-negative integer $n$ and for $j=2$, $3,$ and $4$, 
$$
\begin{array}{>{\hfil$}p{5 cm}<{$\hfil} >{\hfil$}p{5 cm}<{$\hfil}}
g_T(T_{4,4n+1}) =  2n, &  g_T(T_{4,4n+3})= 2n+1,\\
g_T(T_{5,5n+1}) =  4n, & g_T(T_{6,6n+1}) =  6n,~\text{and}\\
\multicolumn{2}{c}{4n+j-2\leq g_T(T_{5,5n+j}) \leq 4n + j -1.}
\end{array}$$
\end{theorem}
Theorem \ref{theorem:TuraevLink} gives bounds on the Turaev genus of some infinite families of torus links on six or fewer strands.

Another way to measure how far a link $L$ is from alternating is via its dealternating number $\dalt(L)$ \cite{Adams:Almost}. The dealternating number $\dalt(D)$ of a link diagram $D$ is the minimum number of crossing changes needed to transform $D$ into an alternating diagram. The {\em dealternating number} of a link $L$ is the minimum dealternating number of any diagram $D$ of $L$. Abe and Kishimoto \cite{AbeKishimoto:3-braid} show that $g_T(L)\leq \dalt(L)$ for any link $L$.  The following theorem gives bounds on the dealternating numbers of some families of torus knots on six or fewer strands.
\begin{theorem}
\label{theorem:DaltBest}
For each non-negative integer $n$ and for $j=2$, $3,$ and $4$,
$$
\begin{array}{>{\hfil$}p{7 cm}<{$\hfil} >{\hfil$}p{7 cm}<{$\hfil}}
2n \leq \dalt(T_{4,4n+1})\leq 2n+1, &
2n+1 \leq \dalt(T_{4,4n+3}) \leq  2n+2,\\
4n\leq \dalt(T_{5,5n+1}) \leq 4n+1, &
4n + j -2 \leq \dalt(T_{5,5n+j}) \leq 4n+ j,~\text{and}\\
\multicolumn{2}{c}{6n \leq \dalt(T_{6,6n+1}) \leq  6n+2.}
\end{array}$$
\end{theorem}

We prove Theorems \ref{theorem:TuraevExact} and \ref{theorem:DaltBest} in two steps: first we compute a lower bound coming from knot Floer homology, and second we find diagrams with the indicated Turaev genus or dealternating number. The knot Floer homology $\widehat{HFK}(K)$ of a knot $K$ is a categorification of the Alexander polynomial of $K$ developed by Ozsv\'ath and Szab\'o \cite{OS:HFK} and independently by Rasmussen \cite{Rasmussen:HFK}. The knot Floer homology $\widehat{HFK}(K)$ of $K$ is a bigraded $\mathbb{Z}$-module with Alexander grading $s$ and Maslov grading $m$. It decomposes into direct summands $\widehat{HFK}(K) = \bigoplus_{s,m\in\mathbb{Z}} \widehat{HFK}_m(K,s)$, and the symmetrized Alexander polynomial $\Delta_K(t)$ of $K$ can be recovered as a filtered Euler characteristic of knot Floer homology:
$$\Delta_K(t) = \sum_{s,m\in\mathbb{Z}} (-1)^m \rank \widehat{HFK}_m(K,s) \cdot t^s.$$ Define
\begin{align*}
\delta_{\max}(K)= & \; \max\{s - m ~|~ \widehat{HFK}_m(K,s) \neq 0\}~and\\
\delta_{\min}(K) = & \; \min\{s - m ~|~ \widehat{HFK}_m(K,s) \neq 0\}.
\end{align*}
The {\em width} of the knot Floer homology is defined as 
$$\operatorname{width} \widehat{HFK}(K) = \delta_{\max}(K) - \delta_{\min}(K) + 1.$$

Lowrance \cite{Lowrance:HFK} proved that the width of the knot Floer homology of a knot gives a lower bound on its Turaev genus. Ozsv\'ath and Szab\'o \cite{OS:Lens} gave an algorithm to compute the knot Floer homology of a torus knot (or any knot with a lens or $L$-space surgery) from its Alexander polynomial.  We use the Ozsv\'ath and Szab\'o algorithm to compute the width of the knot Floer homology $\widehat{HFK}(T_{p,q})$ for many torus knots $T_{p,q}$. 
\begin{theorem}
\label{theorem:HFKWidth}
For each positive integer $p$ and non-negative integer $n$,
\begin{align*}
\width \widehat{HFK}(T_{p,pn+1}) = & \, n\left\lfloor \frac{(p-1)^2}{4}\right\rfloor + 1~\text{and}\\
~\width\widehat{HFK}(T_{p,pn-1}) = & \, n\left\lfloor \frac{(p-1)^2}{4}\right\rfloor - \left\lfloor \frac{p-1}{2} \right\rfloor + 1.
\end{align*}
Moreover for each non-negative integer $n$, 
$$\width \widehat{HFK}(T_{5,5n+2}) =  4n+1~\text{and}~
\width \widehat{HFK}(T_{5,5n+3}) =  4n+2.$$
\end{theorem}
We conjecture that there is a recursive formula to compute $\width\widehat{HFK}(T_{p,q})$. The base case of the recursion is $\width \widehat{HFK}(T_{1,q})=1$ for all $q$. Also, since $T_{p,q}=T_{q,p}$, it follows that $\operatorname{width} \widehat{HFK}(T_{p,q}) = \operatorname{width} \widehat{HFK}(T_{q,p})$. These two rules together with Equation \ref{equation:WidthConjecture} would give a way to evaluate $\operatorname{width} \widehat{HFK}(T_{p,q})$ for any  pair of positive coprime integers $(p,q)$. Theorem \ref{theorem:HFKWidth} implies the conjecture is true for all pairs of positive coprime integers $(p,q)$ where $q\equiv \pm 1 \mod p$ or $p\leq 6$. A computer computation shows that the conjecture is also true when $p$ and $q$ are less than $250$.
\begin{conjecture}
\label{conjecture:Recursive}
Let $p$ and $q$ be positive coprime integers with $p<q$. Then
\begin{equation}
\label{equation:WidthConjecture}
\operatorname{width}\widehat{HFK}(T_{p,q}) - \operatorname{width}\widehat{HFK}(T_{p,q-p}) = \left\lfloor \frac{(p-1)^2}{4} \right\rfloor.
\end{equation}
\end{conjecture}
Assuming Conjecture \ref{conjecture:Recursive}, the best-case scenario for the Turaev genus of a torus knot is that it satisfies an analogous recursive relation. In this line of thinking, we ask the following question.
\begin{question}
Is it true that
$$g_T(T_{p,q}) - g_T(T_{p,q-p}) = \left\lfloor \frac{(p-1)^2}{4} \right\rfloor$$
for all positive coprime integers $p$ and $q$ with $p<q$?
\end{question}

This paper is organized as follows. In Section \ref{section:background}, we recall some basic facts about the Turaev genus and dealternating number of a link. In Section \ref{section:lowerbound}, we discuss lower bounds on Turaev genus and dealternating number, and we prove Theorem \ref{theorem:HFKWidth}. Finally, in Section \ref{section:diagrams}, we prove Theorems \ref{theorem:TuraevExact} and \ref{theorem:DaltBest}.\medskip

\noindent{\bf Acknowledgement:} The authors are thankful for helpful conversations with Moshe Cohen, Oliver Dasbach, and John McCleary.

\section{Turaev genus and dealternating number}
\label{section:background}

In this section, we review some properties of the Turaev surface, the Turaev genus of a link, and the dealternating numbers of a link. Champanerkar and Kofman \cite{CK:Survey} wrote an excellent recent survey of Turaev genus.

Each crossing in a link diagram $D$ has an $A$-resolution and a $B$-resolution, as in Figure \ref{figure:Resolution}. The collection of curves in the plane obtained by choosing either an $A$-resolution or a $B$-resolution for each crossing in $D$ is called a {\em Kauffman state} of $D$. The {\em all-$A$ state} is the Kauffman state obtained by choosing an $A$-resolution for each crossing. Similarly, the {\em all-$B$ state} is the Kauffman state obtained by choosing a $B$-resolution for each crossing.

\begin{figure}[h]
$$\begin{tikzpicture}[>=stealth, scale=.8]
\draw (-1,-1) -- (1,1);
\draw (-1,1) -- (-.25,.25);
\draw (.25,-.25) -- (1,-1);
\draw (-3,0) node[above]{\Large{$A$}};
\draw[->,very thick] (-2,0) -- (-4,0);
\draw (3,0) node[above]{\Large{$B$}};
\draw[->,very thick] (2,0) -- (4,0);
\draw (-5,1) arc (120:240:1.1547cm);
\draw (-7,-1) arc (-60:60:1.1547cm);
\draw (5,1) arc (210:330:1.1547cm);
\draw (7,-1) arc (30:150:1.1547cm);
\end{tikzpicture}$$
\caption{The $A$-resolution and $B$-resolution of a crossing.}
\label{figure:Resolution}
\end{figure}
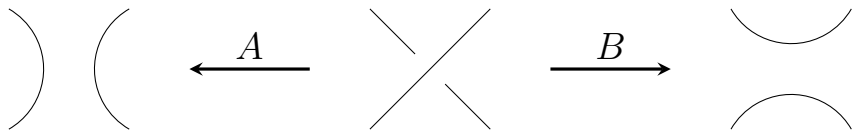

Consider the link diagram $D$ as embedded on the projection sphere $S^2$. Embed the all-$B$ state just inside the sphere and the all-$A$ state just outside the sphere. Away from neighborhoods of crossings, connect the all-$B$ and all-$A$ state with bands. In a neighborhood of a crossing, connect the all-$B$ state with the all-$A$ state with a saddle, as in Figure \ref{figure:Saddle}. The resulting surface is a cobordism between the all-$B$ and all-$A$ states of $D$. The {\em Turaev surface} $F(D)$ of $D$ is obtained by capping off each of the boundary components with disks. 
\begin{figure}[h]
$$\begin{tikzpicture}
\begin{scope}[thick]
\draw [rounded corners = 10mm] (0,0) -- (3,1.5) -- (6,0);
\draw (0,0) -- (0,1);
\draw (6,0) -- (6,1);
\draw [rounded corners = 5mm] (0,1) -- (2.5, 2.25) -- (0.5, 3.25);
\draw [rounded corners = 5mm] (6,1) -- (3.5, 2.25) -- (5.5,3.25);
\draw [rounded corners = 5mm] (0,.5) -- (3,2) -- (6,.5);
\draw [rounded corners = 7mm] (2.23, 2.3) -- (3,1.6) -- (3.77,2.3);
\draw (0.5,3.25) -- (0.5, 2.25);
\draw (5.5,3.25) -- (5.5, 2.25);
\end{scope}

\begin{pgfonlayer}{background2}
\fill [lsugold]  [rounded corners = 10 mm] (0,0) -- (3,1.5) -- (6,0) -- (6,1) -- (3,2) -- (0,1); 
\fill [lsugold] (6,0) -- (6,1) -- (3.9,2.05) -- (4,1);
\fill [lsugold] (0,0) -- (0,1) -- (2.1,2.05) -- (2,1);
\fill [lsugold] (2.23,2.28) --(3.77,2.28) -- (3.77,1.5) -- (2.23,1.5);

\fill [white, rounded corners = 7mm] (2.23,2.3) -- (3,1.6) -- (3.77,2.3);
\fill [lsugold] (2,2) -- (2.3,2.21) -- (2.2, 1.5) -- (2,1.5);
\fill [lsugold] (4,2) -- (3.7, 2.21) -- (3.8,1.5) -- (4,1.5);
\end{pgfonlayer}

\begin{pgfonlayer}{background4}
\fill [lsupurple] (.5,3.25) -- (.5,2.25) -- (3,1.25) -- (2.4,2.2);
\fill [rounded corners = 5mm, lsupurple] (0.5,3.25) -- (2.5,2.25) -- (2,2);
\fill [lsupurple] (5.5,3.25) -- (5.5,2.25) -- (3,1.25) -- (3.6,2.2);
\fill [rounded corners = 5mm, lsupurple] (5.5, 3.25) -- (3.5,2.25) -- (4,2);
\end{pgfonlayer}

\draw [thick] (0.5,2.25) -- (1.6,1.81);
\draw [thick] (5.5,2.25) -- (4.4,1.81);
\draw [thick] (0.5,2.75) -- (2.1,2.08);
\draw [thick] (5.5,2.75) -- (3.9,2.08);

\begin{pgfonlayer}{background}
\draw [black!50!white, rounded corners = 8mm, thick] (0.5, 2.25) -- (3,1.25) -- (5.5,2.25);
\draw [black!50!white, rounded corners = 7mm, thick] (2.13,2.07) -- (3,1.7)  -- (3.87,2.07);
\end{pgfonlayer}
\draw [thick, dashed, rounded corners = 2mm] (3,1.85) -- (2.8,1.6) -- (2.8,1.24);
\draw (0,0.5) node[left]{$D$};
\draw (1.5,3) node{$A$};
\draw (4.5,3) node{$A$};
\draw (3.8,.8) node{$B$};
\draw (5.3, 1.85) node{$B$};
\end{tikzpicture}$$
\caption{In a neighborhood of each crossing of $D$, the all-$B$ state is connected to the all-$A$ state with a saddle.}
\label{figure:Saddle}
\end{figure}
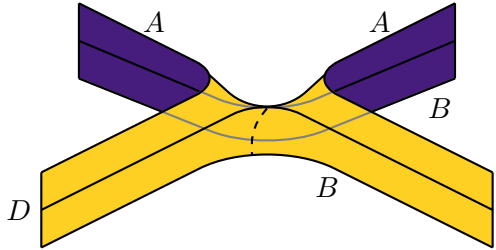 

The Turaev surface of $D$ is a Heegaard surface in $S^3$. The diagram $D$ is alternating on the surface of $F(D)$ (when viewed from one side of the surface). Let $c(D)$ be the number of crossings in $D$, and let $s_A(D)$ and $s_B(D)$ be the number of components in the all-$A$ and all-$B$ states of $D$ respectively. A link diagram can be thought of as a graph embedded in the plane where the crossings of $D$ are the vertices of the graph and the arcs of $D$ between the crossings are the edges of the graph. We say that $D$ is connected if it is connected when thought of as a graph. We denote the genus of the Turaev surface of $D$ by $g_T(D)$. If $D$ is connected, then 
\begin{equation}
\label{equation:g_T(D)}
g_T(D)= \frac{1}{2}\left( 2 + c(D) - s_A(D) - s_B(D) \right).
\end{equation}
Figure \ref{figure:Turaev45} depicts the Turaev surface of the $(4,5)$-torus knot.

\begin{figure}[h]
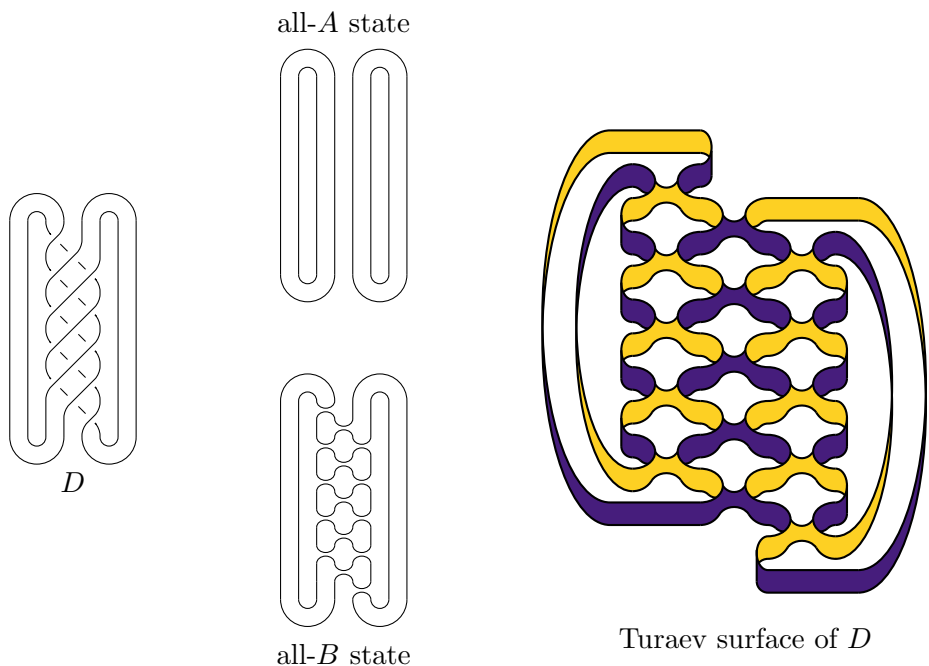

$$
$$

\caption{A diagram $D$ of the $(4,5)$ torus knot, its all-$A$ and all-$B$ states, and its Turaev surface before capping off with disks.}
\label{figure:Turaev45}
\end{figure}

The Turaev surface has genus zero if and only if $D$ is a connected sum of alternating diagrams. Armond and Lowrance \cite{ArmLow:Turaev} and independently Kim \cite{Kim:TuraevClassification} characterized link diagrams whose Turaev surface is genus one or two. Dasbach et. al. \cite{DFKLS:Jones} proved that the Jones polynomial is an evaluation of the Bollob\'as-Riordan polynomial of a graph cellularly embedded in the Turaev surface. Dasbach and Lowrance \cite{DasLow:TuraevKh} used the same graphs embedded in the Turaev surface to give a model for Khovanov homology.

The {\em Turaev genus} $g_T(L)$ of a link $L$ is defined as
$$g_T(L) = \min \{ g_T(D) ~|~ D~\text{is a diagram of}~L\}.$$
The Turaev genus of a link $L$ is zero if and only if $L$ is alternating. To compute the Turaev genus of a non-alternating link, one typically finds a diagram that is believed to be Turaev genus minimizing and then uses a computable obstruction or lower bound to prove that the diagram is in fact of minimal Turaev genus. Obstructions and lower bounds come from the Jones polynomial \cite{DasLow:TuraevJones}, Khovanov homology \cite{CKS:Khovanov}, knot Floer homology \cite{Lowrance:HFK}, or from comparing certain concordance invariants \cite{DasLow:Concordance}. The lower bound coming from knot Floer homology will be particularly important for us, and we will discuss it further in Section \ref{section:lowerbound}.

It has proven somewhat difficult to compute the Turaev genus of many infinite families of knots or links. Since a non-alternating link $L$ with a genus one Turaev surface has $g_T(L)=1$, it follows that non-alternating pretzel and Montesinos links are Turaev genus one. Abe \cite{Abe:Adequate} proved that any adequate link diagram is Turaev genus minimizing. Abe and Kishimoto \cite{AbeKishimoto:3-braid} computed the Turaev genus of three stranded torus knots, and Lowrance \cite{Lowrance:Twisted} computed the Turaev genus of many closed $3$-braids.

The {\em dealternating number} $\dalt(D)$ of a link diagram $D$ is the minimum number of crossing changes needed to transform $D$ into an alternating diagram. The {\em dealternating number} $\dalt(L)$ of the link $L$ is defined by
$$\dalt(L) = \min\{\dalt(D)~|~ D~\text{is a diagram of}~L\}.$$
A link whose dealternating number is one is called {\em almost alternating}. Adams et. al. \cite{Adams:Almost} defined the dealternating number of a link and studied almost alternating links. In particular, they proved that a prime almost alternating knot is either torus or hyperbolic. Abe \cite{Abe:Alternation} proved that the only almost alternating torus knots are $T_{3,4}$, $T_{3,5}$, and their mirrors.

The dealternating number of a link has many of the same obstructions and lower bounds as Turaev genus, coming from the Jones polynomial \cite{DasLow:TuraevJones}, Khovanov homology \cite{CK:SpanningTree}, knot Floer homology \cite{OS:SpanningTree}, and from comparing concordance invariants \cite{Abe:Alternation}. Abe and Kishimoto \cite{AbeKishimoto:3-braid} show that $g_T(L) \leq \dalt(L)$ for any link $L$. It is unknown whether there exists a link $L$ such that $g_T(L) < \dalt(L)$. Non-alternating pretzel links and Montesinos links \cite{KimLee:Pretzel, AbeKishimoto:3-braid} have dealternating number one. Moreover, the Turaev genus and dealternating number of three-stranded torus knots (and many closed $3$-braids) agree \cite{AbeKishimoto:3-braid}. 

The {\em alternation number} $\alt(D)$ of a link diagram $D$ is the fewest number of crossing changes necessary to transform $D$ into a (possibly non-alternating) diagram of an alternating link. The {\em alternation number} $\alt(L)$ of a link $L$ is defined by
$$\alt(L) = \min\{\alt(D)~|~ D~\text{is a diagram of}~L\}.$$
Equivalently, one can think of the alternation number of a link as the Gordian distance between the link and the set of alternating links. The alternation number of a link was defined by Kawauchi \cite{Kawauchi:Alternation}. An immediate consequence of the definition is that $\alt(L)\leq \dalt(L)$ for any link $L$. Feller, Pohlmann, and Zentner \cite{FPZ:Alternation} computed the alternation number of torus knots on four or fewer strands, and Baader, Feller, Lewark, and Zentner \cite{BFLZ:KhWidth} gave bounds on the alternation and dealternating number of some families of torus links on six or fewer strands. Our arguments in Section \ref{section:diagrams} resemble those in \cite{FPZ:Alternation}. See \cite{Lowrance:AltDist} for more comparisons between Turaev genus, dealternating number, alternation number, and other related invariants.

\section{Knot Floer width of $T_{p,q}$}
\label{section:lowerbound}

In this section, we compute the knot Floer width for the torus knot $T_{p,q}$ where $p=5$ and $q$ is arbitrary or where $p$ is arbitrary and $q= pn\pm 1$ for some positive integer $n$. The knot Floer width computation gives a lower bound on both Turaev genus and dealternating number.

Recall that for any knot $K$, 
\begin{align*}
\delta_{\max}(K)= & \; \max\{s - m ~|~ \widehat{HFK}_m(K,s) \neq 0\}~and\\
\delta_{\min}(K) = & \; \min\{s - m ~|~ \widehat{HFK}_m(K,s) \neq 0\}.
\end{align*}
Lowrance \cite{Lowrance:HFK} proved the following theorem.
\begin{theorem} 
\label{theorem:LowerBound}
Let $K$ be a knot of Turaev genus $g_T(K)$. Then
$$\delta_{\max}(K) - \delta_{\min}(K) \leq g_T(K).$$
\end{theorem}

Let $K$ be a knot such that there is an integral surgery on $K$ that yields a lens space. Ozsv\'ath and Szab\'o \cite{OS:Lens} prove that the nonzero coefficients of the Alexander polynomial are all $\pm 1$ and that the knot Floer homology of $K$ can be determined from the Alexander polynomial of $K$, as follows.
\begin{theorem}
Let $K$ be a knot in $S^3$ such that there is an integral surgery on $K$ yielding a lens space. Then there exists a sequence of integers $s_{-k}< \cdots < s_k$ satisfying $s_{\ell} = -s_{-\ell}$ such that the Alexander polynomial of $K$ can be expressed as
$$\Delta_K(t) = (-1)^k + \sum_{\ell=1}^k (-1)^{k-\ell}(t^{s_\ell} + t^{s_{-\ell}}).$$
For $-k\leq \ell \leq k$ define
$$m_\ell = \begin{cases}
0 & \text{if}~\ell=k,\\
m_{\ell+1} -2 (s_{\ell+1} - s_\ell)+1 & \text{if}~k-\ell~\text{is odd},\\
m_{\ell+1}-1&\text{if}~k-\ell > 0~\text{is even.}
\end{cases}$$
Then $\widehat{HFK}_{m_\ell}(K,s_\ell) \cong \mathbb{Z}$ for each $\ell$, and $\widehat{HFK}_m(K,s) = 0$ otherwise.
\end{theorem}

In order to compute the lower bound for $g_T(K)$, we only need the quantities $\delta_\ell = s_\ell - m_\ell$ rather than the pairs $(m_\ell,s_\ell)$. Moreover, since $\widehat{HFK}_m(K,s) \cong \widehat{HFK}_{m-2s}(K,-s)$, it follows that we only need to consider $\delta_\ell$ for $\ell=0,1,\dots, k$. These observations lead to the following corollary.
\begin{corollary}
\label{corollary:HFKWidth}
Let $K$ be a knot in $S^3$ such that there is an integral surgery on $K$ yielding a lens space. Suppose that the Alexander polynomial of $K$ is given by 
$$\Delta_K(t) = (-1)^k + \sum_{\ell=1}^k (-1)^{k-\ell}(t^{s_\ell} + t^{s_{-\ell}}).$$
For $\ell=0,1,\dots, k$, define
$$\delta_\ell = \begin{cases}
s_k & \text{if}~\ell=k,\\
\delta_{\ell+1} + (s_{\ell+1} - s_{\ell}) -1& \text{if}~k-\ell~\text{is odd},\\
\delta_{\ell+1} - (s_{\ell+1} - s_{\ell}) +1 &\text{if}~k-\ell > 0~\text{is even.}
\end{cases}$$
Then 
$$\delta_{\max}(K) = \max\{\delta_\ell~|~0\leq \ell \leq k\}~\text{and}~\delta_{\min}(K) = \min\{\delta_\ell~|~0\leq \ell \leq k\}.$$
\end{corollary}

Since $pq\pm 1$ surgery on the torus knot $T_{p,q}$ is a lens space \cite{Moser:Surgery}, Corollary \ref{corollary:HFKWidth} and Theorem \ref{theorem:LowerBound} can be used to give a lower bound on the Turaev genus of $T_{p,q}$. The symmetrized Alexander polynomial of the $(p,q)$ torus knot is
\begin{equation}
\label{equation:TorusAlex}
\Delta_{T_{p,q}}(t) =t^{-(p-1)(q-1)/2} \frac{(t^{pq}-1)(t-1)}{(t^p-1)(t^q-1)}.
\end{equation}

\begin{example} The Alexander polynomial of $T_{4,5}$ is 
$$\Delta_{T_{4,5}}(t)=t^{-6}-t^{-5}+t^{-2}-1+t^2-t^5+t^6,$$
and its knot Floer homology is given in Table \ref{table:HFK}. We have $\delta_{\max}(T_{4,5})=6$ and $\delta_{\min}(T_{4,5})=4$, and thus the Turaev genus of $T_{4,5}$ must be at least two.
\begin{table}[h]
\begin{tabular}{ | r || c | c | c | c | c | c | c | c | c | c | c | c | c |}
\multicolumn{14}{c}{ $\widehat{HFK}(T_{4,5})$}\\
\hline
$s \backslash m$ & -12 & -11 & -10 & -9 & - 8 & -7 & - 6 & -5  & -4 & -3 & -2 & -1 & 0\\
\hline \hline
6 & & & & & & & & & & & & & \cellcolor{lsugold}$\mathbb{Z}$\\
\hline
5 & & & & & & & & & & & & \cellcolor{lsugold}$\mathbb{Z}$ & \cellcolor{lsugold}\\
\hline
4 & & & & & & & & & & & \cellcolor{lsugold} & \cellcolor{lsugold} & \cellcolor{lsugold}\\
\hline
3 & & & & & & & & & &\cellcolor{lsugold} &\cellcolor{lsugold} & \cellcolor{lsugold} & \\
\hline
2 & & & & & & & & & \cellcolor{lsugold} & \cellcolor{lsugold} &\cellcolor{lsugold}$\mathbb{Z}$ &  & \\
\hline
1 & & & & & & & & \cellcolor{lsugold} & \cellcolor{lsugold} & \cellcolor{lsugold} & & & \\
\hline
0 & & & & & & & \cellcolor{lsugold} & \cellcolor{lsugold}$\mathbb{Z}$ & \cellcolor{lsugold} & & & &\\
\hline
-1 & & & & & &\cellcolor{lsugold} &\cellcolor{lsugold} & \cellcolor{lsugold}& & & &  &\\
\hline
-2 & & & & & \cellcolor{lsugold} &\cellcolor{lsugold} & \cellcolor{lsugold}$\mathbb{Z}$& & & & &  & \\
\hline
-3 & & & & \cellcolor{lsugold} & \cellcolor{lsugold} & \cellcolor{lsugold} & & & & & & & \\
\hline
-4 & & & \cellcolor{lsugold} & \cellcolor{lsugold} &\cellcolor{lsugold} & & & & & & & &\\
\hline
-5 & &\cellcolor{lsugold}$\mathbb{Z}$ &\cellcolor{lsugold} & \cellcolor{lsugold} & & & & & & & & & \\
\hline
-6 & \cellcolor{lsugold}$\mathbb{Z}$ & \cellcolor{lsugold} & \cellcolor{lsugold} & & & & & & & & & & \\
\hline
\end{tabular}
\vspace{5pt}
\caption{The knot Floer homology of $T_{4,5}$.}
\label{table:HFK}
\end{table}
\end{example}

In order to apply Corollary \ref{corollary:HFKWidth} to a knot $K$, one must express the Alexander polynomial of $K$ as a Laurent polynomial. However, Equation \ref{equation:TorusAlex} expresses the Alexander polynomial of $T_{p,q}$ as a rational function. The following proposition gives Laurent polynomial formulas for $\Delta_{T_{p,q}}(t)$ for certain values of $p$ and $q$.

\begin{proposition}
\label{proposition:Alex}
If $p=2k+1$ is odd, then
{
\setlength{\belowdisplayskip}{0pt}%
\setlength{\abovedisplayskip}{0pt}%
\begin{align*}
\Delta_{T_{p,pn+1}}(t)= & \; 1 + \sum_{\epsilon\in\{\pm1\}} \sum_{i=1}^k \sum_{j=0}^{n-1} t^{\epsilon(p[(k-i+1)n-j]-i)}(t^{\epsilon i}-1),~\text{and}\\
\Delta_{T_{p,pn-1}}(t) = & \; 1+ \sum_{\epsilon\in\{\pm 1\}} \sum_{i=1}^k \sum_{j=0}^{n-1} t^{\epsilon p[(k-i+1)n-j-1]}(t^{\epsilon i}-1).\\
\intertext{If $p=2k$ is even, then}
\Delta_{T_{p,pn+1}}(t) = & \; (-1)^{n+1}  \sum_{\epsilon\in\{\pm 1\}} \sum_{i=1}^{k-1} \sum_{j=0}^{n-1} t^{\epsilon(p[(k-i+1)n-j]-kn-i)} (t^{\epsilon i}-1)\\
& \; + \sum_{\epsilon\in\{\pm 1\}} \sum_{i=1}^{n-1} (-1)^{n-i-1} t^{\epsilon i k}, \text{and}\\
\Delta_{T_{p,pn-1}}(t) = & \; (-1)^{n+1}  \sum_{\epsilon\in\{\pm 1\}} \sum_{i=1}^{k-1} \sum_{j=0}^{n-1} t^{\epsilon(p[(k-i)n-j-1]+kn)}(t^{\epsilon i}-1)\\
& \; + \sum_{\epsilon\in\{\pm 1\}} \sum_{i=1}^{n-1} (-1)^{n-i-1}  t^{\epsilon i k}.\\
\intertext{For the remaining torus knots on five strands, we have}
\Delta_{T_{5,5n+2}}(t) = &\; 1 +\sum_{\epsilon\in\{\pm 1\}} \sum_{j=0}^n t^{\epsilon(10n-5j+1)}(t^\epsilon -1)\\
& \; + \sum_{\epsilon\in\{\pm 1\}}\sum_{j=0}^{n-1} t^{\epsilon(5n-5j-4)}(t^{4\epsilon} - t^{3\epsilon} + t^{\epsilon} - 1),~\text{and}\\
\Delta_{T_{5,5n+3}}(t) = &\; 1 + \sum_{\epsilon\in\{\pm 1\}} \sum_{j=0}^{n-1} t^{\epsilon(10n-5j+3)}(t^\epsilon -1) \\
& \; +  \sum_{\epsilon\in\{\pm 1\}} \sum_{j=0}^{n} t^{\epsilon(5n-5j)} (t^{4\epsilon} - t^{3\epsilon} + t^{\epsilon} - 1).\\
\end{align*}
}
\end{proposition}

\begin{proof}
We only prove the result for $T_{p,pn+1}$ where $p=2k+1$ is odd. The other results follow from an analogous strategy. Since the formula for the symmetrized Alexander polynomial of $T_{p,q}$ is 
$$\Delta_{T_{p,q}}(t) = t^{-(p-1)(q-1)/2}\frac{(t^{pq}-1)(t-1)}{(t^p-1)(t^q-1)},$$
our approach is to show that $ t^{-(p-1)(q-1)/2}(t^{pq}-1)(t-1)$ is the product of $(t^p-1)(t^q-1)$ and our formula. 

Let $p=2k+1$ and $q=pn+1$. Then
\begin{align}
& \;(t^q-1)(t^p-1)  \sum_{i=1}^k \sum_{j=0}^{n-1} t^{p[(k-i+1)n-j]-i}(t^ {i}-1) \\
= &\; t^{pkn+pn} \sum_{i=1}^k (t^q-1) t^{-ipn-i} \sum_{j=0}^{n-1} (t^p-1)(t^{i-jp} - t^{-jp})\\
= & \; t^{pkn+pn} \sum_{i=1}^k (t^q-1) t^{-ipn-i} \sum_{j=0}^{n-1} (t^{i+p-jp} - t^{p-jp} - t^{i-jp} + t^{-jp})\\
= &\; t^{pkn+pn}  \sum_{i=1}^k (t^q-1) t^{-ipn-i}(t^q-1)(t^{p+i} - t^p - t^{p-n+i} + t^{p-n})\\
= &\; t^{pkn+pn} \sum_{i=1}^k(t^{p+q-ipn} -t^{p+q}-t^{p+1+i} +t^{p+1} - t^{p+i} + t^{p-pn-ipn}-t^{p-pn-iq})\\
= & \; t^{pkn+p+q}-t^{pkn+p+q-1} -t^{p+q} + t^{pn+p-k} + t^p - t^{p-k}.
\end{align}
In the above equation, (3.5) and (3.7) follow from their respective previous steps because the sums are telescoping.
One can similarly prove that
\begin{align*}
&\; (t^p-1)(t^q-1)\sum_{i=1}^k \sum_{j=0}^{n-1} t^{-p[(k-i+1)n-j]+i}(t^{-i}-1)\\
= & \; -t^{pn+p-k} +t^q + t^{p-k} -1 - t^{1-pkn}+t^{-pkn}.
\end{align*}
Therefore
\begin{align*}
&\; (t^p-1)(t^q-1)\left(1 + \sum_{\epsilon\in\{\pm1\}} \sum_{i=1}^k \sum_{j=0}^{n-1} t^{\epsilon(p[(k-i+1)n-j]-i)}(t^{\epsilon i}-1)\right)\\
=&\; (t^p-1)(t^q-1) +(t^{pkn+p+q}-t^{pkn+p+q-1} -t^{p+q} + t^{pn+p-k} + t^p - t^{p-k}) \\
&\; + ( -t^{pn+p-k} +t^q + t^{p-k} -1 - t^{1-pkn}+t^{-pkn})\\
= &\; t^{kpn+p+q} - t^{kpn +p + q -1} - t^{1-kpn} + t^{-kpn}\\
= & \; t^{-(p-1)(q-1)/2}(t^{pq}-1)(t-1),
\end{align*}
proving our result.
\end{proof}

We prove Theorem \ref{theorem:HFKWidth} using Corollary \ref{corollary:HFKWidth} and Proposition \ref{proposition:Alex}.
\begin{proof}[Proof of Theorem \ref{theorem:HFKWidth}]

The proof of this theorem breaks up into cases, one for each of the formulas appearing in Proposition \ref{proposition:Alex}. In each case, the symmetry of $\widehat{HFK}(K)$ ensures that it is enough to consider the constant term and the $\epsilon=1$ terms in the sum. The strategy for each case is to apply Corollary \ref{corollary:HFKWidth} to Proposition \ref{proposition:Alex}. For notational convenience, we will reindex the $\delta_{\ell}$-terms in Corollary \ref{corollary:HFKWidth} to more closely match the indices of the sums in Proposition \ref{proposition:Alex}.

Suppose $p=2k+1$. The terms with positive degree in $\Delta_{T_{p,pn+1}}(t)$ come in pairs. We reindex the terms in Corollary \ref{corollary:HFKWidth} as $\delta_{i,j,1}$ and $\delta_{i,j,2}$ to correspond to the $t^i$ and $1$ terms respectively in the sum with indices $i$ and $j$. The initial $\delta$-term of Corollary \ref{corollary:HFKWidth} is $\delta_{1,0,1}=pkn$. Each $\delta_{i,j,2}$ term occurs an odd number of times after the first step in the iteration. Hence $\delta_{i,j,2} = \delta_{i,j,1} + i - 1$. Each $\delta_{i,j,1}$ term occurs an even number of times after the first step. For each $i$, we have  $\delta_{i,j+1,1} = \delta_{i,j,1} - p + 2i$ if $0< j < n-1$ and  $\delta_{i+1,0,1}=\delta_{i,n-1,1} - p+2i$.  Finally, if $\delta_{0}$ is the term corresponding to the constant term in the Alexander polynomial, then $\delta_{0}=\delta_{k,n-1,1}-1$. Thus $\delta_{\max}(T_{p,pn+1})=\delta_{1,0,1}$, $\delta_{\min}(T_{p,pn+1})=\delta_{0}$ and
\begin{align*}
\delta_{\max}(T_{p,pn+1}) - \delta_{\min}(T_{p,pn+1}) = & \; n \sum_{i=1}^k p-2i\\
= & \; n \sum_{i=1}^k 2i-1\\
= & \; nk^2 \\
= & n \left\lfloor \frac{(p-1)^2}{4}\right\rfloor.
\end{align*}

The terms with positive degree in $\Delta_{T_{p,pn-1}}$ come in pairs. We reindex the terms in Corollary \ref{corollary:HFKWidth} as $\delta_{i,j,1}$ and $\delta_{i,j,2}$ to correspond to the $t^i$ and $1$ terms respectively in the sum with indices $i$ and $j$. The initial $\delta$-term is $\delta_{1,0,1} = pkn - 2k$. Each $\delta_{i,j,2}$ term occurs an odd number of times after the first step in the iteration. Hence $\delta_{i,j,2} = \delta_{i,j,1} + i - 1$. Each $\delta_{i,j,1}$ term occurs an even number of steps after the initial term. For each $i$, we have $\delta_{i,j+1,1} = \delta_{i,j,1} - p + 2i$ if $0 < j < n-1$ and $\delta_{i+1,0,1} = \delta_{i,n-1,1} - p +2i + 1$. The term corresponding to the constant term in the Alexander polynomial is $\delta_{k,n-1,2}$. Thus $\delta_{\max}(T_{p,pn-1})=\delta_{1,0,1}$, $\delta_{\min}(T_{p,pn-1})=\delta_{k,n-1,1}$ and
\begin{align*}
\delta_{\max}(T_{p,pn-1}) - \delta_{\min}(T_{p,pn-1}) = & \; (n-1) \sum_{i=1}^k (p-2i) + \sum_{i=1}^k (p-2i -1) \\
= & \; (n-1) \sum_{i=1}^k (2i-1) + 2 \sum_{i=1}^{k-1} i\\
= & \; (n-1)k^2 + (k-1)k\\
= & \; nk^2 - k\\
= & \; n\left\lfloor \frac{(p-1)^2}{4}\right\rfloor - \left\lfloor \frac{p-1}{2} \right\rfloor.
\end{align*}

Let $p=2k$. There are two sums  in $\Delta_{T_{p,pn+1}}(t)$. In the first sum, the terms come in pairs. The terms will be labeled by $\delta_{i,j,1}$ and $\delta_{i,j,2}$ to correspond to the $t^i$ and $1$ terms respectively. In the second sum, we label each term as $\tilde{\delta}_i$ to correspond with the term $t^{ik}$. The constant term is again labeled $\delta_0$.  The initial $\delta$-term is $\delta_{1,0,1}=(p-1)kn$. Each $\delta_{i,j,2}$ term occurs an odd number of steps after the initial term. Hence $\delta_{i,j,2}=\delta_{i,j,1} + i - 1$. Each $\delta_{i,j,1}$ term occurs an even number of steps after the initial term. For each $i$, we have $\delta_{i,j+1,1} = \delta_{i,j,1} - p +2i$ if $0<j<n-1$ and $\delta_{i+1,0,1}=\delta_{i,n-1,1} - p+2i$. The term in the second sum with highest degree is the term indexed by $n-1$. We have $\tilde{\delta}_{n-1} = \delta_{k-1,n-1,2} - k = \delta_{k-1,n-1,1} +( k-1)-1 -k  = \delta_{k-1,n-1,1}-2.$ After that point we have
$\tilde{\delta}_{n-1-i} = \tilde{\delta}_{n-1}$ if $i$ is even and $\tilde{\delta}_{n-1-i} = \tilde{\delta}_{n-1}+k-1$ if $i$ is odd. Finally $\delta_0 = \tilde{\delta}_{n-1} + k -1$. Thus $\delta_{\max}(T_{p,pn+1}) = \delta_{1,0,1}$, $\delta_{\min}(T_{p,pn+1})=\tilde{\delta}_{n-1}$, and
\begin{align*}
\delta_{\max}(T_{p,pn+1}) - \delta_{\min}(T_{p,pn+1})= & \; n \sum_{i=1}^{k-1} p-2i\\
= &\; 2n \sum_{i=1}^{k-1}i\\
= & \; nk(k+1)\\
= &\;  n \left\lfloor \frac{(p-1)^2}{4}\right\rfloor.
 \end{align*}
 
There are two sums in $\Delta_{T_{p,pn-1}}(t)$. In the first sum, the terms come in pairs, and will be labeled by $\delta_{i,j,1}$ and $\delta_{i,j,2}$ to correspond to the $t^i$ and $1$ terms respectively. In the second sum, we label the term corresponding to $t^{ik}$ by $\tilde{\delta}_i$. The term corresponding to the constant term in the Alexander polynomial is again labeled $\delta_0$. The initial $\delta$-term is $\delta_{1,0,1} = (p-1)(kn-1)$. Each $\delta_{i,j,2}$ term occurs an odd number of steps after the initial term. Hence $\delta_{i,j,2}=\delta_{i,j,+1} - i + 1$. Each $\delta_{i,j,1}$ term occurs an even number of steps after the initial term. For each $i$, we have $\delta_{i,j+1,1} = \delta_{i,j,1} - p + 2i$ if $0 < j < n-1$ and $\delta_{i+1,0,1} = \delta_{i,n-1,1} - p +2i + 1$. The term in the second sum with greatest degree is $\tilde{\delta}_{n-1}$. We have $\tilde{\delta}_{n-1} = \delta_{k-1,n-1,2} - k + 1 = (\delta_{k-1,n-1,1} +k-2)  - k +1=\delta_{k-1,n-1,1} - 1$. After that point we have
$\tilde{\delta}_{n-1-i} = \tilde{\delta}_{n-1}$ if $i$ is even and $\tilde{\delta}_{n-1-i} = \tilde{\delta}_{n-1}+k-1$ if $i$ is odd.
Finally $\delta_0 =  \tilde{\delta}_{n-1} + k -1$. Thus $\delta_{\max}(T_{p,pn-1}) = \delta_{1,0,1}$, $\delta_{\min}(T_{p,pn-1})=\tilde{\delta}_{n-1}$, and
\begin{align*}
\delta_{\max}(T_{p,pn-1}) - \delta_{\min}(T_{p,pn-1}) = & \;(n-1) \sum_{i=1}^k (p-2i) + \sum_{i=1}^k (p-2i -1)\\
= &\; (n-1) \sum_{i=1}^{k-1} 2i + \sum_{i=1}^{k-1} 2i-1\\
= &\; (n-1)(k-1)k + (k-1)^2\\
= & \; nk^2 - nk - k +1\\
= &  \; n\left\lfloor \frac{(p-1)^2}{4}\right\rfloor - \left\lfloor \frac{p-1}{2} \right\rfloor.
\end{align*}
This proves the theorem for $T_{p,pn\pm 1}$. A similar analysis yields the knot Floer width for the knots $T_{5,5n+2}$ and $T_{5,5n+3}$.
\end{proof}

\section{Turaev genus minimizing diagrams of $T_{p,q}$}
\label{section:diagrams}

In this section, we construct diagrams of $T_{p,q}$ whose Turaev genus and dealternating number are equal to or just slightly larger than the lower bounds given by Theorems \ref{theorem:HFKWidth} and \ref{theorem:LowerBound}. We also prove Theorems \ref{theorem:TuraevExact} and \ref{theorem:DaltBest}.

Many of the diagrams in this section are in closed braid form. Let $B_p$ denote the braid group on $p$ strands, and let $\sigma_i$ denote the braid where strand $i+1$ passes over strand $i$ as in Figure \ref{figure:braidgen}. The braid group $B_p$ is generated by $\sigma_i$ for $i=1,\dots, p-1$. The relations in $B_p$ come in two formats:
\begin{align}
\label{equation:rel1}
\sigma_i \sigma_j = &\;  \sigma_j \sigma_i~\text{if $|i-j|>1$, and}\\
\label{equation:rel2}
\sigma_i \sigma_{i+1}\sigma_i =&\; \sigma_{i+1}\sigma_i\sigma_{i+1}~\text{if $i=1,\dots,p-2$.}
\end{align}

Since the braid words are rather long, we adopt the following convention. The braid generator $\sigma_i$ will be denoted by the integer $i$. A product of braid generators $\sigma_{i_1}\sigma_{i_2}\cdots \sigma_{i_k}$ is represented by the string $i_1 i_2 \cdots i_k$. A power of braid generators $(\sigma_{i_1}\sigma_{i_2}\cdots \sigma_{i_k})^j$ is represented by $(i_1 i_2 \cdots i_k)^j$. We only use positive braid generators in this article.

\begin{figure}[h]
$$\begin{tikzpicture}[thick]
\draw (0,1) -- (0,0);
\draw (1,1) -- (1,0);
\draw (4,1) -- (3,0);
\draw (3,1) -- (3.3,.7);
\draw (3.7,.3) -- (4,0);
\draw (6,1) -- (6,0);
\draw (7,1) -- (7,0);
\draw (5,.5) node{$\cdots$};
\draw (2,.5) node{$\cdots$};
\draw (0,1) node[above]{1};
\draw (1,1) node[above]{2};
\draw (3,1) node[above]{i};
\draw (4,1) node[above]{i+1};
\draw (6,1) node[above]{p-1};
\draw (7,1) node[above]{p};
\end{tikzpicture}$$
\caption{The braid generator $\sigma_i$ in $B_p$, the braid group on $p$ strands .}
\label{figure:braidgen}
\end{figure}
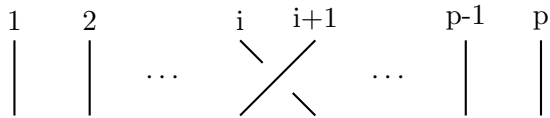

The torus knot $T_{p,q}$ is the closure of the braid $(1 2 3 \cdots p-1)^q$. A full twist $\Delta_p\in B_p$ is the braid $(1 2 3 \cdots p-1)^p$. The full twist $\Delta_p$ is in the center of the braid group $B_p$. If $q=pn + r$ where $0\leq r < p$, then $T_{p,q}$ is the closure of the braid $\Delta_p^n (1 2 3 \cdots p-1)^r$. Alternate forms of the full twists $\Delta_4$, $\Delta_5$, and $\Delta_6$ are the building blocks for the proofs of Theorems \ref{theorem:TuraevExact} and \ref{theorem:DaltBest}. Our strategy is to first find a diagram of the closure of $\Delta_p$ whose Turaev surface has small genus, then to do the same for the closure of $(\Delta_p)^n$, and finally for the closure of $\Delta_p^n (1 2 3 \cdots p-1)^r$. 

For any link diagram $D$, the inequality $g_T(D)\leq \dalt(D)$ holds. So a strategy that minimizes $\dalt(D)$ will give diagrams of small Turaev genus as well. Since all of our braids are positive, changing the even indexed crossings ($\sigma_2$ in $T_{4,q}$ and $\sigma_2$ and $\sigma_4$ in $T_{5,q}$ and $T_{6,q}$) will result in an alternating diagram.  Work from \cite{Lowrance:Twisted} implies that replacing any crossing with a positive power of that crossing does not change the genus of the Turaev surface. This suggests our strategy to minimize the Turaev genus of the closure of $\Delta_p$ should be to find a diagram with the fewest number of groups of even indexed crossings.

 Figures \ref{figure:4-twist}, \ref{figure:5-twist}, and \ref{figure:6-twist} show that
 \begin{align}
 \Delta_4 = & \; 113321322132,\\
 \Delta_5 = & \; (2311234311)^2,~\text{and}\\
 \Delta_6 = & 2\zeta234\eta43,\\
 \intertext{where}
 \zeta = &\; 133545332334513\in B_6~\text{and}\\
\eta =&\;  315531213\in B_6.
 \end{align}
 In Figures \ref{figure:4-twist}, \ref{figure:5-twist}, and \ref{figure:6-twist}, the shaded regions indicate the portions of the diagram that will be changed to obtain the subsequent diagram.
 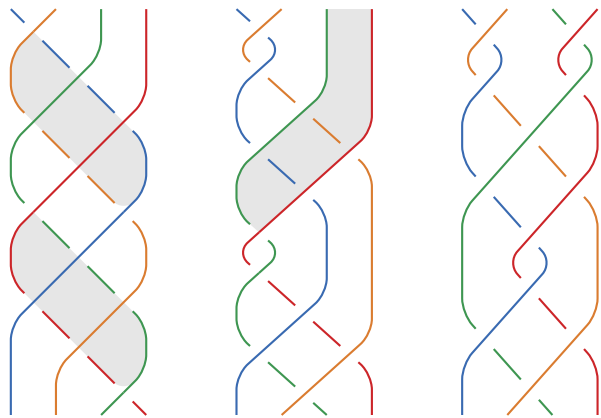
\begin{figure}[h]
 $$\begin{tikzpicture}[scale = .6]
 \begin{scope}[rounded corners = 2mm, thick]
 	\draw[Red] (3,0) -- (2.7,.3);
 	\draw[ Red] (2.3,.7) -- (1.7,1.3);
	 \draw[ Red] (1.3,1.7) -- (.7,2.3);
	 \draw[ Red] (.3,2.7) -- (0,3) -- (0,4) -- (.5,4.5);
	 \draw[Red](.5,4.5) -- (3,7) -- (3,9);
	 
	 \draw[Green] (2,0) -- (2.5,.5);
	 \draw[Green] (2.5,.5) -- (3,1) -- (3,2) -- (2.7,2.3);
	 \draw[ Green] (2.3,2.7) -- (1.7,3.3);
	 \draw[Green] (1.3,3.7) -- (.7,4.3);
	 \draw[Green] (.3,4.7) -- (0,5) -- (0,6) -- (2,8) -- (2,9);
	 
	 \draw[Orange] (1,0) -- (1,1) -- (3,3) -- (3,4) -- (2.7,4.3);
	 \draw[Orange] (2.3,4.7) -- (1.7,5.3);
	 \draw[Orange] (1.3,5.7) -- (.7,6.3);
	 \draw[Orange] (.3,6.7) -- (0,7) -- (0,8) -- (0.5,8.5);
	 \draw[Orange] (0.5,8.5) -- (1,9);
	 
	 \draw[Blue] (0,0) -- (0,2) -- (2.5,4.5);
	 \draw[Blue] (2.5,4.5) -- (3,5) -- (3,6) -- (2.7,6.3);
	 \draw[Blue] (2.3,6.7) -- (1.7,7.3);
	 \draw[Blue] (1.3,7.7) -- (.7,8.3);
	 \draw[Blue] (.3,8.7) -- (0,9);
	 
	 \begin{pgfonlayer}{background}
	 \fill[black!10!white] (.5,8.5) -- (0,8) -- (0,7) -- (2.5,4.5) -- (3,5) -- (3,6) -- (.5,8.5);
	  \fill[black!10!white, yshift=-4cm] (.5,8.5) -- (0,8) -- (0,7) -- (2.5,4.5) -- (3,5) -- (3,6) -- (.5,8.5);

	 \end{pgfonlayer}

 \end{scope}
 
\begin{scope}[rounded corners = 2mm, thick, xshift  = 5cm, yscale = .9]

	\draw[Red] (3,0) -- (3,1) -- (2.7,1.3);
	\draw[Red] (2.3,1.7) -- (1.7,2.3);
	\draw[Red] (1.3,2.7) -- (.7,3.3);
	\draw[Red] (.3,3.7) -- (0,4) -- (0.5,4.5);
	\draw[Red] (0.5,4.5) -- (3,7) -- (3,10);
	
	\draw[Green] (2,0) -- (1.7,.3);
	\draw[Green] (1.3,.7) -- (.7,1.3);
	\draw[Green] (.3,1.7) -- (0,2) -- (0,3) --(0.5,3.5);
	\draw[Green] (.5,3.5)-- (1,4) -- (.7,4.3);
	\draw[Green] (.3,4.7) -- (0,5) --(0,6) --  (2,8) -- (2,10);
	
	\draw[Orange] (1,10) -- (0,9) -- (.3,8.7);
	\draw[Orange] (.7,8.3) -- (1.3,7.7);
	\draw[Orange] (1.7,7.3) -- (2.3,6.7);
	\draw[Orange] (2.7,6.3) -- (3,6) -- (3,2) -- (1,0);
	
	\draw[Blue] (0,10) -- (.3,9.7);
	\draw[Blue] (.7,9.3) -- (1,9) -- (0,8) -- (0,7) -- (0.3,6.7);
	\draw[Blue] (0,0) -- (0,1) -- (2,3) -- (2,5) -- (1.7,5.3);
	\draw[Blue] (1.3,5.7) -- (.7,6.3);
	
	\begin{pgfonlayer}{background}
	\fill[black!10!white] (.5,4.5) -- (0,5) -- (0,6) -- (2,8) -- (2,10) -- (3,10) -- (3,7) -- (.5,4.5);
	\fill[black!10!white, rounded corners = 0mm] (2,10) -- (3,10) -- (3,9.5) -- (2,9.5);
	\end{pgfonlayer}
  
\end{scope}
\begin{scope}[rounded corners = 2mm, xshift = 10cm, thick, yscale=9/8]
	
	\draw[Red] (3,0) -- (3,1) -- (2.7,1.3);
	\draw[Red] (2.3,1.7) -- (1.7,2.3);
	\draw[Red] (1.3,2.7) -- (1,3) --(3,5) -- (3,6) -- (2.7,6.3);
	\draw[Red] (2.3,6.7) -- (2,7) -- (3,8);
	
	\draw[Blue] (0,0) -- (0,1) -- (2,3) -- (1.7,3.3);
	\draw[Blue] (1.3,3.7) -- (.7,4.3);
	\draw[Blue] (.3,4.7) -- (0,5) -- (0,6) -- (1,7) -- (.7,7.3);
	\draw[Blue] (.3,7.7) -- (0,8);
	
	\draw[Orange] (1,0) -- (3,2) -- (3,4) -- (2.7,4.3);
	\draw[Orange] (2.3,4.7) -- (1.7, 5.3);
	\draw[Orange] (1.3,5.7) -- (.7,6.3);
	\draw[Orange] (.3,6.7) -- (0,7) -- (1,8);
	
	\draw[Green] (2,0) -- (1.7,.3);
	\draw[Green] (1.3,.7) -- (.7,1.3);
	\draw[Green] (.3,1.7) -- (0,2) -- (0,4) -- (3,7) -- (2.7,7.3);
	\draw[Green] (2.3,7.7) -- (2,8);

\end{scope}

\end{tikzpicture}$$
 
\caption{Transforming a full twist $\Delta_4$ on $4$ strands into $113321322132$.}
\label{figure:4-twist}
\end{figure}

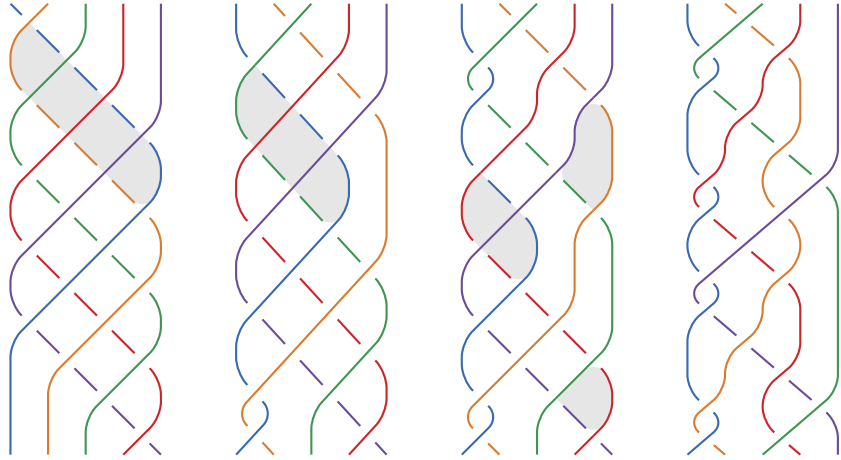
\begin{figure}[h]
$$\begin{tikzpicture}[scale=.5]

\begin{scope}[thick, rounded corners = 2mm]

	\draw[Purple] (4,0) -- (3.7,.3);
	\draw[Purple] (3.3,.7) -- (2.7,1.3);
	\draw[Purple] (2.3,1.7) -- (1.7,2.3);
	\draw[Purple] (1.3,2.7) -- (.7,3.3);
	\draw[Purple] (.3,3.7) -- (0,4) -- (0,5) -- (4,9) -- (4,12);
	
	\draw[Red] (3,0) -- (4,1) -- (4,2) -- (3.7,2.3);
	\draw[Red] (3.3,2.7) -- (2.7,3.3);
	\draw[Red] (2.3,3.7) -- (1.7,4.3);
	\draw[Red] (1.3,4.7) -- (.7,5.3);
	\draw[Red] (.3,5.7) -- (0,6) -- (0,7) -- (3,10) -- (3,12);

	\draw[Green] (2,0) -- (2,1) -- (4,3) -- (4,4) -- (3.7,4.3);
	\draw[Green] (3.3,4.7) -- (2.7,5.3);
	\draw[Green] (2.3,5.7) -- (1.7,6.3);
	\draw[Green] (1.3,6.7) -- (.7,7.3);
	\draw[Green] (.3,7.7) -- (0,8) -- (0,9) -- (2,11) -- (2,12);
	
	\draw[Orange] (1,0) -- (1,2) -- (4,5) -- (4,6) --  (3.7,6.3);
	\draw[Orange] (3.3,6.7) -- (2.7,7.3);
	\draw[Orange] (2.3,7.7) -- (1.7,8.3);
	\draw[Orange] (1.3,8.7) -- (.7,9.3);
	\draw[Orange] (.3,9.7) -- (0,10) -- (0,11) -- (.5,11.5);
	\draw[Orange] (.5,11.5) -- (1,12);
	
	\draw[Blue] (0,0) -- (0,3) -- (3.5,6.5);
	\draw[Blue] (3.5,6.5) --  (4,7) -- (4,8) -- (3.7,8.3);
	\draw[Blue] (3.3,8.7) -- (2.7,9.3);
	\draw[Blue] (2.3,9.7) -- (1.7,10.3);
	\draw[Blue] (1.3,10.7) -- (.7,11.3);
	\draw[Blue] (.3,11.7) -- (0,12);
	
	\begin{pgfonlayer}{background}
	\fill[black!10!white] (.5,11.5) -- (0,11) -- (0,10) -- (3.5,6.5) -- (4,7) -- (4,8) -- (.5,11.5);
	\end{pgfonlayer}

\end{scope}

\begin{scope}[thick, rounded corners = 2mm, xshift = 6cm,yscale = 12/11]
	
	\draw[Orange] (1,0) -- (.7,.3);
	\draw[Orange] (.3,.7) -- (0,1) -- (4,5) -- (4,8) -- (3.7,8.3);
	\draw[Orange] (1,11) -- (1.3,10.7);
	\draw[Orange] (1.7,10.3) -- (2.3,9.7);
	\draw[Orange] (2.7,9.3) -- (3.3,8.7);
	
	\draw[Blue] (0,0) -- (1,1) -- (.7,1.3);
	\draw[Blue] (.3,1.7) -- (0,2) -- (0,3) -- (2.5,5.5);
	\draw[Blue] (2.5,5.5) -- (3,6) -- (3,7) -- (2.7,7.3);
	\draw[Blue] (2.3,7.7) -- (1.7,8.3);
	\draw[Blue] (1.3,8.7) -- (.7,9.3);
	\draw[Blue] (.3,9.7) -- (0,10) -- (0,11);
	
	\draw[Purple] (4,0) -- (3.7,.3);
	\draw[Purple] (3.3,.7) -- (2.7,1.3);
	\draw[Purple] (2.3,1.7) -- (1.7,2.3);
	\draw[Purple] (1.3,2.7) -- (.7,3.3);
	\draw[Purple] (.3,3.7) -- (0,4) -- (0,5) -- (4,9) -- (4,11);
	
	\draw[Red] (3,0) -- (4,1) -- (4,2) -- (3.7,2.3);
	\draw[Red] (3.3,2.7) -- (2.7,3.3);
	\draw[Red] (2.3,3.7) -- (1.7,4.3);
	\draw[Red] (1.3,4.7) -- (.7,5.3);
	\draw[Red] (.3,5.7) -- (0,6) -- (0,7) -- (3,10) -- (3,11);
	
	\draw[Green] (2,0) -- (2,1) -- (4,3) -- (4,4) -- (3.7,4.3);
	\draw[Green] (3.3,4.7) -- (2.7,5.3);
	\draw[Green] (2.3,5.7) -- (1.7,6.3);
	\draw[Green] (1.3,6.7) -- (.7,7.3);
	\draw[Green] (.3, 7.7) -- (0,8) -- (0,9) -- (0.5,9.5);
	\draw[Green] (0.5,9.5) -- (2,11);
	
	\begin{pgfonlayer}{background}
	\fill[black!10!white] (0.5,9.5) -- (0,9) -- (0,8) -- (2.5,5.5) -- (3,6) -- (3,7) -- (.5,9.5);
	\end{pgfonlayer}

\end{scope}
\begin{scope}[thick, rounded corners = 2mm, xshift = 12cm]

	\draw[Orange] (1,0) -- (.7,.3);
	\draw[Orange] (.3,.7) -- (0,1) --  (3,4) -- (3,6) -- (3.5,6.5);
	\draw[Orange] (3.5,6.5) --	(4,7) -- (4,9) -- (3.7,9.3);
	\draw[Orange] (1,12) -- (1.3,11.7);
	\draw[Orange] (1.7,11.3) -- (2.3,10.7);
	\draw[Orange] (2.7,10.3) -- (3.3,9.7);
	
	\draw[Blue] (0,0) -- (1,1) -- (.7,1.3);
	\draw[Blue] (.3,1.7) -- (0,2) -- (0,3) -- (1.5,4.5);
	\draw[Blue] (1.5,4.5) -- (2,5) -- (2,6) -- (1.7,6.3);
	\draw[Blue] (1.3,6.7) -- (.7,7.3);
	\draw[Blue] (.3,7.7) -- (0,8) -- (0,9) -- (1,10) -- (.7,10.3);
	\draw[Blue] (.3,10.7) -- (0,11) -- (0,12);
	
	\draw[Purple] (4,0) -- (3.7,.3);
	\draw[Purple] (3.3,.7) -- (2.7,1.3);
	\draw[Purple] (2.3,1.7) -- (1.7,2.3);
	\draw[Purple] (1.3,2.7) -- (.7,3.3);
	\draw[Purple] (.3,3.7) -- (0,4) -- (0,5) -- (2.5,7.5);
	\draw[Purple] (2.5,7.5) -- (3,8) -- (3,9) -- (3.5,9.5);
	\draw[Purple] (3.5,9.5) --  (4,10) -- (4,12);
	
	\draw[Red] (3,0) --(3.5,.5);
	\draw[Red] (3.5,.5) -- (4,1) -- (4,2) -- (3.7,2.3);
	\draw[Red] (3.3,2.7) -- (2.7,3.3);
	\draw[Red] (2.3,3.7) -- (1.7,4.3);
	\draw[Red] (1.3,4.7) -- (.7,5.3);
	\draw[Red] (.3,5.7) -- (0,6) -- (0,7) -- (0.5,7.5);
	\draw[Red] (.5,7.5) -- (2,9) -- (2,10) -- (3,11) -- (3,12);
	
	\draw[Green] (2,0) -- (2,1) -- (2.5,1.5);
	\draw[Green] (2.5,1.5) -- (3.5,2.5);
	\draw[Green] (3.5,2.5) -- (4,3) -- (4,6) -- (3.7,6.3);
	\draw[Green] (3.3,6.7) -- (2.7,7.3);
	\draw[Green] (2.3,7.7) -- (1.7,8.3);
	\draw[Green](1.3,8.7) -- (.7,9.3);
	\draw[Green] (.3,9.7) -- (0,10) -- (2,12);

	\begin{pgfonlayer}{background}
	\fill[black!10!white] (.5,7.5) -- (0,7) -- (0,6) -- (1.5,4.5) -- (2,5) -- (2,6) -- (.5,7.5);
	\fill[black!10!white] (3.5,6.5) -- (2.5,7.5) -- (3,8) -- (3,9) -- (3.5,9.5) -- (4,9) -- (4,7) -- (3.5,6.5);
	\fill[black!10!white] (2.5,1.5) -- (3.5,2.5) -- (4,2) -- (4,1) -- (3.5,.5) -- (2.5,1.5);
	\end{pgfonlayer}

\end{scope} 

 \begin{scope}[thick, rounded corners = 2mm, xshift = 18cm, yscale=12/14]
 	
	\draw[Blue] (0,0) -- (1,1) -- (.7,1.3);
	\draw[Blue] (.3,1.7) -- (0,2) -- (0,4) -- (1,5) -- (.7,5.3);
	\draw[Blue] (.3,5.7) -- (0,6) -- (0,7) -- (1,8) -- (.7,8.3);
	\draw[Blue] (.3,8.7) -- (0,9) -- (0,11) -- (1,12) -- (.7,12.3);
	\draw[Blue] (.3,12.7) -- (0,13) -- (0,14);
	
	\draw[Orange] (1,0) -- (.7,.3);
	\draw[Orange] (.3,.7) -- (0,1) -- (1,2) --(1,3) -- (2,4) -- (2,5) -- (3,6) -- (3,7) -- (2.7,7.3);
	\draw[Orange] (2.3,7.7) -- (2,8) -- (2,9) -- (3,10) -- (3,12) -- (2.7,12.3);
	\draw[Orange] (1,14) -- (1.3,13.7);
	\draw[Orange] (1.7,13.3) -- (2.3,12.7);
	
	\draw[Green] (2,0) -- (4,2) -- (4,8) -- (3.7,8.3);
	\draw[Green] (3.3,8.7) -- (2.7,9.3);
	\draw[Green] (2.3,9.7) -- (1.7,10.3);
	\draw[Green] (1.3,10.7) -- (.7,11.3);
	\draw[Green] (.3,11.7) -- (0,12) -- (2,14);
	
	\draw[Red] (3,0) -- (2.7,.3);
	\draw[Red] (2.3,.7) -- (2,1) -- (2,2) -- (3,3) -- (3,5) -- (2.7,5.3);
	\draw[Red] (1.7,6.3) -- (2.3,5.7);
	\draw[Red] (1.3,6.7) -- (.7,7.3);
	\draw[Red] (.3,7.7) -- (0,8) -- (1,9) -- (1,10) -- (2,11) -- (2,12) -- (3,13) -- (3,14);
	
	\draw[Purple] (4,0) -- (4,1) -- (3.7,1.3);
	\draw[Purple] (3.3,1.7) -- (2.7,2.3);
	\draw[Purple] (2.3, 2.7) -- (1.7,3.3);
	\draw[Purple] (1.3,3.7) -- (.7,4.3);
	\draw[Purple] (.3,4.7) -- (0,5) -- (4,9) -- (4,14);
 
 \end{scope}
\end{tikzpicture}$$
\caption{Transforming a full twist $\Delta_5$ on $5$ strands into $(2311234311)^2$.}
\label{figure:5-twist}
\end{figure}

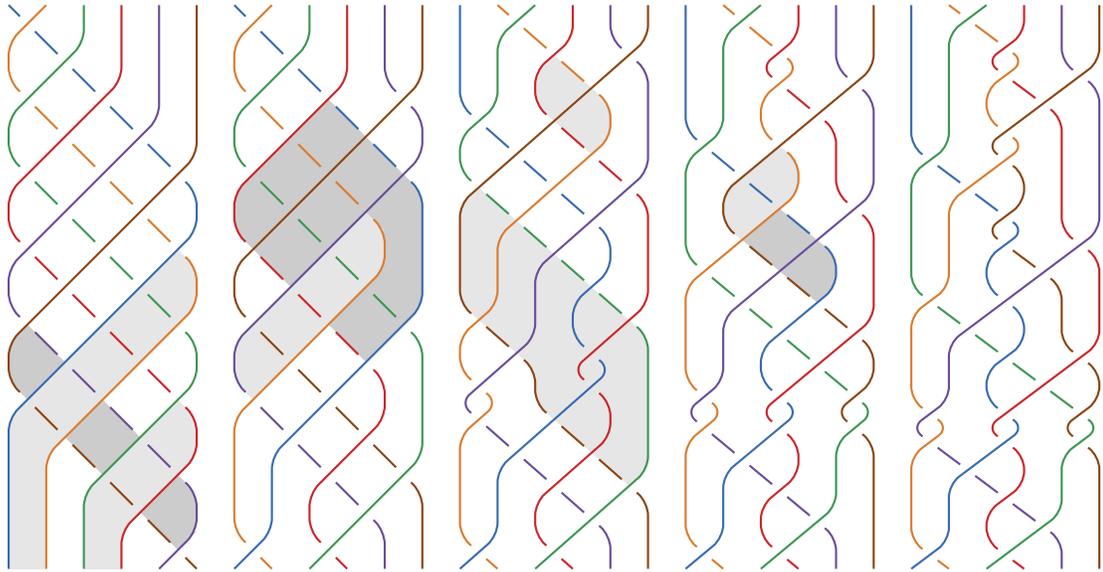
\begin{figure}[h]
$$\begin{tikzpicture}[scale=.5, rounded corners = 1.7mm, thick]
 
	\draw[Blue] (0,0) -- (0,4) -- (5,9)  --(5,10)-- (4.7,10.3);
	\draw[Blue] (4.3,10.7) -- (3.7,11.3);
	\draw[Blue] (3.3,11.7) -- (2.7,12.3);
	\draw[Blue] (2.3,12.7) -- (1.7,13.3);
	\draw[Blue] (1.3, 13.7) -- (.7,14.3);
	\draw[Blue] (.3,14.7) -- (0,15);
	
	\draw[Orange] (1,0) -- (1,3) -- (5,7) -- (5,8) -- (4.7,8.3);
	\draw[Orange] (4.3,8.7) -- (3.7,9.3);
	\draw[Orange] (3.3,9.7) -- (2.7,10.3);
	\draw[Orange] (2.3,10.7) -- (1.7,11.3);
	\draw[Orange] (1.3,11.7) -- (.7,12.3);
	\draw[Orange] (.3,12.7) -- (0,13) -- (0,14) -- (1,15);
	
	\draw[Green] (2,0) -- (2,2) -- (5,5) -- (5,6) -- (4.7,6.3);
	\draw[Green] (4.3,6.7) -- (3.7,7.3);
	\draw[Green] (3.3,7.7) -- (2.7,8.3);
	\draw[Green] (2.3,8.7) -- (1.7,9.3);
	\draw[Green] (1.3,9.7) -- (.7,10.3);
	\draw[Green] (.3,10.7) -- (0,11) -- (0,12) -- (2,14) -- (2,15);
	
	\draw[Red] (3,0) -- (3,1) -- (5,3) -- (5,4) -- (4.7,4.3);
	\draw[Red] (4.3,4.7) -- (3.7,5.3);
	\draw[Red] (3.3,5.7) -- (2.7,6.3);
	\draw[Red] (2.3,6.7) -- (1.7,7.3);
	\draw[Red] (1.3,7.7) -- (.7,8.3);
	\draw[Red] (.3,8.7) -- (0,9) -- (0,10) -- (3,13) -- (3,15);
	
	\draw[Purple] (4,0) -- (5,1) -- (5,2) -- (4.7,2.3);
	\draw[Purple] (4.3,2.7) -- (3.7,3.3);
	\draw[Purple] (3.3,3.7) -- (2.7,4.3);
	\draw[Purple] (2.3,4.7) -- (1.7,5.3);
	\draw[Purple] (1.3,5.7) -- (.7,6.3);
	\draw[Purple] (.3,6.7) -- (0,7) -- (0,8) -- (4,12) -- (4,15);

	\draw[Brown] (5,0) -- (4.7,.3);
	\draw[Brown] (4.3,.7) -- (3.7,1.3);
	\draw[Brown] (3.3,1.7) -- (2.7,2.3);
	\draw[Brown] (2.3,2.7) -- (1.7,3.3);
	\draw[Brown] (1.3,3.7) -- (.7,4.3);
	\draw[Brown] (.3,4.7) -- (0,5) -- (0,6) -- (5,11) -- (5,15);
	
	\begin{pgfonlayer}{background}

	\fill[black!10!white] (4.5,8.5) -- (5,8) -- (5,7) -- (1,3) -- (1,0) -- (0,0) -- (0,4) -- (4.5,8.5);
	\fill[black!10!white,rounded corners=0mm] (0,0) -- (0,1) -- (1,1) -- (1,0);
	\fill[black!10!white] (4.5,4.5) -- (5,4) -- (5,3) -- (3,1) -- (3,0) -- (2,0) --(2,2) -- (4.5,4.5);
	\fill[black!10!white,rounded corners=0mm] (2,0) -- (3,0) -- (3,1) -- (2,1) -- (2,0);
	
	\end{pgfonlayer}
	
	\begin{pgfonlayer}{background2}
	\fill[black!20!white] (.5,6.5) -- (0,6) -- (0,5) --(4.5,.5) -- (5,1) -- (5,2) -- (.5,6.5);
	\fill[black!20!white] (4.5,.5) -- (4,1) -- (5,1);
	\end{pgfonlayer}
\begin{scope}[xshift=6cm, thick]
	
	\draw[Blue] (0,0) -- (1,1) -- (1,3) -- (5,7) -- (5,10) -- (4.7,10.3);
	\draw[Blue] (4.3,10.7) -- (3.7,11.3);
	\draw[Blue] (3.3,11.7) -- (2.7,12.3);
	\draw[Blue] (2.3,12.7) -- (1.7,13.3);
	\draw[Blue] (1.3,13.7) -- (.7, 14.3);
	\draw[Blue] (.3,14.7) -- (0,15);
	
	\draw[Orange] (1,0) -- (.7,.3);
	\draw[Orange] (.3,.7) -- (0,1) -- (0,4) -- (4,8) -- (4,9) -- (3.7,9.3);
	\draw[Orange] (3.3,9.7) -- (2.7,10.3);
	\draw[Orange] (2.3,10.7) -- (1.7,11.3);
	\draw[Orange] (1.3,11.7) -- (.7,12.3);
	\draw[Orange] (.3,12.7) -- (0,13) --(0,14) -- (1,15);
	
	\draw[Green] (2,0) -- (5,3) -- (5,6) -- (4.7,6.3);
	\draw[Green] (4.3,6.7) -- (3.7,7.3);
	\draw[Green] (3.3,7.7) -- (2.7,8.3);
	\draw[Green] (2.3,8.7) -- (1.7,9.3);
	\draw[Green] (1.3,9.7) -- (.7, 10.3);
	\draw[Green] (.3,10.7) -- (0,11) -- (0,12) -- (2,14) -- (2,15);
	
	\draw[Red] (3,0) -- (2.7,.3);
	\draw[Red] (2.3,.7) -- (2,1) -- (2,2) -- (4,4) -- (4,5) -- (3.7,5.3);
	\draw[Red] (3.3,5.7) -- (2.7,6.3);
	\draw[Red] (2.3,6.7) -- (1.7,7.3);
	\draw[Red] (1.3,7.7) -- (.7,8.3);
	\draw[Red] (.3,8.7) -- (0,9) -- (0,10) -- (3,13) -- (3,15);
	
	\draw[Purple] (4,0) -- (4,1) -- (3.7,1.3);
	\draw[Purple] (3.3,1.7) -- (2.7,2.3);
	\draw[Purple] (2.3,2.7) -- (1.7,3.3);
	\draw[Purple] (1.3,3.7) -- (.7,4.3);
	\draw[Purple] (.3,4.7) -- (0,5) -- (0,6) -- (1,7)-- (5,11) -- (5,12) -- (4.7,12.3);
	\draw[Purple] (4.3,12.7) -- (4,13) -- (4,15);
	
	\draw[Brown] (5,0) -- (5,2) -- (4.7,2.3);
	\draw[Brown] (4.3,2.7) -- (3.7,3.3);
	\draw[Brown] (3.3,3.7) -- (2.7,4.3);
	\draw[Brown] (2.3,4.7) -- (1.7,5.3);
	\draw[Brown] (1.3,5.7) -- (.7,6.3);
	\draw[Brown] (.3,6.7) -- (0,7) -- (0,8) -- (5,13) -- (5,15);
	
	\begin{pgfonlayer}{background}
	\fill[black!10!white] (.5,4.5) -- (0,5) -- (0,6) -- (3.5,9.5) -- (4,9) -- (4,8) -- (.5,4.5);
	\fill[black!10!white] (3.5,9.5) -- (3,9) -- (4,9);
	\end{pgfonlayer}
	
	\begin{pgfonlayer}{background2}
		\fill[black!20!white]  (3.5,5.5) -- (0,9) -- (0,10) -- (2.5,12.5) -- (5,10) -- (5,7) -- (3.5,5.5);
		\fill[black!20!white] (2.5,12.5) -- (2,12) -- (3,12);
	\end{pgfonlayer}

\end{scope}
 
 \begin{scope}[xshift = 12cm, thick, yscale= 15/17]
 
 	\draw[Blue] (0,0) -- (1,1) -- (1,3) --(4,6) -- (3.7,6.3); 
	\draw[Blue] (3.3,6.7) -- (3,7) -- (3,8) -- (4,9) -- (4,10) -- (3.7,10.3);
	\draw[Blue] (3.3,10.7) -- (2.7,11.3);
	\draw[Blue] (2.3,11.7) -- (1.7,12.3);
	\draw[Blue] (1.3,12.7) -- (.7,13.3);
	\draw[Blue] (.3,13.7) -- (0,14) -- (0,17);

	\draw[Orange] (1,0) -- (.7,.3);
	\draw[Orange] (.3,.7) -- (0,1) -- (0,4) -- (1,5) -- (.7,5.3);
	\draw[Orange] (.3,5.7) -- (0,6) -- (0,7) -- (1,8) -- (1,10) -- (4,13) -- (4,14) --(3.7,14.3);
	\draw[Orange] (3.3,14.7) -- (2.7,15.3);
	\draw[Orange] (2.3,15.7) -- (1.7,16.3);
	\draw[Orange] (1.3,16.7) -- (1,17);
	
	\draw[Green] (2,0) -- (5,3) -- (5,7) -- (4.7,7.3);
	\draw[Green] (4.3,7.7) -- (3.7,8.3);
	\draw[Green] (3.3,8.7) -- (2.7,9.3);
	\draw[Green] (2.3,9.7) -- (1.7,10.3);
	\draw[Green] (1.3,10.7) -- (.7, 11.3);
	\draw[Green] (.3,11.7) -- (0,12) -- (0,13) -- (1,14) -- (1,16) -- (2,17);;
	
	\draw[Red] (3,0) -- (2.7,.3);
	\draw[Red] (2.3,.7) -- (2,1) -- (2,2) -- (4,4) -- (4,5) -- (3.7,5.3);
	\draw[Red] (3.3,5.7) -- (3,6) -- (5,8) -- (5,11) -- (4.7,11.3);
	\draw[Red] (4.3,11.7) -- (3.7,12.3);
	\draw[Red] (3.3,12.7) -- (2.7,13.3);
	\draw[Red] (2.3,13.7) -- (2,14) -- (2,15) -- (3,16) -- (3,17);

	\draw[Purple] (4,0) -- (4,1) -- (3.7,1.3);
	\draw[Purple] (3.3,1.7) -- (2.7,2.3);
	\draw[Purple] (2.3,2.7) -- (1.7,3.3);
	\draw[Purple] (1.3,3.7) -- (.7,4.3);
	\draw[Purple] (.3,4.7) -- (0,5) -- (2,7) -- (2,9) -- (5,12) -- (5,15) -- (4.7,15.3);
	\draw[Purple] (4.3,15.7) -- (4,16) -- (4,17);
		
	\draw[Brown] (5,0) -- (5,2) -- (4.7,2.3);
	\draw[Brown] (4.3,2.7) -- (3.7,3.3);
	\draw[Brown] (3.3,3.7) -- (2.7,4.3);
	\draw[Brown] (2.3,4.7) --  (2,5) -- (2,6) -- (1.7,6.3);
	\draw[Brown] (1.3,6.7) -- (.7,7.3);
	\draw[Brown] (.3,7.7) -- (0,8) -- (0,11) -- (5,16) -- (5,17);
	
	\begin{pgfonlayer}{background}
	\fill[black!10!white] (4.5,2.5) -- (2,5) -- (2,6) -- (0,8) -- (0,11) -- (.5,11.5) --(5,7) -- (5,3) -- (4.5,2.5);
	\fill[black!10!white] (.5,11.5) -- (1,11) -- (0,11);
	\fill[black!10!white] (3.5,12.5) -- (2,14) -- (2,15) -- (2.5,15.5) -- (4,14) -- (4,13) -- (3.5,12.5);
	\fill[black!10!white] (2.5,15.5) -- (2,15) -- (3,15);
	
	\end{pgfonlayer}
 
 \end{scope}
 
 
 \begin{scope}[xshift = 18cm, thick, yscale = 15/18]
 
 	\draw[Blue] (0,0) -- (1,1) -- (1,3) -- (3,5) -- (2.7,5.3);
	\draw[Blue] (2.3,5.7) -- (2,6) -- (2,7) -- (4,9) -- (4,10) -- (3.7,10.3);
	\draw[Blue] (3.3,10.7) -- (2.7,11.3);
	\draw[Blue] (2.3,11.7) -- (1.7,12.3);
	\draw[Blue] (1.3,12.7) -- (.7,13.3);
	\draw[Blue] (.3,13.7) -- (0,14) -- (0,18);
	
	\draw[Orange] (1,0) -- (.7,.3);
	\draw[Orange] (.3,.7) -- (0,1) -- (0,4) -- (1,5) -- (.7,5.3);
	\draw[Orange] (.3,5.7) -- (0,6) -- (0,9) -- (3,12) -- (3,13) -- (2.7,13.3);
	\draw[Orange] (2.3,13.7) -- (2,14) -- (2,15) -- (3,16) -- (2.7,16.3);
	\draw[Orange] (2.3,16.7) -- (1.7,17.3);
	\draw[Orange] (1.3,17.7) -- (1,18);
	
	\draw[Green] (2,0) -- (4,2) -- (4,4) -- (5,5) -- (4.7,5.3);
	\draw[Green] (4.3,5.7) -- (3.7,6.3);
	\draw[Green] (3.3,6.7) -- (2.7,7.3);
	\draw[Green] (2.3,7.7) -- (1.7,8.3);
	\draw[Green] (1.3,8.7) -- (.7,9.3);
	\draw[Green] (.3,9.7) -- (0,10) -- (0,13) -- (1,14) -- (1,17) -- (2,18);
	
	\draw[Red] (3,0) -- (2.7,.3);
	\draw[Red] (2.3,.7) -- (2,1) -- (2,2) -- (3,3) -- (3,4) --(2.7,4.3);
	\draw[Red] (2.3,4.7) -- (2,5) -- (5,8) -- (5,11) -- (4.7,11.3);
	\draw[Red] (4.3,11.7) -- (4,12) -- (4,14) -- (3.7,14.3);
	\draw[Red] (3.3,14.7) -- (2.7,15.3);
	\draw[Red] (2.3,15.7) -- (2,16) -- (3,17) -- (3,18);
	
	\draw[Purple] (4,0) -- (4,1) -- (3.7,1.3);
	\draw[Purple] (3.3,1.7) -- (2.7,2.3);
	\draw[Purple] (2.3,2.7) -- (1.7,3.3);
	\draw[Purple] (1.3,3.7) -- (.7,4.3);
	\draw[Purple] (.3,4.7) -- (0,5)  -- (1,6) -- (1,8) -- (5,12) -- (5,15) -- (4.7,15.3);
	\draw[Purple] (4.3,15.7) -- (4,16) -- (4,18);
	
	\draw[Brown] (5,0) -- (5,4) -- (4.7,4.3);
	\draw[Brown] (4.3,4.7) -- (4,5) -- (5,6) -- (5,7) -- (4.7,7.3);
	\draw[Brown] (4.3,7.7) -- (3.7,8.3);
	\draw[Brown] (3.3,8.7) -- (2.7,9.3);
	\draw[Brown] (2.3,9.7) -- (1.7,10.3);
	\draw[Brown] (1.3,10.7) -- (1,11) -- (1,12) -- (5,16) -- (5,18);
	
	\begin{pgfonlayer}{background}
	\fill[black!10!white] (1.5,10.5) -- (1,11) --(1,12) -- (2.5,13.5) -- (3,13) -- (3,12) -- (1.5,10.5);
	\fill[black!10!white] (2.5,13.5) -- (2,13) -- (3,13);
	\end{pgfonlayer}
	
	\begin{pgfonlayer}{background2}
	\fill[black!20!white] (3.5, 8.5) -- (1,11) -- (1,12) -- (1.5,12.5) -- (4,10) -- (4,9) -- (3.5,8.5);
	\end{pgfonlayer}

 \end{scope}
 
 
 \begin{scope}[xshift = 24cm, thick, yscale = 15/20]
 
  	\draw[Blue] (0,0) -- (1,1) -- (1,3) -- (3,5) -- (2.7,5.3);
	\draw[Blue] (2.3,5.7) -- (2,6) -- (2,7) -- (3,8) -- (3,9) -- (2.7,9.3);
	\draw[Blue] (2.3,9.7) -- (2,10) -- (2,11) -- (3,12) -- (2.7,12.3);
	\draw[Blue] (2.3,12.7) -- (1.7,13.3);
	\draw[Blue] (1.3,13.7) -- (.7,14.3);
	\draw[Blue] (.3,14.7) -- (0,15) -- (0,20);
	
	\draw[Orange] (1,0) -- (.7,.3);
	\draw[Orange] (.3,.7) -- (0,1) -- (0,4) -- (1,5) -- (.7,5.3);
	\draw[Orange] (.3,5.7) -- (0,6) -- (0,9) --(1,10) -- (1,13) -- (3,15) -- (2.7,15.3);
	\draw[Orange] (2.3,15.7) -- (2,16) -- (2,17) -- (3,18) -- (2.7,18.3);
	\draw[Orange] (2.3,18.7) -- (1.7,19.3);
	\draw[Orange] (1.3,19.7) -- (1,20);
	
	\draw[Green] (2,0) -- (4,2) -- (4,4) -- (5,5) -- (4.7,5.3);
	\draw[Green] (4.3,5.7) -- (3.7,6.3);
	\draw[Green] (3.3,6.7) -- (2.7,7.3);
	\draw[Green] (2.3,7.7) -- (1.7,8.3);
	\draw[Green] (1.3,8.7) -- (.7,9.3);
	\draw[Green] (.3,9.7) -- (0,10) -- (0,14) -- (1,15) -- (1,19) -- (2,20);
		
	\draw[Red] (3,0) -- (2.7,.3);
	\draw[Red] (2.3,.7) -- (2,1) -- (2,2) -- (3,3) -- (3,4) --(2.7,4.3);
	\draw[Red] (2.3,4.7) -- (2,5) -- (5,8) -- (5,11) -- (4.7,11.3);
	\draw[Red] (4.3,11.7) -- (4,12) -- (4,16) -- (3.7,16.3);
	\draw[Red] (3.3,16.7) -- (2.7,17.3);
	\draw[Red] (2.3,17.7) -- (2,18) -- (3,19) -- (3,20);
		
	\draw[Purple] (4,0) -- (4,1) -- (3.7,1.3);
	\draw[Purple] (3.3,1.7) -- (2.7,2.3);
	\draw[Purple] (2.3,2.7) -- (1.7,3.3);
	\draw[Purple] (1.3,3.7) -- (.7,4.3);
	\draw[Purple] (.3,4.7) -- (0,5)  -- (1,6) -- (1,8) -- (5,12) -- (5,17) -- (4.7,17.3); 
	\draw[Purple] (4.3,17.7) -- (4,18) -- (4,20);
	
	\draw[Brown] (5,0) -- (5,4) -- (4.7,4.3);
	\draw[Brown] (4.3,4.7) -- (4,5) -- (5,6) -- (5,7) -- (4.7,7.3); 
	\draw[Brown] (4.3, 7.7) -- (4,8) -- (4,10) -- (3.7,10.3);
	\draw[Brown] (3.3,10.7) -- (2.7,11.3);
	\draw[Brown] (2.3,11.7) -- (2,12) -- (3,13) -- (3,14) -- (2.7,14.3);
	\draw[Brown] (2.3,14.7) -- (2,15) -- (5,18) -- (5,20);
	
 \end{scope}

\end{tikzpicture}$$
\caption{Transforming a full twist $\Delta_6$ on $6$ strands into $2\zeta234\eta43$.}
\label{figure:6-twist}
\end{figure}

Lemmas \ref{lemma:T4q}, \ref{lemma:T5q}, and \ref{lemma:T6q} give equivalent braid words for all torus links on four or five  strands and the torus links of the form $T_{6,6n}$ and $T_{6,6n+1}$. The main computational tools are Equation \ref{equation:rel1} and \ref{equation:rel2}. However, in the case for $T_{5,5n+3}$, we need one additional tool, namely cyclic permutation of the braid word. While this can change the element of the braid group, it does not change the link type of the closure. If two braids $\beta_1$ and $\beta_2$ are related by a cyclic permutation of the braid word, we write $\beta_1 \equiv \beta_2$.

The computations in these proofs can involve lengthy braid words. In order to guide the reader, we adopt the following conventions. If a word $w$ is to be replaced by an equivalent word using the braid relation, we underline the word $\underline{w}$. If we swap two commuting words
$w_1$ and  $w_2$, then we indicate the move by $\underrightarrow{w_1}\underleftarrow{w_2}$. If a word $w$ is to be replaced by an equivalent word coming from a previous computation or the inductive hypothesis, then we indicate it by $\uwave{w}$.

\begin{lemma}
\label{lemma:T4q}
For each positive integer $n$, the following equalities hold in the braid group $B_4$:
\begin{itemize}
\item $(123)^{4n} = 1^{2n} 3^{2n} (2132)^{2n}$,
\item $(123)^{4n+1} = 1^{2n}3^{2n}(2132)^{2n-1}2131213$,
\item $(123)^{4n+2} = 1^{2n+2}3^{2n}(2132)^{2n+1}$, and
\item $(123)^{4n+3} = 1^{2n+2}3^{2n}(2132)^{2n}2131213$.
\end{itemize}
\end{lemma}
\begin{proof}
We prove that  $(123)^{4n} = 1^{2n} 3^{2n} (2132)^{2n}$ by induction on $n$. In the case where $n=1$, 
Figure \ref{figure:4-twist} shows that $(123)^4 = 113321322132$. Recall that $(123)^4$ is in the center of $B_4$. Hence
\begin{align*}
(123)^{4(n+1)} = & \; \uwave{(123)^{4n}}(123)^4\\
=& \; 1^{2n} 3^{2n} \underrightarrow{(2132)^{2n}} \underleftarrow{(123)^4}\\
 = & \; 1^{2n} 3^{2n} \uwave{(123)^4} (2132)^{2n}\\
 = & \; 1^{2n} \underrightarrow{3^{2n}}  \underleftarrow{11}33213221132  (2132)^{2n}\\
 = & \; 1^{2(n+1)} 3^{2(n+1)} (2132)^{2(n+1)}.\\
\end{align*}

Figure \ref{figure:4-commute} shows that $11332132 = 21321133$. Now since $(123)^{4n}$ is in the center of $B_4$, we have
\begin{align*}
(123)^{4n+1} = &\; \uwave{(123)^{4n}} 123\\
= & \; 1^{2n} 3^{2n} (2132)^{2n-1} 213\underline{212}3\\
= & \; 1^{2n} 3^{2n} (2132)^{2n-1} 2131213,\\
(123)^{4n+2} = &\; 123123\uwave{(123)^{4n}}\\
= &\; 12\underrightarrow{3}\underleftarrow{1}231^{2n} 3^{2n}  (2132)^{2n}\\
= & \;121\underline{323} 1^{2n} 3^{2n} (2132)^{2n}\\
 = & \;1\underline{212}321^{2n} 3^{2n}  (2132)^{2n}\\
  = & \;11\underrightarrow{2132} \underleftarrow{1^{2n} 3^{2n}}  (2132)^{2n}\\
  =& \;  1^{2n+2}3^{2n}(2132)^{2n+1},~\text{and}\\
(123)^{4n+3} = & \; \uwave{(123)^{4n+2}}123\\
= & 1^{2n+2}3^{2n}(2132)^{2n}213\underline{212}3\\
= & 1^{2n+2}3^{2n}(2132)^{2n}2131213.
\end{align*}
\end{proof}

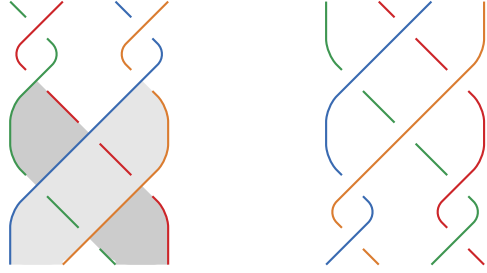
\begin{figure}[h]
$$\begin{tikzpicture}[thick, rounded corners  = 2mm,scale = .7]
	\draw[Blue]  (0,0) -- (0,1) -- (3,4) -- (2.7,4.3);
	\draw[Blue] (2.3,4.7) -- (2,5);
	
	\draw[Orange] (1,0) -- (3,2) --(3,3) -- (2.7,3.3);
	\draw[Orange] (2.3,3.7) -- (2,4) -- (3,5);
	
	\draw[Green] (2,0) -- (1.7,.3);
	\draw[Green] (1.3,.7) -- (.7,1.3);
	\draw[Green] (.3,1.7) -- (0,2) -- (0,3) -- (1,4) -- (.7,4.3);
	\draw[Green] (.3,4.7) -- (0,5);

	\draw[Red] (3,0) -- (3,1) -- (2.7,1.3);
	\draw[Red] (2.3,1.7) -- (1.7,2.3);
	\draw[Red] (1.3,2.7) -- (.7,3.3);
	\draw[Red] (.3,3.7) -- (0,4) -- (1,5);
	
	\begin{pgfonlayer}{background}
	\fill[black!10!white](2.5,3.5) -- (3,3) -- (3,2) -- (1,0) -- (0,0) -- (0,1) --(2.5,3.5); 
	\fill[black!10!white] (0,0) -- (0,.5) -- (1,.5) -- (1,0);
	\end{pgfonlayer}
	
	\begin{pgfonlayer}{background2}
	\fill[black!20!white] (.5,3.5) -- (0,3) -- (0,2) -- (2,0) -- (3,0) -- (3,1) -- (.5,3.5);
	\fill[black!20!white] (2,0) -- (2,.5) -- (3,.5) -- (3,0);
	\end{pgfonlayer}
	
\begin{scope} [xshift = 6cm, thick]

	\draw[Blue] (0,0) -- (1,1) -- (.7,1.3);
	\draw[Blue] (.3,1.7) -- (0,2) -- (0,3) -- (2,5);
	
	\draw[Orange] (1,0) -- (.7,.3);
	\draw[Orange] (.3,.7) -- (0,1) -- (3,4) -- (3,5);
	
	\draw[Green] (2,0) -- (3,1) -- (2.7,1.3);
	\draw[Green] (2.3,1.7) -- (1.7,2.3);
	\draw[Green] (1.3,2.7) -- (.7,3.3);
	\draw[Green] (.3,3.7) -- (0,4) -- (0,5);
	
	\draw[Red] (3,0) -- (2.7,.3);
	\draw[Red] (2.3,.7) -- (2,1) -- (3,2) -- (3,3) -- (2.7,3.3);
	\draw[Red] (2.3,3.7) -- (1.7,4.3);
	\draw[Red] (1.3,4.7) -- (1,5);

\end{scope}
	
\end{tikzpicture}$$
\caption{The braid words $1133$ and $2132$ commute. Thus $11332132 = 21321133$.}
\label{figure:4-commute}
\end{figure}

\begin{lemma}
\label{lemma:T5q}
Let $\alpha,\beta,\gamma\in B_5$ be defined by
$$\alpha = 323311,~\beta=31123112334311,~\text{and}~\gamma=31123112311.$$
For each positive integer $n$, the following equalities hold in $B_5$:
\begin{itemize}
\item $(1234)^{5n} = 2311 \alpha^{n-1} 234 \beta^{n-1}\gamma 43$,
\item $(1234)^{5n+1} = 1323311 \alpha^{n-1} 234 \beta^{n-1}\gamma 343$,
\item $(1234)^{5n+2} =  1231323311\alpha^{n-1} 234 \beta^{n-1}\gamma3433$,
\item $(1234)^{5n+3} \equiv 12131231323311\alpha^{n-1} 234 \beta^{n-1}\gamma 3433,$ and
\item $(1234)^{5n+4} = 2311 \alpha^n 234 \beta^n 311231422.$
\end{itemize}
\end{lemma}
\begin{proof}
We show that $(1234)^{5n} = 2311 \alpha^{n-1} 234 \beta^{n-1}\gamma 43$ by induction on $n$. For $n=1$, Figure \ref{figure:5-twist} shows that $(1234)^5 = (2311234311)^2$. We have 
\begin{align*}
(1234)^{5(n+1)} = & \; \uwave{(1234)^{5n}} \; \uwave{(1234)^5}\\
= & \; 2311 \alpha^{n-1} 23\underrightarrow{4 \beta^{n-1}\gamma 43} \underleftarrow{(2311234311)^2}\\
= &  \; 2311 \alpha^{n-1} \underline{232}311234311 23112343\underrightarrow{11} \underleftarrow{4}\beta^{n-1}\gamma 43 \\
= & \; 2311 \alpha^{n-1} 323311234311231123\underline{434}11  \beta^{n-1}\gamma 43\\
= & \;  2311 \alpha^{n-1} \uwave{323311}234\uwave{31123112334311} \beta^{n-1}\gamma 43\\
= & \;  2311 \alpha^{n-1} \alpha 234\beta \beta^{n-1}\gamma 43\\ 
= & \; 2311 \alpha^{n} 234 \beta^{n}\gamma 43.
\intertext{Also,}
(1234)^{5n+1} =&\; \underrightarrow{(1234)^{5n}} \underleftarrow{(123}4)\\
= & \; 123 \uwave{(1234)^{5n}} 4\\
= &\; 1 \underline{232} 311 \alpha^{n-1} 234 \beta^{n-1}\gamma \underline{434}\\
= &\; 1323311 \alpha^{n-1} 234 \beta^{n-1}\gamma 343.\\
\intertext{Figure \ref{figure:5-braid1} shows that $12341234 = 12312343$. Thus }
(1234)^{5n+2} = &\; (1234)^{5n} \uwave{(12341234)}\\
= &\; \underrightarrow{(1234)^{5n}} \underleftarrow{(123123}43)\\
= &\; 123123\uwave{(1234)^{5n}} 43\\
= & \; 1231\underline{23 2}311 \alpha^{n-1} 234 \beta^{n-1}\gamma \underline{434}3\\
= & \; 123132 3311 \alpha^{n-1} 234 \beta^{n-1}\gamma 3433.
\intertext{Figure \ref{figure:5-braid2} shows that $(1234)^3 = (123)^3 432$. Thus}
(1234)^{5n+3} = & \; (1234)^{5n}\uwave{(1234)^3}\\
= & \; \underrightarrow{(1234)^{5n}}\underleftarrow{(123)^3} 432\\
= & \; (123)^3 \uwave{(1234)^{5n}} 432\\
= &\; 123 123 1\underline{23 2}311 \alpha^{n-1} 234 \beta^{n-1}\gamma \underline{434}32\\
= & \; 123 123 132 3311 \alpha^{n-1} 234 \beta^{n-1}\gamma 3433\underrightarrow{2}\\
\equiv & \; \underline{212}3 123 132 3311\alpha^{n-1} 234 \beta^{n-1}\gamma 3433\\
= & \; 1213 123 132 3311\alpha^{n-1} 234 \beta^{n-1}\gamma 3433.
\intertext{Also,}
\uwave{(1234)^{5(n+1)}} 
= & \; 2311 \alpha^{n} 234 \beta^{n}31123112\underrightarrow{3}\underleftarrow{11}43\\
= & \; 2311 \alpha^{n} 234 \beta^{n}3112311211\underline{343}\\
= & \; 2311 \alpha^{n} 234 \beta^{n}311231\underrightarrow{1211}\underleftarrow{4}34\\
= & \; 2311 \alpha^{n} 234 \beta^{n}3112314\underline{121}134\\
 = &  \; 2311 \alpha^{n} 234 \beta^{n}31123142\underline{121}34\\
 = &  \; 2311 \alpha^{n} 234 \beta^{n}3112314221234.\\
 \intertext{Therefore}
 (1234)^{5n+4} = & \; 2311 \alpha^{n} 234 \beta^{n}311231422.\\
\end{align*}

\end{proof}

\begin{figure}[h]
$$\begin{tikzpicture}[thick,rounded corners = 2mm, scale = .6]

\draw[Blue] (0,0) -- (0,3) -- (2,5) -- (2,6);
\draw[Orange] (1,0) -- (1,2) -- (3,4) -- (3,6);
\draw[Green] (2,0) -- (2,1) -- (4,3) -- (4,6);
\draw[Red] (3,0) -- (4,1) -- (4,2) -- (3.7,2.3);
\draw[Red] (3.3,2.7) -- (2.7,3.3);
\draw[Red] (2.3,3.7) -- (1.7,4.3);
\draw[Red] (1.3,4.7) -- (.7,5.3);
\draw[Red] (.3,5.7) -- (0,6); 
\draw[Purple] (4,0) -- (3.7,.3);
\draw[Purple] (3.3,.7) -- (2.7,1.3);
\draw[Purple] (2.3, 1.7) -- (1.7,2.3);
\draw[Purple] (1.3,2.7) -- (.7,3.3);
\draw[Purple] (.3,3.7) -- (0,4) -- (0,5) -- (1,6);

\begin{pgfonlayer}{background}
\fill[black!10!white] (2.5,1.5) -- (3.5,2.5) -- (4,2) -- (4,1) -- (3.5,.5) -- (2.5,1.5);
\fill[black!10!white]  (3.5,2.5) -- (3.3,2.3) -- (3.7,2.3);
\fill[black!10!white] (3.5,.5) -- (3.7,.7) -- (3.3,.7);
\end{pgfonlayer}

\begin{scope}[xshift = 7cm, thick, yscale = 6/7]

	\draw[Green] (2,0) -- (4,2) -- (4,7);
	
	\draw[Red] (3,0) -- (2.7,.3);
	\draw[Red] (2.3,.7) -- (2,1) -- (2,2) -- (3,3) -- (3,4) -- (2.7,4.3);
	\draw[Red] (2.3,4.7) -- (1.7,5.3);
	\draw[Red] (1.3,5.7) -- (.7,6.3);
	\draw[Red] (.3,6.7) -- (0,7);
	
	\draw[Purple] (4,0) -- (4,1) -- (3.7,1.3);
	\draw[Purple] (3.3,1.7) -- (2.7,2.3);
	\draw[Purple] (2.3,2.7) -- (1.7,3.3);
	\draw[Purple] (1.3,3.7) -- (.7,4.3);
	\draw[Purple] (.3,4.7) -- (0,5) -- (0,6) -- (1,7);
	
	\draw[Orange] (1,0) -- (1,3) -- (3,5) -- (3,7);
	
	\draw[Blue] (0,0) -- (0,4) -- (2,6) -- (2,7);

\end{scope}

\end{tikzpicture}$$
\caption{In $B_5$ the equality $12341234 = 12312343$ holds.}
\label{figure:5-braid1}
\end{figure}
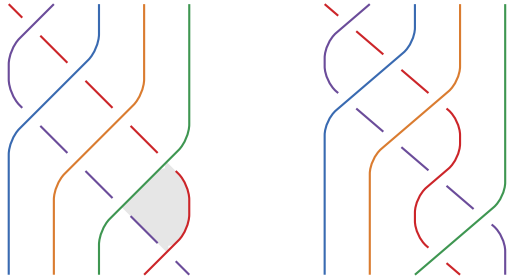

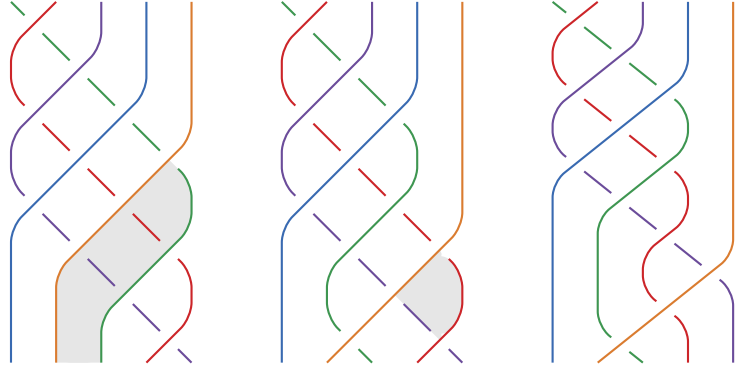
\begin{figure}[h]

$$\begin{tikzpicture}[thick, rounded corners = 2mm, scale = .6]

	\draw[Blue] (0,0) -- (0,3) -- (3,6) -- (3,8);
	
	\draw[Orange] (1,0) -- (1,2) -- (4,5) -- (4,8);
	
	\draw[Green] (2,0) -- (2,1) -- (4,3) -- (4,4) -- (3.7,4.3);
	\draw[Green] (3.3, 4.7) -- (2.7,5.3);
	\draw[Green] (2.3,5.7) -- (1.7,6.3);
	\draw[Green] (1.3,6.7) -- (.7,7.3);
	\draw[Green] (.3,7.7) -- (0,8);
	
	\draw[Red] (3,0) -- (4,1) -- (4,2) -- (3.7,2.3);
	\draw[Red] (3.3,2.7) -- (2.7,3.3);
	\draw[Red] (2.3,3.7) -- (1.7,4.3);
	\draw[Red] (1.3,4.7) -- (.7,5.3);
	\draw[Red] (.3,5.7) -- (0,6) -- (0,7) -- (1,8);
	
	\draw[Purple] (4,0) -- (3.7,.3);
	\draw[Purple] (3.3,.7) -- (2.7,1.3);
	\draw[Purple] (2.3,1.7) -- (1.7,2.3);
	\draw[Purple] (1.3,2.7) -- (.7,3.3);
	\draw[Purple] (.3,3.7) -- (0,4) -- (0,5) --(2,7) -- (2,8);
	
	\begin{pgfonlayer}{background}
	\fill[black!10!white] (3.5,4.5) -- (4,4) -- (4,3) -- (2,1) -- (2,0) -- (1,0) -- (1,2) -- (3.5,4.5);
	\fill[black!10!white] (1,0) -- (1,.5) -- (2,.5) -- (2,0);
	\end{pgfonlayer}
	
\begin{scope}[xshift = 6cm, thick]

	\draw[Blue] (0,0) -- (0,3) -- (3,6) -- (3,8);
	
	\draw[Orange] (1,0) -- (4,3) -- (4,8);
		
	\draw[Green] (2,0) -- (1.7,.3);
	\draw[Green] (1.3,.7) -- (1,1) -- (1,2) -- (3,4) --(3,5) -- (2.7,5.3);
	\draw[Green] (2.3,5.7) -- (1.7,6.3);
	\draw[Green] (1.3,6.7) -- (.7,7.3);
	\draw[Green] (.3,7.7) -- (0,8);
	
	\draw[Red] (3,0) -- (4,1) -- (4,2) -- (3.7,2.3);
	\draw[Red] (3.3,2.7) -- (2.7,3.3);
	\draw[Red] (2.3,3.7) -- (1.7,4.3);
	\draw[Red] (1.3,4.7) -- (.7,5.3);
	\draw[Red] (.3,5.7) -- (0,6) -- (0,7) -- (1,8);
	
	\draw[Purple] (4,0) -- (3.7,.3);
	\draw[Purple] (3.3,.7) -- (2.7,1.3);
	\draw[Purple] (2.3,1.7) -- (1.7,2.3);
	\draw[Purple] (1.3,2.7) -- (.7,3.3);
	\draw[Purple] (.3,3.7) -- (0,4) -- (0,5) --(2,7) -- (2,8);
	
	\begin{pgfonlayer}{background}
	\fill[black!10!white] (2.5,1.5) -- (3.5,.5) -- (4,1) -- (4,2) -- (3.5,2.5) -- (2.5,1.5);
	\fill[black!10!white] (3.5,.5) -- (3.7,.7) -- (3.3,.7);
	\fill[black!10!white] (3.5,2.5) -- (3.7,2.3) -- (3.3,2.3);
	\end{pgfonlayer}

\end{scope}

\begin{scope}[xshift = 12cm, thick, yscale = 8/10]

	\draw[Blue] (0,0) -- (0,5) -- (3,8) -- (3,10);
	
	\draw[Orange] (1,0) -- (4,3) -- (4,10);
		
	\draw[Green] (2,0) -- (1.7,.3);
	\draw[Green] (1.3,.7) -- (1,1) -- (1,4) -- (3,6) --(3,7) -- (2.7,7.3);
	\draw[Green] (2.3,7.7) -- (1.7,8.3);
	\draw[Green] (1.3,8.7) -- (.7,9.3);
	\draw[Green] (.3,9.7) -- (0,10);
	
	\draw[Red] (3,0) --(3,1) -- (2.7,1.3);
	\draw[Red] (2.3,1.7) -- (2,2) -- (2,3) -- (3,4) -- (3,5)--(2.7,5.3);
	\draw[Red] (2.3, 5.7) -- (1.7,6.3);
	\draw[Red] (1.3, 6.7) -- (.7,7.3); 
	\draw[Red] (.3,7.7) -- (0,8) -- (0,9) -- (1,10);

	\draw[Purple] (4,0) -- (4,2) -- (3.7,2.3);
	\draw[Purple] (3.3,2.7) -- (2.7,3.3);
	\draw[Purple] (2.3,3.7) -- (1.7,4.3);
	\draw[Purple] (1.3,4.7) -- (.7,5.3);
	\draw[Purple] (.3,5.7) -- (0,6) -- (0,7) --(2,9) -- (2,10);

\end{scope}

\end{tikzpicture}$$
\caption{In $B_5$ the equality $(1234)^3 = (123)^3 432$ holds.}
\label{figure:5-braid2}
\end{figure}

\begin{lemma}
\label{lemma:T6q}
Let $\zeta,\eta\in B_6$ be defined by
$$\zeta = 133545332334513~\text{and}~ \eta= 315531213.$$
For each positive integer $n$, the following equalities hold in the braid group $B_6$:
\begin{itemize}
\item $(12345)^{6n} = 2(\zeta323)^{n-1}\zeta 234 (\eta343)^{n-1}\eta 43$,
\item $(12345)^{6n+1} = 1323(\zeta323)^{n-1}\zeta 234 (\eta343)^{n-1}\eta3435$,
\end{itemize}
\end{lemma}
\begin{proof}
We show the first equality by induction on $n$. If $n=1$, then
Figure \ref{figure:6-twist} shows that $(12345)^6 = 2 \zeta 234 \eta 43$. We have
\begin{align*}
(12345)^{6(n+1)} = & \; \uwave{(12345)^{6n}} \; \uwave{(12345)^6}\\
= & \; 2(\zeta323)^{n-1}\zeta 23\underrightarrow{4 (\eta343)^{n-1}\eta 43} \underleftarrow{2 \zeta 234 \eta 43}\\
= & \; 2(\zeta323)^{n-1}\zeta \underline{232} \zeta 234 \eta \underline{434} (\eta343)^{n-1}\eta 43\\
= & \; 2\uwave{(\zeta323)^{n-1}\zeta 32 3} \zeta 234 \uwave{\eta 343 (\eta343)^{n-1}}\eta 43\\
= & \; 2(\zeta323)^n \zeta234 (\eta343)^n \eta 43,
\intertext{as desired. Also,}
(12345)^{6n+1} =& \;\underrightarrow{(12345)^{6n}} \underleftarrow{(123}45)\\
= &\; 123\uwave{(12345)^{6n}}45\\
= &  \; 1\underline{23 2}(\zeta323)^{n-1}\zeta 234 (\eta343)^{n-1}\eta \underline{43 4}5\\
= & \; 1323(\zeta323)^{n-1}\zeta 234 (\eta343)^{n-1}\eta 34 35.
\end{align*} 
\end{proof}

Let $D_{p,q}$ be the diagram of $T_{p,q}$ appearing on the right hand side in Lemmas \ref{lemma:T4q}, \ref{lemma:T5q}, and \ref{lemma:T6q}. Propositions \ref{proposition:TuraevT4q}, \ref{proposition:TuraevT5q}, and \ref{proposition:TuraevT6q} compute the genus of the Turaev surface for each such $D_{p,q}$.

\begin{proposition}
\label{proposition:TuraevT4q}
Let $D_{4,4n+i}$ be the closure of the braid diagrams appearing on the right hand side of Lemma \ref{lemma:T4q} for $i=0,1,2,$ and $3$. Then
$$g_T(D_{4,4n}) = g_T(D_{4,4n+1}) = 2n~\text{and}~g_T(D_{4,4n+2}) = g_T(D_{4,4n+3}) = 2n+1.$$
\end{proposition}
\begin{proof}
The diagram $D_{4,4n+i}$ has $12n+3i$ crossings. The number of components in the all-$A$ state of $D_{4,4n+i}$ is $s_A(D_{4,4n+i}) = 4$. Figure \ref{figure:B4states} shows the all-$B$ states of $D_{4,4n+i}$. From that figure, one can conclude
$$s_B(D_{4,4n}) = 8n-2,~ s_B(D_{4,4n+1}) = 8n+1, ~s_B(D_{4,4n+2}) = 8n+2,~\text{and}~s_B(D_{4,4n+3}) = 8n+5.$$
Equation \ref{equation:g_T(D)} then implies the result.
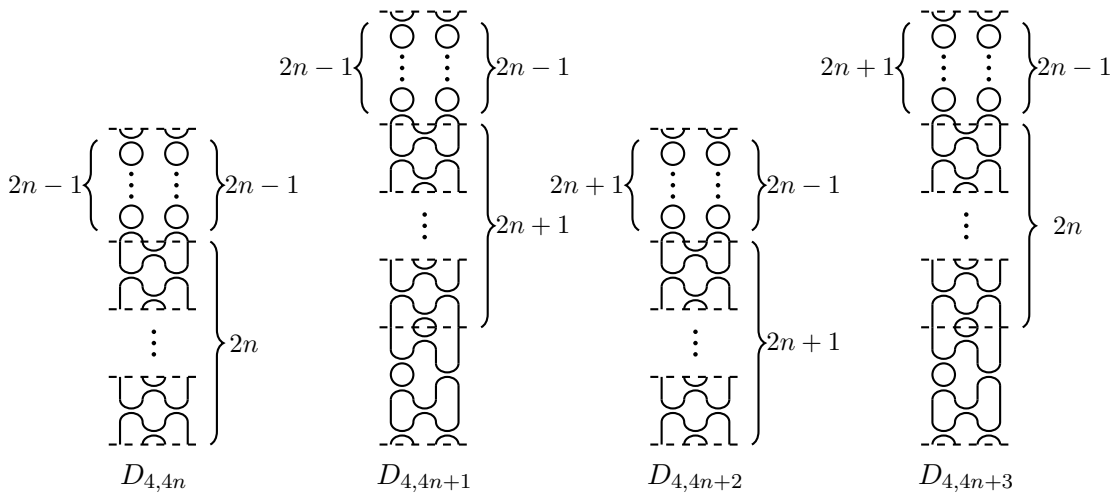
\begin{figure}[h]
$$\begin{tikzpicture}[thick, scale = .3]

\draw[dashed] (-.5,0) -- (3.5,0);

\draw (0,0) -- (0,1);
\draw (1,0) arc (180:0:.5cm and .4cm);
\draw (3,0) -- (3,1);
\draw (0,1) arc (180:0:.5cm and .4cm);
\draw (2,1) arc (180:0:.5cm and .4cm);
\draw (1,1) arc (180:360:.5cm and .4cm);

\begin{scope}[yscale = -1, yshift = -3cm]
	\draw[dashed] (-.5,0) -- (3.5,0);

	\draw (0,0) -- (0,1);
	\draw (1,0) arc (180:0:.5cm and .4cm);
	\draw (3,0) -- (3,1);
	\draw (0,1) arc (180:0:.5cm and .4cm);
	\draw (2,1) arc (180:0:.5cm and .4cm);
	\draw (1,1) arc (180:360:.5cm and .4cm);
\end{scope}

\fill (1.5,4.5) circle (.1cm);
\fill (1.5,4) circle (.1cm);
\fill (1.5,5) circle (.1cm);

\begin{scope}[yshift = 6cm]
\draw[dashed] (-.5,0) -- (3.5,0);

\draw (0,0) -- (0,1);
\draw (1,0) arc (180:0:.5cm and .4cm);
\draw (3,0) -- (3,1);
\draw (0,1) arc (180:0:.5cm and .4cm);
\draw (2,1) arc (180:0:.5cm and .4cm);
\draw (1,1) arc (180:360:.5cm and .4cm);

\begin{scope}[yscale = -1, yshift = -3cm]
	\draw[dashed] (-.5,0) -- (3.5,0);

	\draw (0,0) -- (0,1);
	\draw (1,0) arc (180:0:.5cm and .4cm);
	\draw (3,0) -- (3,1);
	\draw (0,1) arc (180:0:.5cm and .4cm);
	\draw (2,1) arc (180:0:.5cm and .4cm);
	\draw (1,1) arc (180:360:.5cm and .4cm);
\end{scope}
\end{scope}

\draw (0,9) arc (180:0:.5cm and .4cm);
\draw (2,9) arc (180:0:.5cm and .4cm);
\draw (.5,10.1) circle (.5cm);
\draw (2.5,10.1) circle (.5cm);

\fill (.5, 11) circle (.1cm); 
\fill (.5,11.5) circle (.1cm);
\fill (.5, 12) circle (.1cm);

\fill (2.5, 11) circle (.1cm); 
\fill (2.5,11.5) circle (.1cm);
\fill (2.5, 12) circle (.1cm);

\draw (.5, 12.9) circle (.5cm);
\draw (2.5,12.9) circle (.5cm);

\draw (0,14) arc (180:360:.5cm and .4cm);
\draw (2,14) arc (180:360:.5cm and .4cm);
\draw[dashed] (-.5,14) -- (3.5,14);

\draw[decoration={brace,amplitude=5pt},decorate] (-1,9.5) -- (-1,13.5);
\draw (-3.3,11.5) node{\small{$2n-1$}};

\draw[decoration={brace,amplitude=5pt, mirror}, decorate] (4,9.5) -- (4,13.5);
\draw (6.3,11.5) node{\small{$2n-1$}};

\draw[decoration={brace,amplitude=5pt}, decorate] (4,9) -- (4,0);
\draw (5.5,4.5) node{\small{$2n$}};

\draw (1.5,-1.5) node{$D_{4,4n}$};

\begin{scope}[xshift = 12cm]

	\draw[dashed] (-.5,0) -- (3.5,0);
	\draw (0,0) arc (180:0:.5cm and .4cm);
	\draw (2,0) arc (180:0:.5cm and .4cm);
	\draw (0,1) arc (180:360:.5cm and .4cm);
	\draw (2,1) arc (180:360:.5cm and .4cm);
	\draw (1,1) arc (180:0:.5cm and .4cm);
	\draw (0,1) -- (0,2);
	\draw (0,2) arc (180:0:.5cm and .4cm);
	\draw (1,2) arc (180:360:.5cm and .4cm);
	\draw (.5,3.1) circle (.5cm);
	\draw (0, 4.2) arc (180:360:.5cm and .4cm);
	\draw (1,4.2)  arc (180:0:.5cm and .4cm);
	\draw (2,4.2) -- (2,3.6);
	\draw (2,2) -- (2,2.6);
	\draw (2,2.6) arc (180:0:.5cm and .4cm);
	\draw (2,3.6) arc (180:360:.5cm and .4cm);
	\draw (3,2.6) -- (3,1);
	\draw (1,5.2) arc (180:360:.5cm and .4cm);
	\draw (0,4.2) -- (0,5.2);
	\draw (3,3.6) -- (3,5.2);
	

	\begin{scope}[yshift = 5.2 cm]
		\draw[dashed] (-.5,0) -- (3.5,0);

		\draw (0,0) -- (0,1);
		\draw (1,0) arc (180:0:.5cm and .4cm);
		\draw (3,0) -- (3,1);
		\draw (0,1) arc (180:0:.5cm and .4cm);
		\draw (2,1) arc (180:0:.5cm and .4cm);
		\draw (1,1) arc (180:360:.5cm and .4cm);

		\begin{scope}[yscale = -1, yshift = -3cm]
			\draw[dashed] (-.5,0) -- (3.5,0);

			\draw (0,0) -- (0,1);
			\draw (1,0) arc (180:0:.5cm and .4cm);
			\draw (3,0) -- (3,1);
			\draw (0,1) arc (180:0:.5cm and .4cm);				
			\draw (2,1) arc (180:0:.5cm and .4cm);
			\draw (1,1) arc (180:360:.5cm and .4cm);
		\end{scope}

		\fill (1.5,4.5) circle (.1cm);
		\fill (1.5,4) circle (.1cm);
		\fill (1.5,5) circle (.1cm);

		\begin{scope}[yshift = 6cm]
			\draw[dashed] (-.5,0) -- (3.5,0);

			\draw (0,0) -- (0,1);
			\draw (1,0) arc (180:0:.5cm and .4cm);
			\draw (3,0) -- (3,1);
			\draw (0,1) arc (180:0:.5cm and .4cm);
			\draw (2,1) arc (180:0:.5cm and .4cm);
			\draw (1,1) arc (180:360:.5cm and .4cm);

		\begin{scope}[yscale = -1, yshift = -3cm]
			\draw[dashed] (-.5,0) -- (3.5,0);

			\draw (0,0) -- (0,1);
			\draw (1,0) arc (180:0:.5cm and .4cm);
			\draw (3,0) -- (3,1);
			\draw (0,1) arc (180:0:.5cm and .4cm);
			\draw (2,1) arc (180:0:.5cm and .4cm);
			\draw (1,1) arc (180:360:.5cm and .4cm);
		\end{scope}
		\end{scope}

		\draw (0,9) arc (180:0:.5cm and .4cm);
		\draw (2,9) arc (180:0:.5cm and .4cm);
		\draw (.5,10.1) circle (.5cm);
		\draw (2.5,10.1) circle (.5cm);

		\fill (.5, 11) circle (.1cm); 
		\fill (.5,11.5) circle (.1cm);
		\fill (.5, 12) circle (.1cm);

		\fill (2.5, 11) circle (.1cm); 
		\fill (2.5,11.5) circle (.1cm);
		\fill (2.5, 12) circle (.1cm);

		\draw (.5, 12.9) circle (.5cm);
		\draw (2.5,12.9) circle (.5cm);

		\draw (0,14) arc (180:360:.5cm and .4cm);
		\draw (2,14) arc (180:360:.5cm and .4cm);
		\draw[dashed] (-.5,14) -- (3.5,14);

		\draw[decoration={brace,amplitude=5pt},decorate] (-1,9.5) -- (-1,13.5);
		\draw (-3.3,11.5) node{\small{$2n-1$}};

		\draw[decoration={brace,amplitude=5pt, mirror}, decorate] (4,9.5) -- (4,13.5);
		\draw (6.3,11.5) node{\small{$2n-1$}};

		\draw[decoration={brace,amplitude=5pt}, decorate] (4,9) -- (4,0);
		\draw (6.3,4.5) node{\small{$2n+1$}};

	\end{scope}
	
	\draw(1.5,-1.5) node{$D_{4,4n+1}$};

\end{scope}


\begin{scope}[xshift = 24cm]

\draw[dashed] (-.5,0) -- (3.5,0);

\draw (0,0) -- (0,1);
\draw (1,0) arc (180:0:.5cm and .4cm);
\draw (3,0) -- (3,1);
\draw (0,1) arc (180:0:.5cm and .4cm);
\draw (2,1) arc (180:0:.5cm and .4cm);
\draw (1,1) arc (180:360:.5cm and .4cm);

\begin{scope}[yscale = -1, yshift = -3cm]
	\draw[dashed] (-.5,0) -- (3.5,0);

	\draw (0,0) -- (0,1);
	\draw (1,0) arc (180:0:.5cm and .4cm);
	\draw (3,0) -- (3,1);
	\draw (0,1) arc (180:0:.5cm and .4cm);
	\draw (2,1) arc (180:0:.5cm and .4cm);
	\draw (1,1) arc (180:360:.5cm and .4cm);
\end{scope}

\fill (1.5,4.5) circle (.1cm);
\fill (1.5,4) circle (.1cm);
\fill (1.5,5) circle (.1cm);

\begin{scope}[yshift = 6cm]
\draw[dashed] (-.5,0) -- (3.5,0);

\draw (0,0) -- (0,1);
\draw (1,0) arc (180:0:.5cm and .4cm);
\draw (3,0) -- (3,1);
\draw (0,1) arc (180:0:.5cm and .4cm);
\draw (2,1) arc (180:0:.5cm and .4cm);
\draw (1,1) arc (180:360:.5cm and .4cm);

\begin{scope}[yscale = -1, yshift = -3cm]
	\draw[dashed] (-.5,0) -- (3.5,0);

	\draw (0,0) -- (0,1);
	\draw (1,0) arc (180:0:.5cm and .4cm);
	\draw (3,0) -- (3,1);
	\draw (0,1) arc (180:0:.5cm and .4cm);
	\draw (2,1) arc (180:0:.5cm and .4cm);
	\draw (1,1) arc (180:360:.5cm and .4cm);
\end{scope}
\end{scope}

\draw (0,9) arc (180:0:.5cm and .4cm);
\draw (2,9) arc (180:0:.5cm and .4cm);
\draw (.5,10.1) circle (.5cm);
\draw (2.5,10.1) circle (.5cm);

\fill (.5, 11) circle (.1cm); 
\fill (.5,11.5) circle (.1cm);
\fill (.5, 12) circle (.1cm);

\fill (2.5, 11) circle (.1cm); 
\fill (2.5,11.5) circle (.1cm);
\fill (2.5, 12) circle (.1cm);

\draw (.5, 12.9) circle (.5cm);
\draw (2.5,12.9) circle (.5cm);

\draw (0,14) arc (180:360:.5cm and .4cm);
\draw (2,14) arc (180:360:.5cm and .4cm);
\draw[dashed] (-.5,14) -- (3.5,14);

\draw[decoration={brace,amplitude=5pt},decorate] (-1,9.5) -- (-1,13.5);
\draw (-3.3,11.5) node{\small{$2n+1$}};

\draw[decoration={brace,amplitude=5pt, mirror}, decorate] (4,9.5) -- (4,13.5);
\draw (6.3,11.5) node{\small{$2n-1$}};

\draw[decoration={brace,amplitude=5pt}, decorate] (4,9) -- (4,0);
\draw (6.3,4.5) node{\small{$2n+1$}};

\draw (1.5,-1.5) node{$D_{4,4n+2}$};

\end{scope}
\begin{scope}[xshift = 36cm]

	\draw[dashed] (-.5,0) -- (3.5,0);
	\draw (0,0) arc (180:0:.5cm and .4cm);
	\draw (2,0) arc (180:0:.5cm and .4cm);
	\draw (0,1) arc (180:360:.5cm and .4cm);
	\draw (2,1) arc (180:360:.5cm and .4cm);
	\draw (1,1) arc (180:0:.5cm and .4cm);
	\draw (0,1) -- (0,2);
	\draw (0,2) arc (180:0:.5cm and .4cm);
	\draw (1,2) arc (180:360:.5cm and .4cm);
	\draw (.5,3.1) circle (.5cm);
	\draw (0, 4.2) arc (180:360:.5cm and .4cm);
	\draw (1,4.2)  arc (180:0:.5cm and .4cm);
	\draw (2,4.2) -- (2,3.6);
	\draw (2,2) -- (2,2.6);
	\draw (2,2.6) arc (180:0:.5cm and .4cm);
	\draw (2,3.6) arc (180:360:.5cm and .4cm);
	\draw (3,2.6) -- (3,1);
	\draw (1,5.2) arc (180:360:.5cm and .4cm);
	\draw (0,4.2) -- (0,5.2);
	\draw (3,3.6) -- (3,5.2);
	

	\begin{scope}[yshift = 5.2 cm]
		\draw[dashed] (-.5,0) -- (3.5,0);

		\draw (0,0) -- (0,1);
		\draw (1,0) arc (180:0:.5cm and .4cm);
		\draw (3,0) -- (3,1);
		\draw (0,1) arc (180:0:.5cm and .4cm);
		\draw (2,1) arc (180:0:.5cm and .4cm);
		\draw (1,1) arc (180:360:.5cm and .4cm);

		\begin{scope}[yscale = -1, yshift = -3cm]
			\draw[dashed] (-.5,0) -- (3.5,0);

			\draw (0,0) -- (0,1);
			\draw (1,0) arc (180:0:.5cm and .4cm);
			\draw (3,0) -- (3,1);
			\draw (0,1) arc (180:0:.5cm and .4cm);				
			\draw (2,1) arc (180:0:.5cm and .4cm);
			\draw (1,1) arc (180:360:.5cm and .4cm);
		\end{scope}

		\fill (1.5,4.5) circle (.1cm);
		\fill (1.5,4) circle (.1cm);
		\fill (1.5,5) circle (.1cm);

		\begin{scope}[yshift = 6cm]
			\draw[dashed] (-.5,0) -- (3.5,0);

			\draw (0,0) -- (0,1);
			\draw (1,0) arc (180:0:.5cm and .4cm);
			\draw (3,0) -- (3,1);
			\draw (0,1) arc (180:0:.5cm and .4cm);
			\draw (2,1) arc (180:0:.5cm and .4cm);
			\draw (1,1) arc (180:360:.5cm and .4cm);

		\begin{scope}[yscale = -1, yshift = -3cm]
			\draw[dashed] (-.5,0) -- (3.5,0);

			\draw (0,0) -- (0,1);
			\draw (1,0) arc (180:0:.5cm and .4cm);
			\draw (3,0) -- (3,1);
			\draw (0,1) arc (180:0:.5cm and .4cm);
			\draw (2,1) arc (180:0:.5cm and .4cm);
			\draw (1,1) arc (180:360:.5cm and .4cm);
		\end{scope}
		\end{scope}

		\draw (0,9) arc (180:0:.5cm and .4cm);
		\draw (2,9) arc (180:0:.5cm and .4cm);
		\draw (.5,10.1) circle (.5cm);
		\draw (2.5,10.1) circle (.5cm);

		\fill (.5, 11) circle (.1cm); 
		\fill (.5,11.5) circle (.1cm);
		\fill (.5, 12) circle (.1cm);

		\fill (2.5, 11) circle (.1cm); 
		\fill (2.5,11.5) circle (.1cm);
		\fill (2.5, 12) circle (.1cm);

		\draw (.5, 12.9) circle (.5cm);
		\draw (2.5,12.9) circle (.5cm);

		\draw (0,14) arc (180:360:.5cm and .4cm);
		\draw (2,14) arc (180:360:.5cm and .4cm);
		\draw[dashed] (-.5,14) -- (3.5,14);

		\draw[decoration={brace,amplitude=5pt},decorate] (-1,9.5) -- (-1,13.5);
		\draw (-3.3,11.5) node{\small{$2n+1$}};

		\draw[decoration={brace,amplitude=5pt, mirror}, decorate] (4,9.5) -- (4,13.5);
		\draw (6.3,11.5) node{\small{$2n-1$}};

		\draw[decoration={brace,amplitude=5pt}, decorate] (4,9) -- (4,0);
		\draw (6,4.5) node{\small{$2n$}};

	\end{scope}
	
	\draw (1.5,-1.5) node{$D_{4,4n+3}$};

\end{scope}

\end{tikzpicture}$$
\caption{The all-$B$ states of $D_{4,4n+i}$ for $i=0, 1,2$, and $3$.}
\label{figure:B4states}
\end{figure}

\end{proof}

\begin{proposition}
\label{proposition:TuraevT5q}
Let $D_{5,5n+i}$ be the closure of the braid diagrams appearing on the right hand side of Lemma \ref{lemma:T5q} for $i=0,1,2,3,$ and $4$. Then
$$g_T(D_{5,5n}) = g_T(D_{5,5n+1}) = 4n~\text{and}~g_T(D_{5,5n+j}) = 4n+ j -1$$
for $j=2, 3$, and $4$.
\end{proposition}
\begin{proof}
The diagram $D_{5,5n+i}$ has $20n+4i$ crossings. The number of components in the all-$A$ state of $D_{5,5n+i}$ is $s_A(D_{5,5n+i}) = 5$. Since the expressions for $D_{5,5n+i}$ in Lemma \ref{lemma:T5q} contain $\alpha^{n-1}$ and $\beta^{n-1}$ for $i=0,1,2,$ and $3$, we must handle the case where $n=1$ separately from the case where $n>1$. Figure \ref{figure:B5states1} shows the all-$B$ states of $D_{5,5}$, $D_{5,6}$, $D_{5,7}$, and $D_{5,8}$. From that figure one can see that
$$s_B(D_{5,5}) = 9,~s_B(D_{5,6})=13, ~s_B(D_{5,7}) = 15,~\text{and}~s_B(D_{5,8}) = 17.$$
Equation \ref{equation:g_T(D)} then implies that
$$g_T(D_{5,5}) = g_T(D_{5,4})= 4,~g_T(D_{5,7}) = 5,~\text{and}~g_T(D_{5,8}) = 6.$$

\begin{figure}[h]
$$\begin{tikzpicture}[thick,scale = .3]

\draw (2,-1.5) node{$D_{5,5}$};
\draw[dashed] (-.5, 0) -- (4.5,0);

\draw (0,0) arc (180:0:.5cm and .4cm);
\draw (2,0) arc (180:0:.5cm and .4cm);
\draw (2,1) arc (180:360:.5cm and .4cm);
\draw (3,1) arc (180:0:.5cm and .4cm);
\draw (4,0) -- (4,1);
\draw (3,2) arc (180:360:.5cm and .4cm);
\draw (2,2) arc (180:0:.5cm and .4cm);
\draw (2,2) -- (2,1);
\draw (.5,1.5) circle (.5cm);
\draw (0,3) arc (180:360:.5cm and .4cm);
\draw (1,3)  arc (180:0:.5cm and .4cm);
\draw (2,3) arc (180:360:.5cm and .4cm);
\draw (0,3) -- (0,4);
\draw (0,4) arc (180:0:.5cm and .4cm);
\draw (0.5,5.2) circle (.5cm);
\draw (1,4) arc (180:360:.5cm and .4cm);
\draw (2,5.3) -- (2,4);
\draw (3,5.3) -- (3,3);
\draw (2,5.3) arc (180:0:.5cm and .4cm);
\draw (0,6.3)  arc (180:360:.5cm and .4cm);
\draw (1,6.3) arc (180:0:.5cm and .4cm);
\draw (2,6.3) arc (180:360:.5cm and .4cm);
\draw (1,7.3) arc (180:360:.5cm and .4cm);
\draw (0,6.3) -- (0,7.3);
\draw (3,6.3) -- (3,7.3);
\draw (4,7.3) -- (4,2);
\draw[dashed] (-.5,7.3) -- (4.5,7.3);
\draw (-2,4.2) node{\small{$23112$}};
\draw (-2,3.1) node{\small{$34311$}};

\begin{scope}[yshift = 7.3cm]
\draw (0,0) arc (180:0:.5cm and .4cm);
\draw (2,0) arc (180:0:.5cm and .4cm);
\draw (2,1) arc (180:360:.5cm and .4cm);
\draw (3,1) arc (180:0:.5cm and .4cm);
\draw (4,0) -- (4,1);
\draw (3,2) arc (180:360:.5cm and .4cm);
\draw (2,2) arc (180:0:.5cm and .4cm);
\draw (2,2) -- (2,1);
\draw (.5,1.5) circle (.5cm);
\draw (0,3) arc (180:360:.5cm and .4cm);
\draw (1,3)  arc (180:0:.5cm and .4cm);
\draw (2,3) arc (180:360:.5cm and .4cm);
\draw (0,3) -- (0,4);
\draw (0,4) arc (180:0:.5cm and .4cm);
\draw (0.5,5.2) circle (.5cm);
\draw (1,4) arc (180:360:.5cm and .4cm);
\draw (2,5.3) -- (2,4);
\draw (3,5.3) -- (3,3);
\draw (2,5.3) arc (180:0:.5cm and .4cm);
\draw (0,6.3)  arc (180:360:.5cm and .4cm);
\draw (1,6.3) arc (180:0:.5cm and .4cm);
\draw (2,6.3) arc (180:360:.5cm and .4cm);
\draw (1,7.3) arc (180:360:.5cm and .4cm);
\draw (0,6.3) -- (0,7.3);
\draw (3,6.3) -- (3,7.3);
\draw (4,7.3) -- (4,2);
\draw[dashed] (-.5,7.3) -- (4.5,7.3);
\draw (-2,4.2) node{\small{$23112$}};
\draw (-2,3.1) node{\small{$34311$}};

\end{scope}


\begin{scope}[xshift = 10cm]

\draw[dashed] (-.5,0) -- (4.5,0);

\draw (2,0) arc (180:0:.5cm and .4cm);
\draw (2,1) arc (180:360:.5cm and .4cm);
\draw (3,1) arc (180:0:.5cm and .4cm);
\draw (4,1) -- (4,0);
\draw (2,1) -- (2,2);
\draw (2,2) arc (180:0:.5cm and .4cm);
\draw (3,2) arc (180:360:.5cm and .4cm);
\draw (4,2) -- (4,3);
\draw (2,3) arc (180:360:.5cm and .4cm);
\draw (0,0) -- (0,3);
\draw (1,0) -- (1,3);
\draw (-1.8,1.5) node {\small{$343$}};

\draw[dashed] (-.5,3) -- (4.5,3);
\begin{scope}[yshift=3cm]
	\draw (0,0)  arc (180:0:.5cm and .4cm);
	\draw (.5,1) ellipse (.5cm and .4cm);
	\draw (2,0) -- (2,.5);
	\draw (3,0) -- (3,.5);
	\draw (2,.5) arc (180:0:.5cm and .4cm);
	\draw (2,1.5) arc (180:360:.5cm and .4cm);
	\draw (2,1.5) -- (2,2);
	\draw (1,2) arc (180:0:.5cm and .4cm);
	\draw (0,2) arc (180:360:.5cm and .4cm);
	\draw (0,2) -- (0,3);
	\draw (0,3)  arc (180:0:.5cm and .4cm);
	\draw (1,3) arc (180:360:.5cm and .4cm);
	\draw (2,3) -- (2,3.5);
	\draw (3,3.5) -- (3,1.5);
	\draw (2,3.5) arc  (180:0:.5cm and .4cm);
	
	\begin{scope}[yshift = 3cm]
		\draw (.5,1) ellipse (.5cm and .4cm);
		\draw (2,1.5) arc (180:360:.5cm and .4cm);
		\draw (2,1.5) -- (2,2);
		\draw (1,2) arc (180:0:.5cm and .4cm);
		\draw (0,2) arc (180:360:.5cm and .4cm);
		\draw (0,2) arc (180:360:.5cm and .4cm);
		\draw (0,2) -- (0,3);
		\draw (0,3)  arc (180:0:.5cm and .4cm);
		\draw (1,3) arc (180:360:.5cm and .4cm);
		\draw (2,3) -- (2,3.5);
		\draw (3,3.5) -- (3,1.5);
		\draw (2,3.5) arc  (180:0:.5cm and .4cm);
	\end{scope}
	
	\draw (.5,7) ellipse (.5cm and .4cm);
	
	\draw (0,8) arc (180:360:.5cm and .4cm);
	\draw (2,7.5)  arc (180:360:.5cm and .4cm);
	\draw (2,8) -- (2,7.5);
	\draw (3,8) -- (3,7.5);
	\draw (4,0) -- (4,8);
	
	\draw[dashed] (-.5,8) -- (4.5,8);
	\draw (-1.5,4) node{\small{$\gamma$}};
\end{scope}

\draw (3,11) arc (180:0:.5cm and .4cm);
\draw (2,11) -- (2,12);
\draw (2,12) arc (180:0:.5cm and .4cm);
\draw (3,12)   arc (180:360:.5cm and .4cm);
\draw (2,13) arc (180:360:.5cm and .4cm);
\draw (1,13) arc (180:0:.5cm and .4cm);
\draw (1,13) -- (1,11);
\draw (3,13) -- (3,14);
\draw (2,14) arc (180:0:.5cm and .4cm);
\draw (1,14) arc (180:360:.5cm and .4cm);
\draw (0,14) arc (180:0:.5cm and .4cm);
\draw (0,11) -- (0,14);
\draw (.5,15) ellipse (.5cm and .4cm);
\draw (2.5,15) ellipse (.5cm and .4cm);

\draw (0,16) arc (180:360:.5cm and .4cm);
\draw (1,16) arc (180:0:.5cm and .4cm);
\draw (2,16) arc (180:360:.5cm and .4cm);
\draw (0,16) -- (0,17);
\draw (3,16) -- (3,17);
\draw (0,17) arc (180:0:.5cm and .4cm);
\draw (1,17)  arc (180:360:.5cm and .4cm);
\draw (2,17) arc (180:0:.5cm and .4cm);

\draw (0,18) arc (180:360:.5cm and .4cm);
\draw (2,18) arc (180:360:.5cm and .4cm);
\draw (4,18) -- (4,12);

\draw[dashed] (-.5,18) -- (4.5,18);
\draw (-2,15) node{\small{$13233$}};
\draw (-2,14) node{\small{$11234$}};

\draw (2,-1.5) node{$D_{5,6}$};

\end{scope}

\begin{scope}[xshift = 20cm]

\draw[dashed] (-.5,0) -- (4.5,0);
\draw (-2,2) node{\small{$3433$}};
\draw (2,-1.5) node {$D_{5,7}$};
\draw (2,0)  arc (180:0:.5cm and .4cm);
\draw (2.5,1) ellipse (.5cm and .4cm);
\draw (3,2)  arc (180:0:.5cm and .4cm);
\draw (4,0) -- (4,2);
\draw (2,2) arc (180:360:.5cm and .4cm);
\draw (2,2) -- (2,3);
\draw (2,3) arc (180:0:.5cm and .4cm);
\draw (3,3) arc (180:360:.5cm and .4cm);
\draw (4,3) -- (4,4);
\draw (2,4) arc (180:360:.5cm and .4cm);
\draw (0,0) -- (0,4);
\draw (1,0) -- (1,4);

\draw[dashed] (-.5,4) -- (4.5,4);
\begin{scope}[yshift = 4cm]
\draw (0,0)  arc (180:0:.5cm and .4cm);
	\draw (.5,1) ellipse (.5cm and .4cm);
	\draw (2,0) -- (2,.5);
	\draw (3,0) -- (3,.5);
	\draw (2,.5) arc (180:0:.5cm and .4cm);
	\draw (2,1.5) arc (180:360:.5cm and .4cm);
	\draw (2,1.5) -- (2,2);
	\draw (1,2) arc (180:0:.5cm and .4cm);
	\draw (0,2) arc (180:360:.5cm and .4cm);
	\draw (0,2) -- (0,3);
	\draw (0,3)  arc (180:0:.5cm and .4cm);
	\draw (1,3) arc (180:360:.5cm and .4cm);
	\draw (2,3) -- (2,3.5);
	\draw (3,3.5) -- (3,1.5);
	\draw (2,3.5) arc  (180:0:.5cm and .4cm);
	
	\begin{scope}[yshift = 3cm]
		\draw (.5,1) ellipse (.5cm and .4cm);
		\draw (2,1.5) arc (180:360:.5cm and .4cm);
		\draw (2,1.5) -- (2,2);
		\draw (1,2) arc (180:0:.5cm and .4cm);
		\draw (0,2) arc (180:360:.5cm and .4cm);
		\draw (0,2) arc (180:360:.5cm and .4cm);
		\draw (0,2) -- (0,3);
		\draw (0,3)  arc (180:0:.5cm and .4cm);
		\draw (1,3) arc (180:360:.5cm and .4cm);
		\draw (2,3) -- (2,3.5);
		\draw (3,3.5) -- (3,1.5);
		\draw (2,3.5) arc  (180:0:.5cm and .4cm);
	\end{scope}
	
	\draw (.5,7) ellipse (.5cm and .4cm);
	
	\draw (0,8) arc (180:360:.5cm and .4cm);
	\draw (2,7.5)  arc (180:360:.5cm and .4cm);
	\draw (2,8) -- (2,7.5);
	\draw (3,8) -- (3,7.5);
	\draw (4,0) -- (4,8);
	
	\draw[dashed] (-.5,8) -- (4.5,8);
	\draw (-1.5,4) node{\small{$\gamma$}};

\end{scope}

\begin{scope}[yshift=1cm]
\draw (3,11) arc (180:0:.5cm and .4cm);
\draw (2,11) -- (2,12);
\draw (2,12) arc (180:0:.5cm and .4cm);
\draw (3,12)   arc (180:360:.5cm and .4cm);
\draw (2,13) arc (180:360:.5cm and .4cm);
\draw (1,13) arc (180:0:.5cm and .4cm);
\draw (1,13) -- (1,11);
\draw (3,13) -- (3,14);
\draw (2,14) arc (180:0:.5cm and .4cm);
\draw (1,14) arc (180:360:.5cm and .4cm);
\draw (0,14) arc (180:0:.5cm and .4cm);
\draw (0,11) -- (0,14);
\draw (.5,15) ellipse (.5cm and .4cm);
\draw (2.5,15) ellipse (.5cm and .4cm);

\end{scope}

\draw (0,17) arc (180:360:.5cm and .4cm);
\draw (1,17) arc (180:0:.5cm and .4cm);
\draw (2,17) arc (180:360:.5cm and .4cm);
\draw (3,17) -- (3,18);
\draw (2,18) arc (180:0:.5cm and .4cm);
\draw (1,18) arc (180:360:.5cm and .4cm);
\draw (0,17) -- (0,18.5);
\draw (0,18.5) arc (180:0:.5cm and .4cm);
\draw (1,18.5) -- (1,18);
\draw (2.5,19) ellipse (.5cm and .4cm);
\draw (0,19.5) arc (180:360:.5cm and .4cm);
\draw (0,19.5) -- (0,21);
\draw (1,19.5) -- (1,20);
\draw (0,21)  arc (180:0:.5cm and .4cm);
\draw (1,21) arc (180:360:.5cm and .4cm);
\draw (1,20) arc (180:0:.5cm and .4cm);
\draw (2,20) arc (180:360:.5cm and .4cm);
\draw (0,22) arc (180:360:.5cm and .4cm);
\draw (2,22) -- (2,21);
\draw (3,22) -- (3,20);
\draw (4,22) -- (4,13);

\draw[dashed] (-.5,22) -- (4.5,22);

\draw (-2.2,17.5) node{\small{$1231323$}};
\draw (-2.2,16.5) node{\small{$311234$}};

\end{scope}
\begin{scope}[xshift = 30 cm]
\draw[dashed] (-.5,0) -- (4.5,0);
\draw (-2,2) node{\small{$3433$}};
\draw (2,-1.5) node {$D_{5,8}$};
\draw (2,0)  arc (180:0:.5cm and .4cm);
\draw (2.5,1) ellipse (.5cm and .4cm);
\draw (3,2)  arc (180:0:.5cm and .4cm);
\draw (4,0) -- (4,2);
\draw (2,2) arc (180:360:.5cm and .4cm);
\draw (2,2) -- (2,3);
\draw (2,3) arc (180:0:.5cm and .4cm);
\draw (3,3) arc (180:360:.5cm and .4cm);
\draw (4,3) -- (4,4);
\draw (2,4) arc (180:360:.5cm and .4cm);
\draw (0,0) -- (0,4);
\draw (1,0) -- (1,4);

\draw[dashed] (-.5,4) -- (4.5,4);
\begin{scope}[yshift = 4cm]
\draw (0,0)  arc (180:0:.5cm and .4cm);
	\draw (.5,1) ellipse (.5cm and .4cm);
	\draw (2,0) -- (2,.5);
	\draw (3,0) -- (3,.5);
	\draw (2,.5) arc (180:0:.5cm and .4cm);
	\draw (2,1.5) arc (180:360:.5cm and .4cm);
	\draw (2,1.5) -- (2,2);
	\draw (1,2) arc (180:0:.5cm and .4cm);
	\draw (0,2) arc (180:360:.5cm and .4cm);
	\draw (0,2) -- (0,3);
	\draw (0,3)  arc (180:0:.5cm and .4cm);
	\draw (1,3) arc (180:360:.5cm and .4cm);
	\draw (2,3) -- (2,3.5);
	\draw (3,3.5) -- (3,1.5);
	\draw (2,3.5) arc  (180:0:.5cm and .4cm);
	
	\begin{scope}[yshift = 3cm]
		\draw (.5,1) ellipse (.5cm and .4cm);
		\draw (2,1.5) arc (180:360:.5cm and .4cm);
		\draw (2,1.5) -- (2,2);
		\draw (1,2) arc (180:0:.5cm and .4cm);
		\draw (0,2) arc (180:360:.5cm and .4cm);
		\draw (0,2) arc (180:360:.5cm and .4cm);
		\draw (0,2) -- (0,3);
		\draw (0,3)  arc (180:0:.5cm and .4cm);
		\draw (1,3) arc (180:360:.5cm and .4cm);
		\draw (2,3) -- (2,3.5);
		\draw (3,3.5) -- (3,1.5);
		\draw (2,3.5) arc  (180:0:.5cm and .4cm);
	\end{scope}
	
	\draw (.5,7) ellipse (.5cm and .4cm);
	
	\draw (0,8) arc (180:360:.5cm and .4cm);
	\draw (2,7.5)  arc (180:360:.5cm and .4cm);
	\draw (2,8) -- (2,7.5);
	\draw (3,8) -- (3,7.5);
	\draw (4,0) -- (4,8);
	
	\draw[dashed] (-.5,8) -- (4.5,8);
	\draw (-1.5,4) node{\small{$\gamma$}};

\end{scope}

\begin{scope}[yshift=1cm]
\draw (3,11) arc (180:0:.5cm and .4cm);
\draw (2,11) -- (2,12);
\draw (2,12) arc (180:0:.5cm and .4cm);
\draw (3,12)   arc (180:360:.5cm and .4cm);
\draw (2,13) arc (180:360:.5cm and .4cm);
\draw (1,13) arc (180:0:.5cm and .4cm);
\draw (1,13) -- (1,11);
\draw (3,13) -- (3,14);
\draw (2,14) arc (180:0:.5cm and .4cm);
\draw (1,14) arc (180:360:.5cm and .4cm);
\draw (0,14) arc (180:0:.5cm and .4cm);
\draw (0,11) -- (0,14);
\draw (.5,15) ellipse (.5cm and .4cm);
\draw (2.5,15) ellipse (.5cm and .4cm);

\end{scope}

\draw (0,17) arc (180:360:.5cm and .4cm);
\draw (1,17) arc (180:0:.5cm and .4cm);
\draw (2,17) arc (180:360:.5cm and .4cm);
\draw (3,17) -- (3,18);
\draw (2,18) arc (180:0:.5cm and .4cm);
\draw (1,18) arc (180:360:.5cm and .4cm);
\draw (0,17) -- (0,18.5);
\draw (0,18.5) arc (180:0:.5cm and .4cm);
\draw (1,18.5) -- (1,18);
\draw (2.5,19) ellipse (.5cm and .4cm);
\draw (0,19.5) arc (180:360:.5cm and .4cm);
\draw (0,19.5) -- (0,21);
\draw (1,19.5) -- (1,20);
\draw (0,21)  arc (180:0:.5cm and .4cm);
\draw (1,21) arc (180:360:.5cm and .4cm);
\draw (1,20) arc (180:0:.5cm and .4cm);

\draw(.5,22) ellipse (.5cm and .4cm);
\draw (2,20) arc (180:360:.5cm and .4cm);
\draw (3,20) -- (3,21.5);
\draw (2,21) -- (2,21.5);
\draw (2,21.5) arc (180:0:.5cm and .4cm);
\draw (2,22.5) arc (180:360:.5cm and .4cm);
\draw (2,22.5) -- (2,23);
\draw (3,22.5) --(3,23);
\draw (1,23) arc (180:0:.5cm and .4cm);
\draw (0,23)  arc (180:360:.5cm and .4cm);
\draw (0,23) -- (0,24);
\draw (0,24)  arc (180:0:.5cm and .4cm);
\draw (1,24) arc (180:360:.5cm and .4cm);
\draw (2,24) -- (2,25);
\draw (3,25) -- (3,22.5);
\draw (0,25) arc (180:360:.5cm and .4cm);
\draw (4,25) -- (4,13);

\draw[dashed] (-.5,25) -- (4.5,25);

\draw (-2,19.5) node{\small{$121312$}};
\draw (-2,18.5) node{\small{$313233$}};
\draw (-2,17.5) node{\small{$11234$}};

\end{scope}

\end{tikzpicture}$$
\caption{The all-$B$ states of $D_{5,5+i}$ for $i=0, 1,2$, and $3$.}
\label{figure:B5states1}
\end{figure}
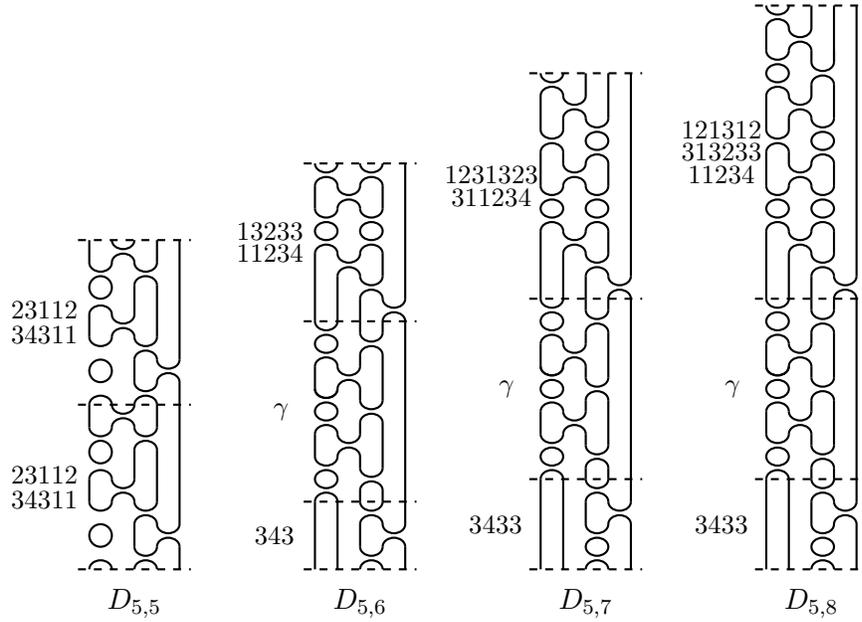

If $q\ge 9$, then $D_{5,q}$ contains at least one $\alpha$ and $\beta$ as a sub-word in its braid word. Figure \ref{figure:stack} shows that consecutive $\alpha$ words add $4$ components to the all-$B$ state and consecutive $\beta$ words add $8$ components to the all-$B$ state. This observation allows us to compute the number of components in the all-$B$ state of $D_{5,q}$ by only considering braid words with $\alpha$ and $\beta$ appearing once. A braid word where $\alpha^{\ell}$ and $\beta^m$ are replaced by $\alpha$ and $\beta$ respectively will be called {\em reduced}.

\begin{figure}[h]
$$\begin{tikzpicture}[thick,scale=.3]

\draw[dashed] (-.5,0) -- (4.5,0);
\draw (2,-1.5) node{$\alpha^2$};

\draw (0,0) arc (180:0:.5cm and .4cm);
\draw (2,0) arc (180:0:.5cm and .4cm);
\draw (.5,1) ellipse (.5cm and .4cm);
\draw (2.5,1) ellipse (.5cm and .4cm);

\draw (0,2)  arc (180:360:.5cm and .4cm);
\draw (1,2)  arc (180:0:.5cm and .4cm);
\draw (2,2) arc (180:360:.5cm and .4cm);
\draw (3,2) -- (3,3);
\draw (2,3) arc (180:0:.5cm and .4cm);
\draw (1,3) arc (180:360:.5cm and .4cm);
\draw (0,2) -- (0,4);
\draw (1,3) -- (1,4);
\draw (2,4) arc (180:360:.5cm and .4cm);
\draw (4,0) -- (4,4);
\draw (-1.5,2) node{\small{$\alpha$}};

\begin{scope}[yshift = 4cm]
\draw[dashed] (-.5,0) -- (4.5,0);

\draw (0,0) arc (180:0:.5cm and .4cm);
\draw (2,0) arc (180:0:.5cm and .4cm);
\draw (.5,1) ellipse (.5cm and .4cm);
\draw (2.5,1) ellipse (.5cm and .4cm);

\draw (0,2)  arc (180:360:.5cm and .4cm);
\draw (1,2)  arc (180:0:.5cm and .4cm);
\draw (2,2) arc (180:360:.5cm and .4cm);
\draw (3,2) -- (3,3);
\draw (2,3) arc (180:0:.5cm and .4cm);
\draw (1,3) arc (180:360:.5cm and .4cm);
\draw (0,2) -- (0,4);
\draw (1,3) -- (1,4);
\draw (2,4) arc (180:360:.5cm and .4cm);
\draw (4,0) -- (4,4);
\draw (-1.5,2) node{\small{$\alpha$}};

\draw[dashed] (-.5,4) -- (4.5,4);
\end{scope}

\begin{scope}[xshift = 10cm]

\draw[dashed] (-.5,0) -- (4.5,0);
\draw (2,-1.5) node {$\beta^2$};
\draw (2,0) arc (180:0:.5cm and .4cm);
\draw (2,1) arc (180:360:.5cm and .4cm);
\draw (3,1) arc (180:0:.5cm and .4cm);
\draw (4,1) -- (4,0);

\draw (2,1) -- (2,2);
\draw (.5,1.5) ellipse (.5cm and .4cm);

\draw (0,0) -- (0,.5);
\draw (1,0) -- (1,.5);
\draw (0,.5) arc (180:0:.5cm and .4cm);

\draw (0, 2.5) arc (180:360:.5cm and .4cm);
\draw (2,2) arc (180:0:.5cm and .4cm);
\draw (3,2) arc (180:360:.5cm and .4cm);
\draw (2.5,3) ellipse (.5cm and .4cm);

\draw (2,4)  arc (180:360:.5cm and .4cm);
\draw (1,4) arc (180:0:.5cm and .4cm);
\draw (1,4) -- (1,2.5);
\draw (0,2.5) -- (0,5);
\draw (0,5) arc (180:0:.5cm and .4cm);
\draw (1,5) arc (180:360:.5cm and .4cm);
\draw (2,5) -- (2,5.5);
\draw (2,5.5) arc (180:0:.5cm and .4cm);
\draw (3,5.5) -- (3,4);

\draw (0.5, 6) ellipse (.5cm and .4cm);

\draw (2,6.5) arc (180:360:.5cm and .4cm);
\draw (2,6.5) -- (2,7);
\draw (3,6.5) -- (3,8.5);
\draw (0,7) arc (180:360:.5cm and .4cm);
\draw (1,7) arc (180:0:.5cm and .4cm);
\draw (0,7) -- (0,8);
\draw (0,8) arc (180:0:.5cm and .4cm);
\draw (1,8) arc (180:360:.5cm and .4cm);
\draw (.5,9) ellipse (.5cm and .4cm);
\draw (0,10) arc (180:360:.5cm and .4cm);
\draw (2,10) -- (2,9.5);
\draw (3,10) -- (3,9.5);
\draw (2,9.5) arc (180:360:.5cm and .4cm);
\draw (2,8) -- (2,8.5);
\draw (2,8.5) arc (180:0:.5cm and .4cm);
\draw (4,2) -- (4,10);
\draw (-1.5,5) node{\small{$\beta$}};

\begin{scope}[yshift = 10 cm]
\draw[dashed] (-.5,0) -- (4.5,0);
\draw (2,0) arc (180:0:.5cm and .4cm);
\draw (2,1) arc (180:360:.5cm and .4cm);
\draw (3,1) arc (180:0:.5cm and .4cm);
\draw (4,1) -- (4,0);

\draw (2,1) -- (2,2);
\draw (.5,1.5) ellipse (.5cm and .4cm);

\draw (0,0) -- (0,.5);
\draw (1,0) -- (1,.5);
\draw (0,.5) arc (180:0:.5cm and .4cm);

\draw (0, 2.5) arc (180:360:.5cm and .4cm);
\draw (2,2) arc (180:0:.5cm and .4cm);
\draw (3,2) arc (180:360:.5cm and .4cm);
\draw (2.5,3) ellipse (.5cm and .4cm);

\draw (2,4)  arc (180:360:.5cm and .4cm);
\draw (1,4) arc (180:0:.5cm and .4cm);
\draw (1,4) -- (1,2.5);
\draw (0,2.5) -- (0,5);
\draw (0,5) arc (180:0:.5cm and .4cm);
\draw (1,5) arc (180:360:.5cm and .4cm);
\draw (2,5) -- (2,5.5);
\draw (2,5.5) arc (180:0:.5cm and .4cm);
\draw (3,5.5) -- (3,4);

\draw (0.5, 6) ellipse (.5cm and .4cm);

\draw (2,6.5) arc (180:360:.5cm and .4cm);
\draw (2,6.5) -- (2,7);
\draw (3,6.5) -- (3,8.5);
\draw (0,7) arc (180:360:.5cm and .4cm);
\draw (1,7) arc (180:0:.5cm and .4cm);
\draw (0,7) -- (0,8);
\draw (0,8) arc (180:0:.5cm and .4cm);
\draw (1,8) arc (180:360:.5cm and .4cm);
\draw (.5,9) ellipse (.5cm and .4cm);
\draw (0,10) arc (180:360:.5cm and .4cm);
\draw (2,10) -- (2,9.5);
\draw (3,10) -- (3,9.5);
\draw (2,9.5) arc (180:360:.5cm and .4cm);
\draw (2,8) -- (2,8.5);
\draw (2,8.5) arc (180:0:.5cm and .4cm);
\draw (4,2) -- (4,10);
\draw (-1.5,5) node{\small{$\beta$}};

\draw[dashed] (-.5,10) -- (4.5,10);

\end{scope}

\end{scope}

\end{tikzpicture}$$
\caption{Consecutive $\alpha$ braids adds $3$ components to the all-$B$ state. Consecutive $\beta$ braids adds $8$ components to the all-$B$ state.}
\label{figure:stack}
\end{figure}
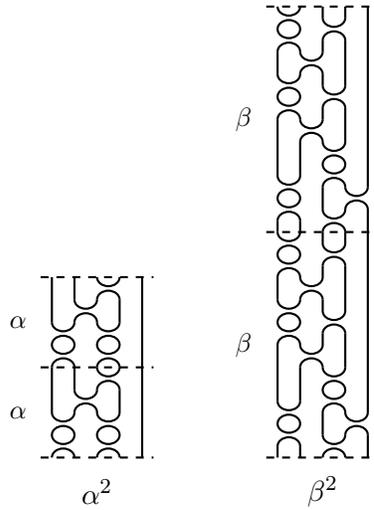

Figure \ref{figure:B5states2} shows the reduced all-$B$ states of $D_{5,q}$ for $q\geq 9$. If $n\geq 2$, then the reduced all-$B$ state of $D_{5,5n}$ has $21$ components. Hence $s_B(D_{5,5n}) = 21 + 12(n-2) = 12n-3$. Equation \ref{equation:g_T(D)} implies that $g_T(D_{5,5n})=4n$. If $n\geq 2$, then the reduced all-$B$ state of $D_{5,5n+1}$ has $25$ components. Hence $s_B(D_{5,5n+1}) = 25 + 12(n-2) = 12n+1$.  Equation \ref{equation:g_T(D)} implies that $g_T(D_{5,5n+1})=4n$. If $n\geq 2$, then the reduced all-$B$ state of $D_{5,5n+2}$ has $27$ components. Hence $s_B(D_{5,5n+2}) = 27+12(n-2)=12n+3$.  Equation \ref{equation:g_T(D)} implies that $g_T(D_{5,5n+2})=4n+1$. If $n\geq 2$, then the reduced all-$B$ state of $D_{5,5n+3}$ has $29$ components. Hence $s_B(D_{5,5n+3}) = 29+12(n-2)= 12n+5$.  Equation \ref{equation:g_T(D)} implies that $g_T(D_{5,5n+3})=4n+2$. Finally, if $n\geq 1$, then the reduced all-$B$ state of $D_{5,5n+4}$ has $19$ components. Hence $s_B(D_{5,5n+4}) = 19+12(n-1) = 12n+7$.  Equation \ref{equation:g_T(D)} implies that $g_T(D_{5,5n+4})=4n+3$.
\begin{figure}[h]
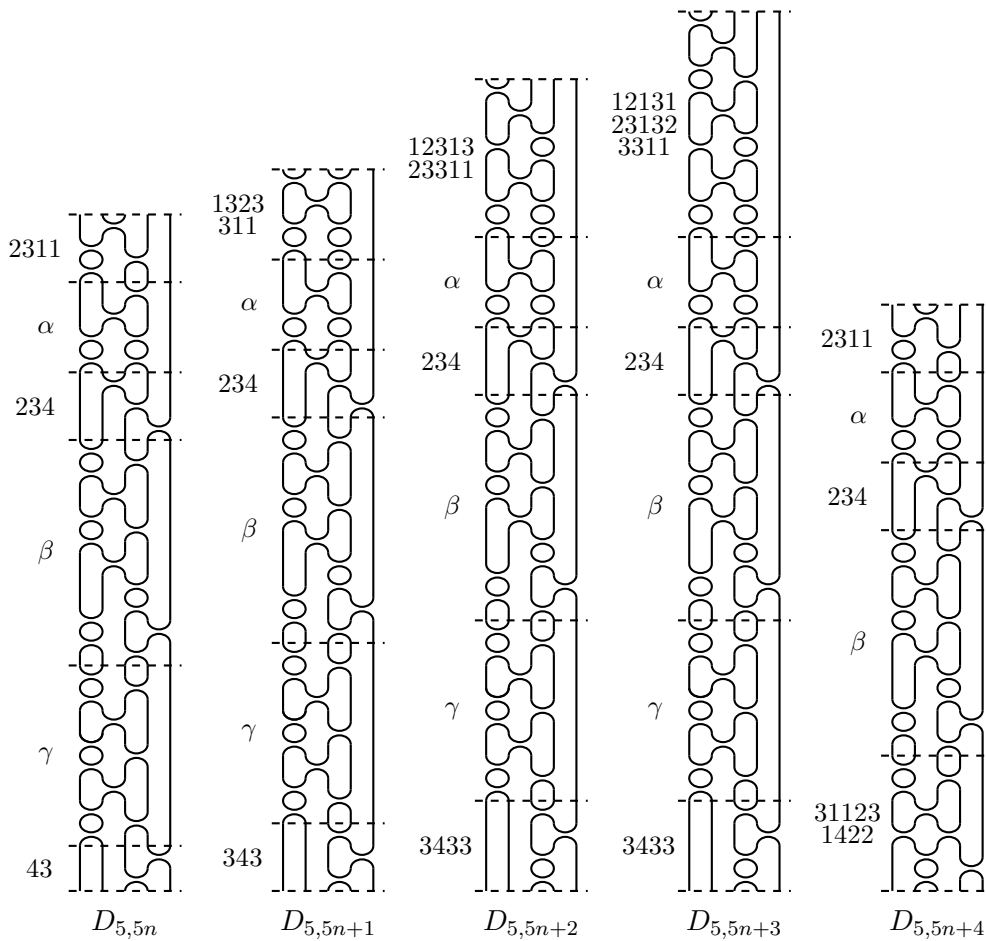

$$
$$

\caption{The reduced all-$B$ states of $D_{5,5n+i}$ for $i=0, 1,2,3$, and $4$.}
\label{figure:B5states2}
\end{figure}
\end{proof}

\begin{proposition}
\label{proposition:TuraevT6q}
Let $D_{6,6n}$ and $D_{6,6n+1}$ be the closure of the braid diagrams appearing on the right hand side of Lemma \ref{lemma:T6q}. Then
$$g_T(D_{6,6n}) = g_T(D_{6,6n+1}) = 6n.$$
\end{proposition}
\begin{proof}
For $i=0$ and $1$, the diagram $D_{6,6n+i}$ has $30n+5i$ crossings. The number of components in the all-$A$ state of $D_{6,6n+i}$ is $s_A(D_{6,6n+i})= 6$. Since the expressions for $D_{6,6n+i}$ in Lemma \ref{lemma:T6q} contain $(\alpha323)^{n-1}$ and $(\beta343)^{n-1}$, we must handle the case where $n=1$ separately from the case where $n>1$. Figure \ref{figure:B6states1} shows the all-$B$ states of $D_{6,6}$ and $D_{6,7}$. From that figure one can see that
$$s_B(D_{6,6}) = 14,~\text{and}~s_B(D_{6,7})=19$$
Equation \ref{equation:g_T(D)} then implies that
$$g_T(D_{6,6}) = g_T(D_{6,7})= 6.$$
\begin{figure}[h]
$$\begin{tikzpicture}[thick, scale = .3, rotate around={90:(0,0)}]

\draw (2,0) arc (180:0:.5cm and .4cm);
\draw (4,0) -- (4,1);
\draw (3,1) arc (180:0:.5cm and .4cm);
\draw (2,1) arc (180:360:.5cm and .4cm);
\draw (2,1) -- (2,2);
\draw (2,2) arc (180:0:.5cm and .4cm);
\draw (3,2)  arc (180:360:.5cm and .4cm);
\draw (0,0) -- (0,2);
\draw (1,0) -- (1,2);
\draw (0,2)  arc (180:0:.5cm and .4cm);
\draw (0,3) arc (180:360:.5cm and .4cm);
\draw (1,3)  arc (180:0:.5cm and .4cm);
\draw (2,3) arc (180:360:.5cm and .4cm);
\draw (0,3) -- (0,4);
\draw (3,3) -- (3,4);
\draw (0,4) arc (180:0:.5cm and .4cm);
\draw (1,4) arc (180:360:.5cm and .4cm);
\draw (2,4)  arc (180:0:.5cm and .4cm);
\draw (4,4) arc (180:0:.5cm and .4cm);
\draw (4,4) -- (4,2);
\draw (5,4) -- (5,0);
\draw (.5,5) ellipse (.5cm and .4cm);
\draw (2.5,5) ellipse (.5cm and .4cm);
\draw (4.5,5) ellipse (.5cm and .4cm);

\draw (2,6) arc (180:360:.5cm and .4cm);
\draw (3,6) arc (180:0:.5cm and .4cm);
\draw (4,6) arc (180:360:.5cm and .4cm);
\draw (2,6) -- (2,7);
\draw (5,6) -- (5,9);
\draw (2,7) arc (180:0:.5cm and .4cm);
\draw (2,8) arc (180:360:.5cm and .4cm);
\draw (1,8)  arc (180:0:.5cm and .4cm);
\draw (0,6) arc (180:360:.5cm and .4cm);
\draw (1,6)--(1,8);
\draw (3,8) -- (3,9);
\draw (2,9) arc (180:0:.5cm and .4cm);
\draw (1,9) arc (180:360:.5cm and .4cm);
\draw (0,9) arc (180:0:.5cm and .4cm);
\draw (0,6) -- (0,9);
\draw (3,7)  arc (180:360:.5cm and .4cm);
\draw (4,7) -- (4,9);
\draw (4,9) arc (180:0:.5cm and .4cm);

\draw (4,10)  arc (180:360:.5cm and .4cm);
\draw (3,10) arc (180:0:.5cm and .4cm);
\draw (2,10) arc (180:360:.5cm and .4cm);
\draw (2,10) -- (2,11);
\draw (2,11) arc (180:0:.5cm and .4cm);
\draw (3,11) arc (180:360:.5cm and .4cm);
\draw (2.5,12) ellipse(.5cm and .4cm);
\draw (0,10) arc (180:360:.5cm and .4cm);
\draw (1,10) --(1,13);
\draw (1,13)  arc (180:0:.5cm and .4cm);
\draw (2,13) arc (180:360:.5cm and .4cm);
\draw (3,13) -- (3,14);
\draw (2,14) arc (180:0:.5cm and .4cm);
\draw (1,14)  arc (180:360:.5cm and .4cm);
\draw (2.5,15) ellipse(.5cm and .4cm);

\draw (4,11) -- (4,15);
\draw (5,10) -- (5,15);
\draw (4,15) arc (180:0:.5cm and .4cm);

\draw (4,16) arc (180:360:.5cm and .4cm);
\draw (5,17) -- (5,16);
\draw (3,16)  arc (180:0:.5cm and .4cm);
\draw (2,16) arc (180:360:.5cm and .4cm);
\draw (2,17) -- (2,16);
\draw (2,17) arc (180:0:.5cm and .4cm);
\draw (3,17) arc (180:360:.5cm and .4cm);
\draw (4,17) arc (180:0:.5cm and .4cm);
\draw (2.5,18)  ellipse(.5cm and .4cm);
\draw (2,19) arc (180:360:.5cm and .4cm);
\draw (1,19) arc (180:0:.5cm and .4cm);
\draw (0,19) arc (180:360:.5cm and .4cm);
\draw (0,18) arc (180:0:.5cm and .4cm);
\draw (0,18) -- (0,10);
\draw (1,18) -- (1,14);
\draw (1,20) arc (180:360:.5cm and .4cm);
\draw (0,19) -- (0,20);
\draw (3,19) -- (3,20);
\draw (4,20) -- (4,18);
\draw (5,20) -- (5,18);
\draw (4,18) arc (180:360:.5cm and .4cm);

\draw[dashed] (-.5,0) -- (5.5,0);
\draw[dashed] (-.5,20) -- (5.5,20);

\draw[dashed] (-.5,2) -- (5.5,2);
\draw[dashed] (-.5,6) -- (5.5,6);
\draw[dashed] (-.5,9) -- (5.5,9);
\draw[dashed] (-.5,19) -- (5.5,19);

\draw (6.5,1) node{\small{$43$}};
\draw (6.5,4) node{\small{$\eta$}};
\draw (6.5, 7.5) node{\small{$234$}};
\draw (6.5,14) node{\small{$\zeta$}};
\draw (6.5,19.5) node{\small{$2$}};

\draw (3,22) node{$D_{6,6}$};


\begin{scope}[xshift = -10cm, yshift =-2cm]

\draw (2,-1) arc (180:0:.5cm and .4cm);
\draw (3,0) arc (180:0:.5cm and .4cm);
\draw (2,0) arc (180:360:.5cm and .4cm);
\draw (2,1) -- (2,0);
\draw (2,1) arc (180:0:.5cm and .4cm);
\draw (3,1)  arc (180:360:.5cm and .4cm);
\draw (2.5,2) ellipse (.5cm and .4cm);
\draw (0,0) -- (0,2);
\draw (1,0) -- (1,2);
\draw (0,2)  arc (180:0:.5cm and .4cm);
\draw (0,3) arc (180:360:.5cm and .4cm);
\draw (1,3)  arc (180:0:.5cm and .4cm);
\draw (2,3) arc (180:360:.5cm and .4cm);
\draw (0,3) -- (0,4);
\draw (3,3) -- (3,4);
\draw (0,4) arc (180:0:.5cm and .4cm);
\draw (1,4) arc (180:360:.5cm and .4cm);
\draw (2,4)  arc (180:0:.5cm and .4cm);
\draw (4,4) arc (180:0:.5cm and .4cm);
\draw (4,4) -- (4,1);
\draw (4,0) -- (4,-1);
\draw (5,4) -- (5,-1);
\draw (.5,5) ellipse (.5cm and .4cm);
\draw (2.5,5) ellipse (.5cm and .4cm);
\draw (4.5,5) ellipse (.5cm and .4cm);

\draw (2,6) arc (180:360:.5cm and .4cm);
\draw (3,6) arc (180:0:.5cm and .4cm);
\draw (4,6) arc (180:360:.5cm and .4cm);
\draw (2,6) -- (2,7);
\draw (5,6) -- (5,9);
\draw (2,7) arc (180:0:.5cm and .4cm);
\draw (2,8) arc (180:360:.5cm and .4cm);
\draw (1,8)  arc (180:0:.5cm and .4cm);
\draw (0,6) arc (180:360:.5cm and .4cm);
\draw (1,6)--(1,8);
\draw (3,8) -- (3,9);
\draw (2,9) arc (180:0:.5cm and .4cm);
\draw (1,9) arc (180:360:.5cm and .4cm);
\draw (0,9) arc (180:0:.5cm and .4cm);
\draw (0,6) -- (0,9);
\draw (3,7)  arc (180:360:.5cm and .4cm);
\draw (4,7) -- (4,9);
\draw (4,9) arc (180:0:.5cm and .4cm);

\draw (4,10)  arc (180:360:.5cm and .4cm);
\draw (3,10) arc (180:0:.5cm and .4cm);
\draw (2,10) arc (180:360:.5cm and .4cm);
\draw (2,10) -- (2,11);
\draw (2,11) arc (180:0:.5cm and .4cm);
\draw (3,11) arc (180:360:.5cm and .4cm);
\draw (2.5,12) ellipse(.5cm and .4cm);
\draw (0,10) arc (180:360:.5cm and .4cm);
\draw (1,10) --(1,13);
\draw (1,13)  arc (180:0:.5cm and .4cm);
\draw (2,13) arc (180:360:.5cm and .4cm);
\draw (3,13) -- (3,14);
\draw (2,14) arc (180:0:.5cm and .4cm);
\draw (1,14)  arc (180:360:.5cm and .4cm);
\draw (2.5,15) ellipse(.5cm and .4cm);

\draw (4,11) -- (4,15);
\draw (5,10) -- (5,15);
\draw (4,15) arc (180:0:.5cm and .4cm);

\draw (4,16) arc (180:360:.5cm and .4cm);
\draw (5,17) -- (5,16);
\draw (3,16)  arc (180:0:.5cm and .4cm);
\draw (2,16) arc (180:360:.5cm and .4cm);
\draw (2,17) -- (2,16);
\draw (2,17) arc (180:0:.5cm and .4cm);
\draw (3,17) arc (180:360:.5cm and .4cm);
\draw (4,17) arc (180:0:.5cm and .4cm);
\draw (2.5,18)  ellipse(.5cm and .4cm);
\draw (2.5,19) ellipse(.5cm and .4cm);
\draw (2,20)  arc (180:360:.5cm and .4cm);
\draw (1,20) arc (180:0:.5cm and .4cm);

\draw (0,19) arc (180:360:.5cm and .4cm);
\draw (0,18) arc (180:0:.5cm and .4cm);
\draw (0,18) -- (0,10);
\draw (1,18) -- (1,14);

\draw (1,19) -- (1,20);
\draw (0,19) -- (0,21);
\draw (0,21) arc (180:0:.5cm and .4cm);
\draw (1,21) arc (180:360:.5cm and .4cm);
\draw (2,21) arc (180:0:.5cm and .4cm);
\draw (3,21) -- (3,20);
\draw (0,22) arc (180:360:.5cm and .4cm);
\draw (2,22) arc (180:360:.5cm and .4cm);

\draw (4,22) -- (4,18);
\draw (5,22) -- (5,18);
\draw (4,18) arc (180:360:.5cm and .4cm);

\draw(4,-1)  arc (180:360:.5cm and .4cm);
\draw (4,-2) arc (180:0:.5cm and .4cm);
\draw (0,-2) -- (0,0);
\draw (1,-2) -- (1,0);
\draw (2,-2) -- (2,-1);
\draw (3,-2) -- (3,-1);

\draw[dashed] (-.5,-2) -- (5.5,-2);
\draw[dashed] (-.5, 22) -- (5.5,22);
\draw[dashed] (-.5,19) -- (5.5,19);
\draw[dashed] (-.5,9) -- (5.5,9);
\draw[dashed] (-.5,6) -- (5.5,6);
\draw[dashed] (-.5,2) -- (5.5,2);

\draw (6.5,20.5) node{\small{$1323$}};
\draw (6.5,14) node{\small{$\zeta$}};
\draw (6.5, 7.5) node{\small{$234$}};
\draw (6.5,4) node{\small{$\eta$}};
\draw (6.5,0) node{\small{$3435$}};

\draw (2,24) node{$D_{6,7}$};

\end{scope}

\end{tikzpicture}$$
\caption{The all-$B$ states of $D_{6,6}$ and $D_{6,7}$.}
\label{figure:B6states1}
\end{figure}
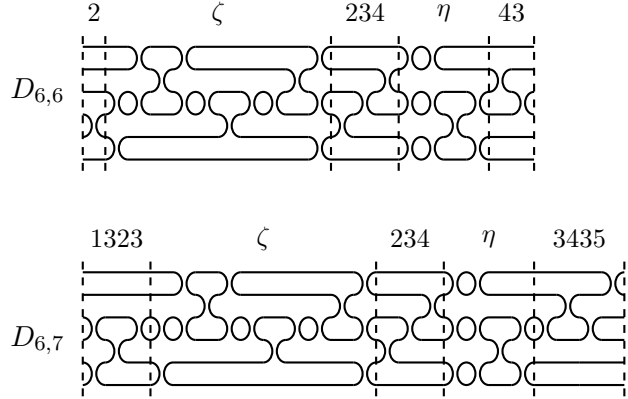

If $n\ge 2$, then $D_{6,6n}$ and $D_{6,6n+1}$ contains at least one $\alpha323$ and $\beta343$ as a sub-word in its braid word. A straightforward computation shows that consecutive $\alpha323$ words add $10$ components to the all-$B$ state and consecutive $\beta343$ words add $8$ components to the all-$B$ state. This observation allows us to compute the number of components in the all-$B$ state of $D_{6,6n}$ and $D_{6,6n+1}$ by only considering braid words with $\alpha323$ and $\beta343$ appearing once. A braid word where $(\alpha323)^{\ell}$ and $(\beta343)^m$ are replaced by $\alpha323$ and $\beta343$ respectively will be called {\em reduced}.

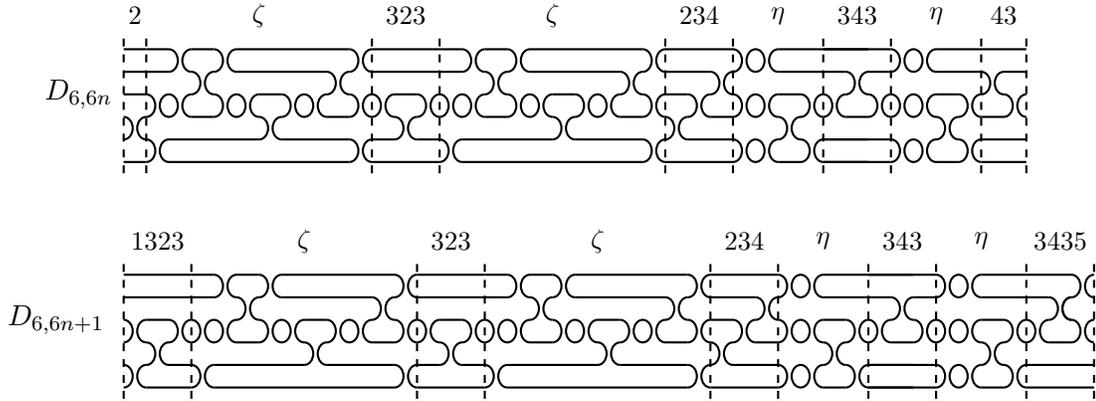
\begin{figure}[h]
$$\begin{tikzpicture}[thick, scale = .3, rotate around={90:(0,0)}]

\draw (2,0) arc (180:0:.5cm and .4cm);
\draw (4,0) -- (4,1);
\draw (3,1) arc (180:0:.5cm and .4cm);
\draw (2,1) arc (180:360:.5cm and .4cm);
\draw (2,1) -- (2,2);
\draw (2,2) arc (180:0:.5cm and .4cm);
\draw (3,2)  arc (180:360:.5cm and .4cm);
\draw (0,0) -- (0,2);
\draw (1,0) -- (1,2);
\draw (0,2)  arc (180:0:.5cm and .4cm);
\draw (0,3) arc (180:360:.5cm and .4cm);
\draw (1,3)  arc (180:0:.5cm and .4cm);
\draw (2,3) arc (180:360:.5cm and .4cm);
\draw (0,3) -- (0,4);
\draw (3,3) -- (3,4);
\draw (0,4) arc (180:0:.5cm and .4cm);
\draw (1,4) arc (180:360:.5cm and .4cm);
\draw (2,4)  arc (180:0:.5cm and .4cm);
\draw (4,4) arc (180:0:.5cm and .4cm);
\draw (4,4) -- (4,2);
\draw (5,4) -- (5,0);
\draw (.5,5) ellipse (.5cm and .4cm);
\draw (2.5,5) ellipse (.5cm and .4cm);
\draw (4.5,5) ellipse (.5cm and .4cm);

\draw (2,6) arc (180:360:.5cm and .4cm);
\draw (4,6) arc (180:360:.5cm and .4cm);
\draw (0,6) arc (180:360:.5cm and .4cm);

\draw (2,6) arc (180:0:.5cm and .4cm);
\draw (2,7) arc (180:360:.5cm and .4cm);
\draw (3,7)  arc (180:0:.5cm and .4cm);
\draw (3,8) arc (180:360:.5cm and .4cm);
\draw (2,8) arc (180:0:.5cm and .4cm);
\draw (2,9) arc (180:360:.5cm and .4cm);
\draw (2,7) -- (2,8);
\draw (0,6) -- (0,9);
\draw (1,6) -- (1,9);
\draw (5,6) -- (5,9);
\draw (4,6) -- (4,7);
\draw (4,8) -- (4,9);
\draw[dashed] (-.5,9) -- (5.5,9);

\draw (6.5,7.5) node{\small{$343$}};

\begin{scope}[yshift=7cm]
\draw (2,2) arc (180:0:.5cm and .4cm);
\draw (0,0) -- (0,2);
\draw (1,0) -- (1,2);
\draw (0,2)  arc (180:0:.5cm and .4cm);
\draw (0,3) arc (180:360:.5cm and .4cm);
\draw (1,3)  arc (180:0:.5cm and .4cm);
\draw (2,3) arc (180:360:.5cm and .4cm);
\draw (0,3) -- (0,4);
\draw (3,3) -- (3,4);
\draw (0,4) arc (180:0:.5cm and .4cm);
\draw (1,4) arc (180:360:.5cm and .4cm);
\draw (2,4)  arc (180:0:.5cm and .4cm);
\draw (4,4) arc (180:0:.5cm and .4cm);
\draw (4,4) -- (4,2);
\draw (5,4) -- (5,0);
\draw (.5,5) ellipse (.5cm and .4cm);
\draw (2.5,5) ellipse (.5cm and .4cm);
\draw (4.5,5) ellipse (.5cm and .4cm);

\draw (2,6) arc (180:360:.5cm and .4cm);
\draw (4,6) arc (180:360:.5cm and .4cm);
\draw (0,6) arc (180:360:.5cm and .4cm);
\draw[dashed] (-.5,6) -- (5.5,6);
\draw (6.5,4) node{\small{$\eta$}};
\end{scope}

\begin{scope}[yshift = 7cm]
\draw (3,6) arc (180:0:.5cm and .4cm);

\draw (2,6) -- (2,7);
\draw (5,6) -- (5,9);
\draw (2,7) arc (180:0:.5cm and .4cm);
\draw (2,8) arc (180:360:.5cm and .4cm);
\draw (1,8)  arc (180:0:.5cm and .4cm);

\draw (1,6)--(1,8);
\draw (3,8) -- (3,9);
\draw (2,9) arc (180:0:.5cm and .4cm);
\draw (1,9) arc (180:360:.5cm and .4cm);
\draw (0,9) arc (180:0:.5cm and .4cm);
\draw (0,6) -- (0,9);
\draw (3,7)  arc (180:360:.5cm and .4cm);
\draw (4,7) -- (4,9);
\draw (4,9) arc (180:0:.5cm and .4cm);

\draw (4,10)  arc (180:360:.5cm and .4cm);
\draw (3,10) arc (180:0:.5cm and .4cm);
\draw (2,10) arc (180:360:.5cm and .4cm);
\draw (2,10) -- (2,11);
\draw (2,11) arc (180:0:.5cm and .4cm);
\draw (3,11) arc (180:360:.5cm and .4cm);
\draw (2.5,12) ellipse(.5cm and .4cm);
\draw (0,10) arc (180:360:.5cm and .4cm);
\draw (1,10) --(1,13);
\draw (1,13)  arc (180:0:.5cm and .4cm);
\draw (2,13) arc (180:360:.5cm and .4cm);
\draw (3,13) -- (3,14);
\draw (2,14) arc (180:0:.5cm and .4cm);
\draw (1,14)  arc (180:360:.5cm and .4cm);
\draw (2.5,15) ellipse(.5cm and .4cm);

\draw (4,11) -- (4,15);
\draw (5,10) -- (5,15);
\draw (4,15) arc (180:0:.5cm and .4cm);

\draw (4,16) arc (180:360:.5cm and .4cm);
\draw (5,17) -- (5,16);
\draw (3,16)  arc (180:0:.5cm and .4cm);
\draw (2,16) arc (180:360:.5cm and .4cm);
\draw (2,17) -- (2,16);
\draw (2,17) arc (180:0:.5cm and .4cm);
\draw (3,17) arc (180:360:.5cm and .4cm);
\draw (4,17) arc (180:0:.5cm and .4cm);
\draw (2.5,18)  ellipse(.5cm and .4cm);
\draw (2,19) arc (180:360:.5cm and .4cm);
\draw (0,19) arc (180:360:.5cm and .4cm);
\draw (0,18) arc (180:0:.5cm and .4cm);
\draw (0,18) -- (0,10);
\draw (1,18) -- (1,14);

\draw (4,19) -- (4,18);
\draw (5,19) -- (5,18);
\draw (4,18) arc (180:360:.5cm and .4cm);

\draw (2,19)  arc (180:0:.5cm and .4cm);

\draw (2,20)  arc (180:360:.5cm and .4cm);
\draw (1,20) arc (180:0:.5cm and .4cm);
\draw (1,21) arc (180:360:.5cm and .4cm);
\draw (2,21) arc (180:0:.5cm and .4cm);
\draw (2,22) arc (180:360:.5cm and .4cm);
\draw (0,19) -- (0,22);
\draw (1,19) -- (1,20);
\draw (1,21) -- (1,22);
\draw (3,20) -- (3,21);
\draw (4,19) -- (4,22);
\draw (5,19) -- (5,22);
\draw[dashed] (-.5,22) -- (5.5,22);
\draw (6.5,20.5) node{\small{$323$}};

\begin{scope}[yshift = 13cm]
\draw (0,9) arc (180:0:.5cm and .4cm);
\draw (2,9) arc (180:0:.5cm and .4cm);

\draw (4,9) arc (180:0:.5cm and .4cm);

\draw (4,10)  arc (180:360:.5cm and .4cm);
\draw (3,10) arc (180:0:.5cm and .4cm);
\draw (2,10) arc (180:360:.5cm and .4cm);
\draw (2,10) -- (2,11);
\draw (2,11) arc (180:0:.5cm and .4cm);
\draw (3,11) arc (180:360:.5cm and .4cm);
\draw (2.5,12) ellipse(.5cm and .4cm);
\draw (0,10) arc (180:360:.5cm and .4cm);
\draw (1,10) --(1,13);
\draw (1,13)  arc (180:0:.5cm and .4cm);
\draw (2,13) arc (180:360:.5cm and .4cm);
\draw (3,13) -- (3,14);
\draw (2,14) arc (180:0:.5cm and .4cm);
\draw (1,14)  arc (180:360:.5cm and .4cm);
\draw (2.5,15) ellipse(.5cm and .4cm);

\draw (4,11) -- (4,15);
\draw (5,10) -- (5,15);
\draw (4,15) arc (180:0:.5cm and .4cm);

\draw (4,16) arc (180:360:.5cm and .4cm);
\draw (5,17) -- (5,16);
\draw (3,16)  arc (180:0:.5cm and .4cm);
\draw (2,16) arc (180:360:.5cm and .4cm);
\draw (2,17) -- (2,16);
\draw (2,17) arc (180:0:.5cm and .4cm);
\draw (3,17) arc (180:360:.5cm and .4cm);
\draw (4,17) arc (180:0:.5cm and .4cm);
\draw (2.5,18)  ellipse(.5cm and .4cm);
\draw (2,19) arc (180:360:.5cm and .4cm);
\draw (0,19) arc (180:360:.5cm and .4cm);
\draw (0,18) arc (180:0:.5cm and .4cm);
\draw (0,18) -- (0,10);
\draw (1,18) -- (1,14);

\draw (4,19) -- (4,18);
\draw (5,19) -- (5,18);
\draw (4,18) arc (180:360:.5cm and .4cm);

\draw[dashed] (-.5,19) -- (5.5,19);
\draw (6.5,14) node{\small{$\zeta$}};

\draw (1,19) arc (180:0:.5cm and .4cm);
\draw (1,20) arc (180:360:.5cm and .4cm);
\draw (0,19) -- (0,20);
\draw (3,19) -- (3,20);
\draw (4,19) -- (4,20);
\draw (5,19) -- (5,20);
\draw (6.5,19.5) node{\small{$2$}};
\draw[dashed] (-.5,20) -- (5.5,20);
\end{scope}

\draw[dashed] (-.5,9) -- (5.5,9);
\draw[dashed] (-.5,19) -- (5.5,19);

\draw (6.5, 7.5) node{\small{$234$}};
\draw (6.5,14) node{\small{$\zeta$}};

\end{scope}

\draw[dashed] (-.5,0) -- (5.5,0);
\draw[dashed] (-.5,2) -- (5.5,2);
\draw[dashed] (-.5,6) -- (5.5,6);
\draw (6.5,1) node{\small{$43$}};
\draw (6.5,4) node{\small{$\eta$}};

\draw (3,42) node{$D_{6,6n}$};


\begin{scope}[xshift = -10cm, yshift = -2cm]

\draw (2,2) arc (180:360:.5cm and .4cm);
\draw (2,1)  arc (180:0:.5cm and .4cm);
\draw (3,1) arc (180:360:.5cm and .4cm);
\draw (3,0) arc (180:0:.5cm and .4cm);
\draw (2,0) arc (180:360:.5cm and .4cm);
\draw (2,-1) arc (180:0:.5cm and .4cm);
\draw (3,0) arc (180:0:.5cm and .4cm);
\draw (4,0) arc (180:360:.5cm and .4cm);
\draw (4,-1) arc (180:0:.5cm and .4cm);
\draw (0,-1) -- (0,2);
\draw (1,-1) -- (1,2);
\draw (2,0) -- (2,1);
\draw (4,2) -- (4,1);
\draw (5,2) -- (5,0);

\draw (2,2) arc (180:0:.5cm and .4cm);
\draw (0,2)  arc (180:0:.5cm and .4cm);
\draw (0,3) arc (180:360:.5cm and .4cm);
\draw (1,3)  arc (180:0:.5cm and .4cm);
\draw (2,3) arc (180:360:.5cm and .4cm);
\draw (0,3) -- (0,4);
\draw (3,3) -- (3,4);
\draw (0,4) arc (180:0:.5cm and .4cm);
\draw (1,4) arc (180:360:.5cm and .4cm);
\draw (2,4)  arc (180:0:.5cm and .4cm);
\draw (4,4) arc (180:0:.5cm and .4cm);
\draw (4,4) -- (4,2);
\draw (5,4) -- (5,2);
\draw (.5,5) ellipse (.5cm and .4cm);
\draw (2.5,5) ellipse (.5cm and .4cm);
\draw (4.5,5) ellipse (.5cm and .4cm);

\draw (2,6) arc (180:360:.5cm and .4cm);
\draw (4,6) arc (180:360:.5cm and .4cm);
\draw (0,6) arc (180:360:.5cm and .4cm);

\draw (2,6) arc (180:0:.5cm and .4cm);
\draw (2,7) arc (180:360:.5cm and .4cm);
\draw (3,7)  arc (180:0:.5cm and .4cm);
\draw (3,8) arc (180:360:.5cm and .4cm);
\draw (2,8) arc (180:0:.5cm and .4cm);
\draw (2,9) arc (180:360:.5cm and .4cm);
\draw (2,7) -- (2,8);
\draw (0,6) -- (0,9);
\draw (1,6) -- (1,9);
\draw (5,6) -- (5,9);
\draw (4,6) -- (4,7);
\draw (4,8) -- (4,9);
\draw[dashed] (-.5,9) -- (5.5,9);

\draw (6.5,7.5) node{\small{$343$}};

\begin{scope}[yshift=7cm]
\draw (2,2) arc (180:0:.5cm and .4cm);
\draw (0,0) -- (0,2);
\draw (1,0) -- (1,2);
\draw (0,2)  arc (180:0:.5cm and .4cm);
\draw (0,3) arc (180:360:.5cm and .4cm);
\draw (1,3)  arc (180:0:.5cm and .4cm);
\draw (2,3) arc (180:360:.5cm and .4cm);
\draw (0,3) -- (0,4);
\draw (3,3) -- (3,4);
\draw (0,4) arc (180:0:.5cm and .4cm);
\draw (1,4) arc (180:360:.5cm and .4cm);
\draw (2,4)  arc (180:0:.5cm and .4cm);
\draw (4,4) arc (180:0:.5cm and .4cm);
\draw (4,4) -- (4,2);
\draw (5,4) -- (5,0);
\draw (.5,5) ellipse (.5cm and .4cm);
\draw (2.5,5) ellipse (.5cm and .4cm);
\draw (4.5,5) ellipse (.5cm and .4cm);

\draw (2,6) arc (180:360:.5cm and .4cm);
\draw (4,6) arc (180:360:.5cm and .4cm);
\draw (0,6) arc (180:360:.5cm and .4cm);
\draw[dashed] (-.5,6) -- (5.5,6);
\draw (6.5,4) node{\small{$\eta$}};
\end{scope}

\begin{scope}[yshift = 7cm]
\draw (3,6) arc (180:0:.5cm and .4cm);

\draw (2,6) -- (2,7);
\draw (5,6) -- (5,9);
\draw (2,7) arc (180:0:.5cm and .4cm);
\draw (2,8) arc (180:360:.5cm and .4cm);
\draw (1,8)  arc (180:0:.5cm and .4cm);

\draw (1,6)--(1,8);
\draw (3,8) -- (3,9);
\draw (2,9) arc (180:0:.5cm and .4cm);
\draw (1,9) arc (180:360:.5cm and .4cm);
\draw (0,9) arc (180:0:.5cm and .4cm);
\draw (0,6) -- (0,9);
\draw (3,7)  arc (180:360:.5cm and .4cm);
\draw (4,7) -- (4,9);
\draw (4,9) arc (180:0:.5cm and .4cm);

\draw (4,10)  arc (180:360:.5cm and .4cm);
\draw (3,10) arc (180:0:.5cm and .4cm);
\draw (2,10) arc (180:360:.5cm and .4cm);
\draw (2,10) -- (2,11);
\draw (2,11) arc (180:0:.5cm and .4cm);
\draw (3,11) arc (180:360:.5cm and .4cm);
\draw (2.5,12) ellipse(.5cm and .4cm);
\draw (0,10) arc (180:360:.5cm and .4cm);
\draw (1,10) --(1,13);
\draw (1,13)  arc (180:0:.5cm and .4cm);
\draw (2,13) arc (180:360:.5cm and .4cm);
\draw (3,13) -- (3,14);
\draw (2,14) arc (180:0:.5cm and .4cm);
\draw (1,14)  arc (180:360:.5cm and .4cm);
\draw (2.5,15) ellipse(.5cm and .4cm);

\draw (4,11) -- (4,15);
\draw (5,10) -- (5,15);
\draw (4,15) arc (180:0:.5cm and .4cm);

\draw (4,16) arc (180:360:.5cm and .4cm);
\draw (5,17) -- (5,16);
\draw (3,16)  arc (180:0:.5cm and .4cm);
\draw (2,16) arc (180:360:.5cm and .4cm);
\draw (2,17) -- (2,16);
\draw (2,17) arc (180:0:.5cm and .4cm);
\draw (3,17) arc (180:360:.5cm and .4cm);
\draw (4,17) arc (180:0:.5cm and .4cm);
\draw (2.5,18)  ellipse(.5cm and .4cm);
\draw (2,19) arc (180:360:.5cm and .4cm);
\draw (0,19) arc (180:360:.5cm and .4cm);
\draw (0,18) arc (180:0:.5cm and .4cm);
\draw (0,18) -- (0,10);
\draw (1,18) -- (1,14);

\draw (4,19) -- (4,18);
\draw (5,19) -- (5,18);
\draw (4,18) arc (180:360:.5cm and .4cm);

\draw (2,19)  arc (180:0:.5cm and .4cm);

\draw (2,20)  arc (180:360:.5cm and .4cm);
\draw (1,20) arc (180:0:.5cm and .4cm);
\draw (1,21) arc (180:360:.5cm and .4cm);
\draw (2,21) arc (180:0:.5cm and .4cm);
\draw (2,22) arc (180:360:.5cm and .4cm);
\draw (0,19) -- (0,22);
\draw (1,19) -- (1,20);
\draw (1,21) -- (1,22);
\draw (3,20) -- (3,21);
\draw (4,19) -- (4,22);
\draw (5,19) -- (5,22);
\draw[dashed] (-.5,22) -- (5.5,22);
\draw (6.5,20.5) node{\small{$323$}};

\begin{scope}[yshift = 13cm]
\draw (0,9) arc (180:0:.5cm and .4cm);
\draw (2,9) arc (180:0:.5cm and .4cm);

\draw (4,9) arc (180:0:.5cm and .4cm);

\draw (4,10)  arc (180:360:.5cm and .4cm);
\draw (3,10) arc (180:0:.5cm and .4cm);
\draw (2,10) arc (180:360:.5cm and .4cm);
\draw (2,10) -- (2,11);
\draw (2,11) arc (180:0:.5cm and .4cm);
\draw (3,11) arc (180:360:.5cm and .4cm);
\draw (2.5,12) ellipse(.5cm and .4cm);
\draw (0,10) arc (180:360:.5cm and .4cm);
\draw (1,10) --(1,13);
\draw (1,13)  arc (180:0:.5cm and .4cm);
\draw (2,13) arc (180:360:.5cm and .4cm);
\draw (3,13) -- (3,14);
\draw (2,14) arc (180:0:.5cm and .4cm);
\draw (1,14)  arc (180:360:.5cm and .4cm);
\draw (2.5,15) ellipse(.5cm and .4cm);

\draw (4,11) -- (4,15);
\draw (5,10) -- (5,15);
\draw (4,15) arc (180:0:.5cm and .4cm);

\draw (4,16) arc (180:360:.5cm and .4cm);
\draw (5,17) -- (5,16);
\draw (3,16)  arc (180:0:.5cm and .4cm);
\draw (2,16) arc (180:360:.5cm and .4cm);
\draw (2,17) -- (2,16);
\draw (2,17) arc (180:0:.5cm and .4cm);
\draw (3,17) arc (180:360:.5cm and .4cm);
\draw (4,17) arc (180:0:.5cm and .4cm);
\draw (2.5,18)  ellipse(.5cm and .4cm);
\draw (2,19) arc (180:360:.5cm and .4cm);
\draw (0,19) arc (180:360:.5cm and .4cm);
\draw (0,18) arc (180:0:.5cm and .4cm);
\draw (0,18) -- (0,10);
\draw (1,18) -- (1,14);

\draw (4,19) -- (4,18);
\draw (5,19) -- (5,18);
\draw (4,18) arc (180:360:.5cm and .4cm);

\draw[dashed] (-.5,19) -- (5.5,19);
\draw (6.5,14) node{\small{$\zeta$}};

\draw (2,19) arc (180:0:.5cm and .4cm);
\draw (2,20)  arc (180:360:.5cm and .4cm);
\draw (1,20)  arc (180:0:.5cm and .4cm);
\draw (1,21) arc (180:360:.5cm and .4cm);
\draw (0,21) arc (180:0:.5cm and .4cm);
\draw (2,21) arc (180:0:.5cm and .4cm);
\draw (0,22) arc (180:360:.5cm and .4cm);
\draw (2,22) arc (180:360:.5cm and .4cm);
\draw (0,19) -- (0,21);
\draw (1,19) -- (1,20);
\draw (3,20) -- (3,21);
\draw (4,19) -- (4,22);
\draw (5,19) -- (5,22);
\draw[dashed] (-.5,22) -- (5.5,22);
\draw (6.5, 20.5) node{\small{$1323$}};
\end{scope}

\draw[dashed] (-.5,9) -- (5.5,9);
\draw[dashed] (-.5,19) -- (5.5,19);

\draw (6.5, 7.5) node{\small{$234$}};
\draw (6.5,14) node{\small{$\zeta$}};

\end{scope}

\draw[dashed] (-.5,-1) -- (5.5,-1);
\draw[dashed] (-.5,2) -- (5.5,2);
\draw[dashed] (-.5,6) -- (5.5,6);
\draw (6.5,.5) node{\small{$3435$}};
\draw (6.5,4) node{\small{$\eta$}};

\draw (3,45) node{$D_{6,6n+1}$};

\end{scope}

\end{tikzpicture}$$

\caption{The reduced all-$B$ states of $D_{6,6n}$ and $D_{6,6n+1}$ for $n\geq 2$.}
\label{figure:B6states2}
\end{figure}

Figure \ref{figure:B6states2} shows the reduced all-$B$ states of $D_{6,6n}$  and $D_{6,6n+1}$ for $n\geq 2$. The reduced all-$B$ state of $D_{6,6n}$ has $32$ components. Hence $s_B(D_{6,6n}) = 32 + 18(n-2) = 18n-4$. Equation \ref{equation:g_T(D)} implies that $g_T(D_{6,6n})=6n$. The reduced all-$B$ state of $D_{6,6n+1}$ has $37$ components. Hence $s_B(D_{6,6n+1}) = 37 + 18(n-2) = 18n+1$.  Equation \ref{equation:g_T(D)} implies that $g_T(D_{6,6n+1})=6n$. 

\end{proof}

\begin{proof}[Proof of Theorem \ref{theorem:TuraevExact}]
Suppose that $n=0$. Since $T_{4,1}$, $T_{5,1}$, and $T_{6,1}$ are all unknots, it follows that their Turaev genera are all zero. Moreover, since $T_{4,3}=T_{3,4}$, we have $g_T(T_{4,3})=1$.

If $n>0$, then the result is implied by Theorem \ref{theorem:HFKWidth} and Propositions \ref{proposition:TuraevT4q}, \ref{proposition:TuraevT5q}, and \ref{proposition:TuraevT6q}.
\end{proof}

Propositions \ref{proposition:TuraevT4q}, \ref{proposition:TuraevT5q}, and \ref{proposition:TuraevT6q} also immediately imply the following theorem.
\begin{theorem}
\label{theorem:TuraevLink}
For $n\geq 1$ and $j=2,3$, and $4$, we have
$$\begin{array}{>{\hfil$}p{5 cm}<{$\hfil} >{\hfil$}p{5 cm}<{$\hfil}}
g_T(T_{4,4n}) \leq  2n, &
g_T(T_{4,4n+2})\leq  2n+1,\\
g_T(T_{5,5n})\leq  4n, &
 g_T(T_{6,6n})\leq  6n.
\end{array}$$
\end{theorem}

The diagrams in Lemmas \ref{lemma:T4q}, \ref{lemma:T5q}, and \ref{lemma:T6q} are the starting points for the proof of Theorem \ref{theorem:DaltBest}.
\begin{proof}[Proof of Theorem \ref{theorem:DaltBest}]
In any of the braid words from Lemmas \ref{lemma:T4q}, \ref{lemma:T5q}, and \ref{lemma:T6q}, changing the even indexed crossings (i.e. $\sigma_2$ in $T_{4,q}$ and $\sigma_2$ and $\sigma_4$ in $T_{5,q}$ and $T_{6,q}$) gives an alternating diagram. This process yields $\dalt(D_{6,6n+1}) = 6n+2$. Thus $6n\leq \dalt(T_{6,6n+1})\leq 6n+2.$ Applying the above process to the $4$ and $5$ stranded torus knots does not lead to the smallest possible dealternating number; we will introduce two tricks to improve the result. 

The first trick that we use is to replace a twist region of even indexed crossings with the same region encircled by one of its incoming strands (as in Figure \ref{figure:wrap}). Let $D_{4,4n+1}'$ and $D_{4,4n+3}'$ be the diagrams $D_{4,4n+1}$ and $D_{4,4n+3}$ with each $\sigma_2$-twist region modified as above. Changing the crossings at the bottom right of each modified twist region and changing the isolated $\sigma_2$ crossing results in an alternating diagram. Thus the dealternating number of $D_{4,4n+1}'$ and $D_{4,4n+3}'$ is the number of isolated $\sigma_2$ crossings plus the number of $\sigma_2$ twist regions, and hence $\dalt(D_{4,4n+1}') = 2n+1$ and $\dalt(D_{4,4n+3}')=2n+3$. Theorem \ref{theorem:HFKWidth} implies that $2n\leq \dalt(T_{4,4n+1}) \leq 2n+1$ and $2n+1 \leq \dalt(T_{4,4n+3})\leq 2n+2$. 

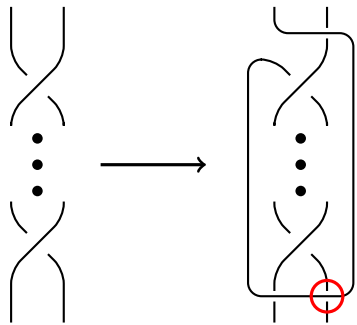
\begin{figure}[h]
$$\begin{tikzpicture}[rounded corners = 1.8mm, thick, scale = .7]
\draw (0,0) -- (0,1) -- (1,2) -- (1,2.3);
\draw (1,0) -- (1,1) -- (.7,1.3);
\draw (.3,1.7) -- (0,2) -- (0,2.3);

\fill (.5,2.5) circle (.1cm);
\fill (.5,3) circle (.1cm);
\fill (.5,3.5) circle (.1cm);

\draw (0,3.8) -- (0,4) -- (1,5) -- (1,6);
\draw (1,3.8) -- (1,4) -- (.7,4.3);
\draw (.3,4.7) -- (0,5) -- (0,6);

\begin{scope} [xshift = 5cm]
\draw (0,6) -- (0,5.5) -- (1.5,5.5) -- (1.5,.5) -- (-.5,.5) -- (-.5,5) -- (0,5) -- (0.3,4.7);
\draw (1,6) -- (1,5.6);
\draw (1,5.4) -- (1,5) -- (0,4) -- (0,3.8);
\draw (1,3.8) -- (1,4) -- (.7,4.3);
\fill (.5,2.5) circle (.1cm);
\fill (.5,3) circle (.1cm);
\fill (.5,3.5) circle (.1cm);

\draw (0,0) -- (0,.4);
\draw (0,.6) -- (0,1) -- (1,2) -- (1,2.3);
\draw (1,0) -- (1,.4);
\draw (1,.6)  -- (1,1) -- (.7,1.3);
\draw (.3,1.7) -- (0,2) -- (0,2.3);

\draw[very thick, red] (1,.5) circle (.3cm);

\end{scope}

\draw[very thick, ->] (1.7,3) -- (3.7,3);

\end{tikzpicture}$$
\caption{Changing the circled crossing on the lower right makes this region alternating.}
\label{figure:wrap}
\end{figure}

The second trick we use is wrap a strand of the link between the $\sigma_3$ and $\sigma_4$ crossings in the diagram $D_{5,5n+1}$ to obtain the diagram $D_{5,5n+1}'$, as in Figure \ref{figure:dalt5}. The diagram $D_{5,5n+1}$ has $n+1$ $\sigma_4$-crossings. The strand between the $\sigma_3$ and $\sigma_4$ crossings has $2n+1$ crossings. Changing $n$ of those crossings results in an alternating diagram. Thus the dealternating number of $D_{5,5n+1}'$ is $n$ plus the number of $\sigma_2$ crossings, yielding $\dalt(D_{5,5n+1}')=4n+1$. Theorem \ref{theorem:HFKWidth} implies that $4n \leq \dalt(T_{5,5n+1}) \leq 4n+1$. The same strategy yields the result for the remaining torus knots on $5$ strands.
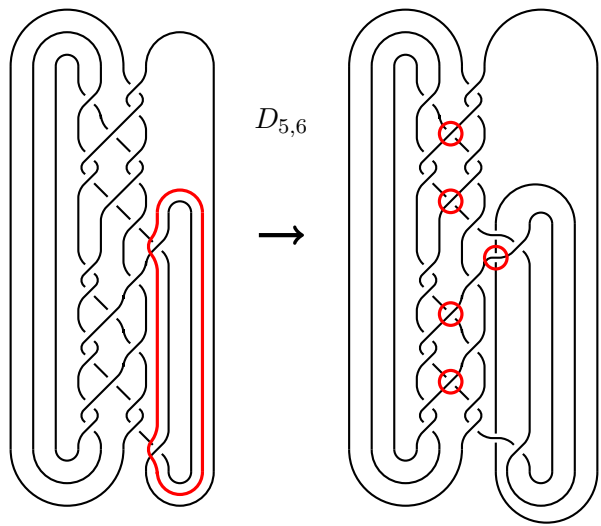
\begin{figure}[h]
 $$\begin{tikzpicture}[thick, scale = .3]
          
          \draw (-6,14) node{$D_{5,6}$};
          
          \draw[->, ultra thick] (-7,9) -- (-5,9);

\begin{scope}[rounded corners = 1mm]

\draw (0,-.5) -- (0,0) -- (1,1) -- (.7,1.3);
\draw (.3,1.7) -- (0,2) -- (0,3) -- (1,4) -- (.7,4.3);
\draw (.3,4.7) -- (0,5) -- (0,6) -- (1,7) -- (.7,7.3);
\draw (.3,7.7) -- (0,8) -- (0,11) -- (1,12) -- (.7,12.3);
\draw (.3,12.7) -- (0,13) -- (0,14.5) -- (1,15.5) -- (1,16.5) ;

\draw (1,-.5) -- (1,0) -- (.7,.3);
\draw (.3,.7) -- (0,1) -- (2,3) -- (2,3.5) -- (3,4.5) -- (3,6.5) -- (2.7,6.8);
\draw (2.3,7.2) -- (2,7.5) -- (2,9) -- (3,10) -- (3,11) -- (2.7,11.3);
\draw (2.3,11.7) -- (2,12) -- (3,13) -- (3,14) -- (2.7,14.3);
\draw (2.3,14.7) -- (2,15) -- (3,16) -- (3,16.5);

\draw (2,-.5) -- (2,0) -- (3,1) -- (2.7,1.3);
\draw (2.3,1.7) -- (1.7,2.3);
\draw (1.3,2.7) -- (.7,3.3);
\draw (.3,3.7) -- (0,4) -- (2,6) -- (2,6.5) -- (3,7.5) -- (3,8) -- (4,8) -- (5,9) -- (5,9.5);

\draw (2.7,.3) -- (3,0) -- (4,0) -- (4.3,-.3);
\draw (4.7,-.7) -- (5,-1) -- (5,-1.5);
\draw (2.3,.7) -- (2,1) -- (3,2) -- (3,3.5) -- (2.7,3.8);
\draw (2.3,4.2) -- (2,4.5) -- (2,5) -- (1.7,5.3);
\draw (1.3,5.7) -- (.7,6.3);
\draw (.3,6.7) -- (0,7) -- (1,8) -- (1,10) -- (3,12) -- (2.7,12.3);
\draw (2.3,12.7) -- (1.7,13.3);
\draw (1.3,13.7) -- (1,14) -- (1,14.5) -- (.7,14.8);
\draw (.3,15.2) -- (0,15.5) -- (0,16.5);

\draw (2,16.5) --  (2,16) -- (2.3,15.7);
\draw (2.7, 15.3) -- (3,15) -- (0,12) -- (.3,11.7);
\draw (.7,11.3) -- (1.3,10.7);
\draw (1.7,10.3) -- (2.3,9.7); 
\draw (2.7,9.3) -- (3,9) -- (4,9) -- (4.3,8.7);
\draw (4.7,8.3) -- (5,8) -- (5,0) -- (4,-1) -- (4,-1.5);

\draw (5,-1.5) arc (-180:0:.5cm);
\draw (5,9.5) arc (180:0:.5cm);
\draw (6,-1.5) -- (6,9.5);

\draw (4,-1.5) arc (-180:0:1.5cm);
\draw (7,-1.5) -- (7,9.5);
\draw (7,9.5) arc (0:180:1.75cm);

\draw (3.5,9.5) -- (3.5,9.2);
\draw (3.5, 8.8) -- (3.5,8.2);
\draw (3.5,7.8) -- (3.5,0.2);
\draw (3.5,-.2) -- (3.5,-1.5);

\draw (0,-.5) arc (0:-180:.5cm);
\draw (1,-.5) arc (0:-180:1.5cm);
\draw (2,-.5) arc (0:-180:2.5cm);

\draw (0,16.5) arc (0:180:.5cm);
\draw (1,16.5) arc (0:180:1.5cm);
\draw (2,16.5) arc (0:180:2.5cm);

\draw (-1,16.5) -- (-1,-.5);
\draw (-2,16.5) -- (-2,-.5);
\draw (-3,16.5) -- (-3,-.5);

\draw (3.5,-1.5) arc (-180:0:2.25cm);
\draw (8,-1.5) -- (8,16.5);

\draw (3,16.5) arc (180:0:2.5cm);

\end{scope}

\draw[red, very thick] (1.5,2.5) circle (.5cm);
\draw[red, very thick] (1.5,5.5) circle (.5cm);
\draw[red, very thick] (1.5,10.5) circle (.5cm);
\draw[red, very thick] (1.5,13.5) circle (.5cm);
\draw[red, very thick] (3.5, 8) circle (.5cm);

\begin{scope}[xshift = -15cm]
\begin{scope}[rounded corners = 1mm]

\draw (0,-.5) -- (0,0) -- (1,1) -- (.7,1.3);
\draw (.3,1.7) -- (0,2) -- (0,3) -- (1,4) -- (.7,4.3);
\draw (.3,4.7) -- (0,5) -- (0,6) -- (1,7) -- (.7,7.3);
\draw (.3,7.7) -- (0,8) -- (0,11) -- (1,12) -- (.7,12.3);
\draw (.3,12.7) -- (0,13) -- (0,14.5) -- (1,15.5) -- (1,16.5) ;

\draw (1,-.5) -- (1,0) -- (.7,.3);
\draw (.3,.7) -- (0,1) -- (2,3) -- (2,3.5) -- (3,4.5) -- (3,6.5) -- (2.7,6.8);
\draw (2.3,7.2) -- (2,7.5) -- (2,9) -- (3,10) -- (3,11) -- (2.7,11.3);
\draw (2.3,11.7) -- (2,12) -- (3,13) -- (3,14) -- (2.7,14.3);
\draw (2.3,14.7) -- (2,15) -- (3,16) -- (3,16.5);

\draw (2,-.5) -- (2,0) -- (3,1) -- (2.7,1.3);
\draw (2.3,1.7) -- (1.7,2.3);
\draw (1.3,2.7) -- (.7,3.3);
\draw (.3,3.7) -- (0,4) -- (2,6) -- (2,6.5) -- (3,7.5) -- (3,8) -- (4,9) -- (4,10);

\draw (2.7,.3) -- (3.3,-.3);
\draw (3.7,-.7) -- (4,-1) -- (4,-1.5);
\draw (2.3,.7) -- (2,1) -- (3,2) -- (3,3.5) -- (2.7,3.8);
\draw (2.3,4.2) -- (2,4.5) -- (2,5) -- (1.7,5.3);
\draw (1.3,5.7) -- (.7,6.3);
\draw (.3,6.7) -- (0,7) -- (1,8) -- (1,10) -- (3,12) -- (2.7,12.3);
\draw (2.3,12.7) -- (1.7,13.3);
\draw (1.3,13.7) -- (1,14) -- (1,14.5) -- (.7,14.8);
\draw (.3,15.2) -- (0,15.5) -- (0,16.5);

\draw (2,16.5) --  (2,16) -- (2.3,15.7);
\draw (2.7, 15.3) -- (3,15) -- (0,12) -- (.3,11.7);
\draw (.7,11.3) -- (1.3,10.7);
\draw (1.7,10.3) -- (2.3,9.7); 
\draw (2.7,9.3) -- (3.3,8.7);
\draw (3.7,8.3) -- (4,8) -- (4,0) -- (3,-1) -- (3,-1.5);

\draw (0,-.5) arc (0:-180:.5cm);
\draw (1,-.5) arc (0:-180:1.5cm);
\draw (2,-.5) arc (0:-180:2.5cm);

\draw (0,16.5) arc (0:180:.5cm);
\draw (1,16.5) arc (0:180:1.5cm);
\draw (2,16.5) arc (0:180:2.5cm);

\draw (-1,16.5) -- (-1,-.5);
\draw (-2,16.5) -- (-2,-.5);
\draw (-3,16.5) -- (-3,-.5);

\draw (4,-1.5) arc (-180:0:.5cm);
\draw (5,10) arc (0:180:.5cm);
\draw (5,-1.5) -- (5,10);

\draw (3,16.5) arc (180:0:1.5cm);
\draw (3,-1.5) arc (-180:0:1.5cm);
\draw (6,-1.5) -- (6,16.5);

\begin{scope}[red, very thick]
\draw (3.5,10) arc (180:0:1cm);
\draw (3.5,-1.5) arc (-180:0:1cm);
\draw (5.5,10) -- (5.5,-1.5);

\draw (3.5,10) -- (3.5,9.25) -- (3,8.5) -- (3.5,7.75) -- (3.5,0.25) -- (3,-.5) -- (3.5,-1.25) -- (3.5,-1.5);

\end{scope}

\end{scope}
\end{scope}

\end{tikzpicture}$$
\caption{Wrapping the strand around the $\sigma_4$ crossings decreases the dealternating number of the diagram by one.}
\label{figure:dalt5}
\end{figure}
\end{proof}

The methods of the above proof can be used to prove the following result. The details are omitted.
\begin{theorem}
For $n\geq 1$, we have
$$\begin{array}{>{\hfil$}p{5 cm}<{$\hfil} >{\hfil$}p{5 cm}<{$\hfil}}
\dalt(T_{4,4n}) \leq  2n+1,&
\dalt(T_{4,4n+2}) \leq 2n+2,\\
\dalt(T_{5,5n}) \leq  4n+1,&
\dalt(T_{6,6n})\leq 6n+2.\\
\end{array}$$
\end{theorem}

As Theorem \ref{theorem:TuraevExact} implies, we were able to find Turaev genus minimizing diagrams for many but not all of the torus knots on six or fewer strands. In the cases where we did not compute the Turaev genus exactly, there are several possibilities. The Turaev genus of these knots could be strictly greater than the lower bound given by knot Floer homology. Alternatively, since our search of diagrams only considered the closures of positive braids, it is possible that the Turaev genus minimizing diagrams of our unsolved cases include negative generators $\sigma_i^{-1}$ or possibly are not even closures of braids. It remains an interesting question to compute the Turaev genus or dealternating numbers of an arbitrary torus knot $T_{p,q}$.

\bibliography{BibTuraevTorus}{}

\newcommand{\etalchar}[1]{$^{#1}$}
\providecommand{\bysame}{\leavevmode\hbox to3em{\hrulefill}\thinspace}
\providecommand{\MR}{\relax\ifhmode\unskip\space\fi MR }
\providecommand{\MRhref}[2]{%
  \href{http://www.ams.org/mathscinet-getitem?mr=#1}{#2}
}
\providecommand{\href}[2]{#2}
\begin{thebibliography}{DFK{\etalchar{+}}08}

\bibitem[ABB{\etalchar{+}}92]{Adams:Almost}
Colin~C. Adams, Jeffrey~F. Brock, John Bugbee, Timothy~D. Comar, Keith~A.
  Faigin, Amy~M. Huston, Anne~M. Joseph, and David Pesikoff, \emph{Almost
  alternating links}, Topology Appl. \textbf{46} (1992), no.~2, 151--165.

\bibitem[Abe09a]{Abe:Alternation}
Tetsuya Abe, \emph{An estimation of the alternation number of a torus knot}, J.
  Knot Theory Ramifications \textbf{18} (2009), no.~3, 363--379.

\bibitem[Abe09b]{Abe:Adequate}
\bysame, \emph{The {T}uraev genus of an adequate knot}, Topology Appl.
  \textbf{156} (2009), no.~17, 2704--2712.

\bibitem[AK10]{AbeKishimoto:3-braid}
Tetsuya Abe and Kengo Kishimoto, \emph{The dealternating number and the
  alternation number of a closed 3-braid}, J. Knot Theory Ramifications
  \textbf{19} (2010), no.~9, 1157--1181.

\bibitem[AL15]{ArmLow:Turaev}
Cody~W. Armond and Adam~M. Lowrance, \emph{Turaev genus and alternating
  decompositions}, arXiv:1507.02771. To appear in Algebr. Geom. Topol., 2015.

\bibitem[BFLZ16]{BFLZ:KhWidth}
Sebastian Baader, Peter Feller, Lukas Lewark, and Raphael Zentner,
  \emph{Khovanov width and dealternation number of positive braid links},
  arXiv:1610.04534, 2016.

\bibitem[CK09]{CK:SpanningTree}
Abhijit Champanerkar and Ilya Kofman, \emph{Spanning trees and {K}hovanov
  homology}, Proc. Amer. Math. Soc. \textbf{137} (2009), no.~6, 2157--2167.

\bibitem[CK14]{CK:Survey}
\bysame, \emph{A survey on the {T}uraev genus of knots}, Acta Math. Vietnam.
  \textbf{39} (2014), no.~4, 497--514.

\bibitem[CKS07]{CKS:Khovanov}
Abhijit Champanerkar, Ilya Kofman, and Neal Stoltzfus, \emph{Graphs on surfaces
  and {K}hovanov homology}, Algebr. Geom. Topol. \textbf{7} (2007), 1531--1540.

\bibitem[DFK{\etalchar{+}}08]{DFKLS:Jones}
Oliver~T. Dasbach, David Futer, Efstratia Kalfagianni, Xiao-Song Lin, and
  Neal~W. Stoltzfus, \emph{The {J}ones polynomial and graphs on surfaces}, J.
  Combin. Theory Ser. B \textbf{98} (2008), no.~2, 384--399.

\bibitem[DL11]{DasLow:Concordance}
Oliver~T. Dasbach and Adam~M. Lowrance, \emph{Turaev genus, knot signature, and
  the knot homology concordance invariants}, Proc. Amer. Math. Soc.
  \textbf{139} (2011), no.~7, 2631--2645.

\bibitem[DL14]{DasLow:TuraevKh}
\bysame, \emph{A {T}uraev surface approach to {K}hovanov homology}, Quantum
  Topol. \textbf{5} (2014), no.~4, 425--486.

\bibitem[DL16]{DasLow:TuraevJones}
\bysame, \emph{Invariants of {T}uraev genus one links}, arXiv:1604.03501. To
  appear in Comm. Anal. Geom., 2016.

\bibitem[FPZ15]{FPZ:Alternation}
Peter Feller, Simon Pohlmann, and Raphael Zentner, \emph{Alternating numbers of
  torus knots with small braid index}, arXiv:1508.05825. To appear in Indiana
  Univ. Math. J., 2015.

\bibitem[Kau87]{Kauffman:StateModels}
Louis~H. Kauffman, \emph{State models and the {J}ones polynomial}, Topology
  \textbf{26} (1987), no.~3, 395--407.

\bibitem[Kaw10]{Kawauchi:Alternation}
Akio Kawauchi, \emph{On alternation numbers of links}, Topology Appl.
  \textbf{157} (2010), no.~1, 274--279.

\bibitem[Kim15]{Kim:TuraevClassification}
Seungwon Kim, \emph{Link diagrams with low {T}uraev genus}, arXiv:1507.02918,
  2015.

\bibitem[KL07]{KimLee:Pretzel}
Dongseok Kim and Jaeun Lee, \emph{Some invariants of pretzel links}, Bull.
  Austral. Math. Soc. \textbf{75} (2007), no.~2, 253--271.

\bibitem[Low08]{Lowrance:HFK}
Adam~M. Lowrance, \emph{On knot {F}loer width and {T}uraev genus}, Algebr.
  Geom. Topol. \textbf{8} (2008), no.~2, 1141--1162.

\bibitem[Low11]{Lowrance:Twisted}
\bysame, \emph{The {K}hovanov width of twisted links and closed 3-braids},
  Comment. Math. Helv. \textbf{86} (2011), no.~3, 675--706.

\bibitem[Low15]{Lowrance:AltDist}
\bysame, \emph{Alternating distances of knots and links}, Topology Appl.
  \textbf{182} (2015), 53--70.

\bibitem[Mos71]{Moser:Surgery}
Louise Moser, \emph{Elementary surgery along a torus knot}, Pacific J. Math.
  \textbf{38} (1971), 737--745.

\bibitem[Mur87]{Murasugi:Jones}
Kunio Murasugi, \emph{Jones polynomials and classical conjectures in knot
  theory}, Topology \textbf{26} (1987), no.~2, 187--194.

\bibitem[OS03]{OS:SpanningTree}
Peter Ozsv{\'a}th and Zolt{\'a}n Szab{\'o}, \emph{Heegaard {F}loer homology and
  alternating knots}, Geom. Topol. \textbf{7} (2003), 225--254 (electronic).

\bibitem[OS04]{OS:HFK}
\bysame, \emph{Holomorphic disks and knot invariants}, Adv. Math. \textbf{186}
  (2004), no.~1, 58--116.

\bibitem[OS05]{OS:Lens}
\bysame, \emph{On knot {F}loer homology and lens space surgeries}, Topology
  \textbf{44} (2005), no.~6, 1281--1300.

\bibitem[Ras03]{Rasmussen:HFK}
Jacob Rasmussen, \emph{Floer homology and knot complements}, Ph.D. thesis,
  Harvard University, 2003.

\bibitem[Thi88]{Thistlethwaite:Jones}
Morwen~B. Thistlethwaite, \emph{An upper bound for the breadth of the {J}ones
  polynomial}, Math. Proc. Cambridge Philos. Soc. \textbf{103} (1988), no.~3,
  451--456.

\bibitem[Tur87]{Turaev:Jones}
V.~G. Turaev, \emph{A simple proof of the {M}urasugi and {K}auffman theorems on
  alternating links}, Enseign. Math. (2) \textbf{33} (1987), no.~3-4, 203--225.

\end{thebibliography}
\bibliographystyle {amsalpha}

\end{document}